\documentclass[
a4paper 
,11pt 
,leqno 
,english
]{scrartcl} 

\usepackage{my-preamble-article} 
\usepackage{my-macros} 
\usepackage{my-specific-macros} 
\usepackage{my-counters-related-preprints} 
\title{Global behaviour of radially symmetric solutions stable at infinity for gradient systems}
\author{Emmanuel \textsc{Risler}}
\begin{document}
\maketitle
\begin{abstract}
This paper is concerned with radially symmetric solutions of parabolic gradient systems of the form
\[
u_t = -\nabla V(u) + \Delta_x u
\]
where the space variable $x$ and the state variable $u$ are multidimensional, and the potential $V$ is coercive at infinity. 
For such systems, under generic assumptions on the potential, the asymptotic behaviour of solutions \emph{stable at infinity}, that is approaching a stable homogeneous equilibrium as $\abs{x}$ goes to $+\infty$, is investigated. It is proved that every such solution approaches a pattern made of a stacked family of radially symmetric bistable fronts travelling to infinity, and around the origin a (possibly non-homogeneous) radially symmetric stationary solution. This behaviour is similar to that of bistable solutions for gradient systems in one unbounded spatial dimension, which is described in companion papers. 
\end{abstract}
\nnfootnote{%
}
\thispagestyle{empty} 
\pagestyle{empty}
\hypersetup{pageanchor=false} 
\newpage
\tableofcontents
\newpage
\hypersetup{pageanchor=true} 
\pagestyle{plain}
\setcounter{page}{1}
\section{Introduction}
This paper deals with the global dynamics of radially symmetric solutions of nonlinear parabolic systems of the form
\begin{equation}
\label{syst_higher_dim}
\tilde{u}_t = -\nabla V(\tilde{u}) + \Delta_x \tilde{u}
\end{equation}
where the time variable $t$ is real, the space variable $x$ lies in the spatial domain $\rr^d$ with $d$ a positive integer, the function $(x,t)\mapsto \tilde{u}(x,t)$ takes its values in $\rr^{\dState}$ with $\dState$ a positive integer, and the nonlinearity is the gradient of a scalar \emph{potential} function $V:\rr^{\dState}\to\rr$, which is assumed to be regular (of class $\ccc^2$) and coercive at infinity (see hypothesis \cref{hyp_coerc} in \vref{subsec:coerc_glob_exist}).
\begin{notation}
As the previous sentence shows, the \emph{state} dimension is thus denoted by $\dState$ whereas the \emph{space} dimension is simply denoted by $d$. The reason for this choice (and for the absence of subscript in the notation for the space dimension) is that, by contrast with the state dimension, the space dimension is ubiquitous in the computations throughout the paper. 
\end{notation}
\emph{Radially symmetric} solutions of system \cref{syst_higher_dim} are functions of the form
\[
\tilde{u}(x,t) = u(r,t)
\,,
\quad\text{where } r=\abs{x} \text{ (the euclidean norm of } x \text{ in } \rr^d\text{),}
\]
and $(r,t)\mapsto u(r,t)$ is defined on $[0,+\infty)\times[0,+\infty)$ with values in $\rr^{\dState}$.
For such functions, system \cref{syst_higher_dim} takes the following form:
\begin{equation}
\label{syst_rad_sym}
u_t = -\nabla V(u) + \frac{d-1}{r} u_r + u_{rr}
\quad\text{with the boundary condition}\quad
\partial_r u(0,t) = 0
\,,
\end{equation}
and it this last system \cref{syst_rad_sym} that will be considered in this paper. 

A fundamental feature of each of systems~\cref{syst_higher_dim,syst_rad_sym} is that they can be recast, at least formally, as gradient flows of energy functionals. If $(w,w')$ is a pair of vectors of $\rr^{\dState}$, let $w\cdot w'$ and $\abs{w} =\sqrt{w\cdot w}$ denote the usual Euclidean scalar product and the usual Euclidean norm, respectively, and let us simply write $w^2$ for $\abs{w}^2$. 

For every function $x\mapsto \tilde{v}(x)$ defined on $\rr^d$ with values in $\rr^{\dState}$, its \emph{energy} (or \emph{Lagrangian} or \emph{action}) with respect to system \cref{syst_higher_dim} is defined (at least formally) by 
\[
\tilde{\eee}[\tilde{v}] = \int_{\rr^d}  \biggl(\frac{1}{2}\abs{\nabla_x \tilde{v}(x)}^2+V\bigl(\tilde{v}(x)\bigr)\biggr)\, dx
\]
where
\[
\abs{\nabla_x \tilde{v}(x)}^2 = \sum_{i=1}^d \sum_{j=1}^{\dState} \bigl(\partial_{x_i}\tilde{v}_j(x)\bigr)^2
\,.
\]
Similarly, for every function $r\mapsto v(r)$ defined on $[0,+\infty)$ with values in $\rr^{\dState}$, its \emph{energy} (or \emph{Lagrangian} or \emph{action}) with respect to system \cref{syst_rad_sym} is defined (at least formally) by 
\begin{equation}
\label{form_en}
\eee[v] = \int_0^{+\infty} r^{d-1} \Bigl(\frac{1}{2}v_r(r)^2+V\bigl(v(r)\bigr)\Bigr)\, dr
\,.
\end{equation}
Note that if $\tilde{v}:\rr^d\to\rr$ denotes a (radially symmetric) function defined as $\tilde{v}(x) = v(\abs{x})$, and if $S_{d-1}$ denotes the surface area of the $d-1$-unit-sphere in $\rr^d$, then 
\begin{equation}
\label{factor_D_d_minus_one_in_definition_of_energy}
\tilde{\eee}[\tilde{v}] = S_{d-1} \, \eee[v]
\,.
\end{equation}
Formally, the differential of the functional $\eee$ defined by \cref{form_en} reads (skipping border terms in the integration by parts):
\[
\begin{aligned}
d\eee[v] \cdot \delta v &= \int_0^{+\infty}  r^{d-1}\bigl( v_r \cdot (\delta v)_r + \nabla V(v) \cdot \delta v \bigr) \, dr \\
&= \int_0^{+\infty}  r^{d-1}\Bigl( - v_{rr} - \frac{d-1}{r} v_r + \nabla V(v)  \Bigr) \cdot \delta v \, dr
\,.
\end{aligned}
\]
In other words, the (formal) gradient of this functional with respect to the $L^2$-scalar product with weight $r^{d-1}$ on functions $[0,+\infty)\to\rr^{\dState}$ reads: 
\[
\nabla_\text{rad}\eee[v] = - v_{rr} - \frac{d-1}{r} v_r + \nabla V(v)
\,,
\]
thus system \cref{syst_rad_sym} can formally be rewritten under the form: 
\[
u_t = - \nabla_\text{rad}\eee[u(\cdot,t)]
\,,
\]
and if $(r,t)\mapsto u(r,t)$ is a solution of this system, then (formally)
\[
\frac{d}{d t}\eee[u(\cdot,t)] = -\int_0^{+\infty} r^{d-1} u_t(r,t)^2\, dr \le 0
\,.
\]

An additional and related feature of system \cref{syst_higher_dim} is that a formal gradient structure exists not only in the laboratory frame, but also in every frame travelling at a constant velocity, \cite{Risler_noInvasionCaseHigherSpace_2020}. What about radially symmetric solutions of system \cref{syst_rad_sym} with respect to the radial coordinate $r$? Let us see this now. 

For every nonnegative quantity $c$, if $(r,t)\mapsto u(r,t)$ and $(\rho,t)\mapsto v(\rho,t)$ are two functions related by 
\[
u(r,t) = v(\rho,t)
\quad\text{for}\quad
r = ct + \rho
\,,
\]
then $u$ is a solution of \cref{syst_rad_sym} if and only if $v$ is a solution of
\begin{equation}
\label{syst_tf}
v_t-c \cdot v_\rho = -\nabla V(v)+\frac{d-1}{ct+\rho}v_\rho + v_{\rho\rho}
\,.
\end{equation}
Now, for every function $w:\rho\mapsto w(\rho)$ defined on $[-ct,+\infty)$ with values in $\rr^{\dState}$, its energy functional with respect to system \cref{syst_tf} may be defined, at least formally, as
\begin{equation}
\label{form_en_tf}
\eee_{c,t}[w]=\int_{-ct}^{+\infty} (ct+\rho)^{d-1}e^{c\rho} \Bigl( \frac{1}{2}w_\rho(\rho)^2 + V\bigl(w(\rho)\bigr) \Bigr)\, d\rho
\,.
\end{equation}
Formally, the differential of this functional reads (skipping border terms in the integration by parts)
\[
\begin{aligned}
d\eee_{c,t}[w] \cdot \delta w &= \int_{-ct}^{+\infty} (ct+\rho)^{d-1}e^{c\rho} \bigl( w_\rho\cdot\delta w_\rho + \nabla V(w)\cdot\delta w \bigr)\, d\rho \\
&= \int_{-ct}^{+\infty} (ct+\rho)^{d-1}e^{c\rho} \Bigl( - w_{\rho\rho} - c w_\rho - \frac{d-1}{ct+\rho} w_\rho + \nabla V(w) \Bigr) \cdot \delta w \, d\rho 
\,.
\end{aligned}
\]
In other words, the (formal) gradient of this functional with respect to the $L^2$-scalar product with weight $(ct+\rho)^{d-1}e^{c\rho}$ on functions $[-ct,+\infty)\to\rr^{\dState}$ reads:
\[
\nabla_{\text{rad}, c} \eee_{c,t}[w] = - w_{\rho\rho} - c w_\rho - \frac{d-1}{ct+\rho} w_\rho + \nabla V(w)
\,,
\]
and system \cref{syst_tf} can formally be rewritten under the form:
\begin{equation}
\label{form_grad_tf}
v_t = - \nabla_{\text{rad}, c} \eee_{c,t}[v(\cdot,t)]
\,.
\end{equation}
Now, if $(\rho,t)\mapsto v(\rho,t)$ is a solution of system \cref{syst_tf}, then (formally):
\begin{equation}
\label{dt_form_en_tf}
\frac{d}{d t}\eee_{c,t}[v(\cdot,t)] = \int_{-ct}^{+\infty} (ct+\rho)^{d-1}e^{c\rho} \Biggl[ -v_t(\rho,t)^2 + \frac{c(d-1)}{ct+\rho}\Bigl(\frac{1}{2}v_\rho^2 + V(v) \Bigr)\Biggr]\, d\rho 
\,.
\end{equation}
What can be seen from these calculations is that, although system \cref{syst_tf} is still formally gradient, the time derivative of the energy \cref{form_en_tf} involves, in addition to the standard dissipation term $-v_t^2$, an additional term, that can be related to the time dependence of the $L^2$-scalar product defining the gradient $\nabla_{\text{rad}, c}$, or viewed as induced by the curvature $1/r=1/ct+\rho$. Only in the limit of large radii (or large positive times, or small curvature) is the expression of the time derivative of energy always nonnegative. In short, the picture is not hopeless, but not as nice as it would be in space dimension $1$.

This gradient structure (``asymptotic'' gradient structure in the case of system \cref{syst_tf}) has been known for a long time \cite{FifeMcLeod_approachTravFront_1977}, but it is only more recently that it received a more detailed attention from several authors (among which S. Heinze, C. B. Muratov, Th. Gallay, R. Joly, and the author \cite{Heinze_variationalApproachTW_2001,Muratov_globVarStructPropagation_2004,GallayRisler_globStabBistableTW_2007,Risler_globCVTravFronts_2008,GallayJoly_globStabDampedWaveBistable_2009}), and that is was shown that this structure is sufficient (in itself, that is without the use of the maximum principle) to prove results of global convergence towards travelling fronts. These ideas have been applied since in different contexts, to prove either global convergence or just existence results, see for instance \cite{Chapuisat_existenceCurvedFront_2007,ChapuisatJoly_asymptProfilesTravFrontBiolEqu_2010,MuratovNovaga_frontPropIVariational_2008,MuratovNovaga_frontPropIISharpReaction_2008,MuratovNovaga_globExpConvTW_2012,AlikakosKatzourakis_heteroclinicTW_2011,AlikakosFusco_ellipticSystemsPhaseTransType_2018,Luo_globStabDampedWaveEqu_2013,BouhoursNadin_variationalApproachRDForcedSpeedDim1_2015,BouhoursGiletti_extinctSpreadClimateAllee_2016,BouhoursGiletti_spreadVanishMonStabRDEqu_2018,OliverBonafoux_heteroclinicTW1dParabSystDegenerate_2021,OliverBonafoux_TWparabAllenCahn_2021,ChenChienHuang_varApproach3PhaseTWgradSyst_2021,ChenCotiZelati_TWSolAllenCahnEqu_2022,OliverBonafouxRisler_globCVPushedTravFronts_2023}. 

Even more recently, the same ideas enabled the author (\cite{Risler_globalRelaxation_2016,Risler_globalBehaviour_2016}) to push one step further (that is, extend to systems) the program initiated by P. C. Fife and J. McLeod in the late seventies with the aim of describing the global asymptotic behaviour (when space is one-dimensional) of every \emph{bistable} solution, that is every solution close to stable homogeneous equilibria at both ends of space (\cite{FifeMcLeod_approachTravFront_1977,FifeMcLeod_phasePlaneDisc_1981,Fife_longTimeBistable_1979}). Under generic assumptions on the potential $V$, these solutions approach a stacked (possibly empty) family of bistable travelling fronts at both ends of space, and approach in between a pattern of stationary solutions going slowly away from one another. These stacked families will be called \emph{terraces} (see \cref{subsubsec:def_prop_terrace} for comments and references on this terminology and a precise definition in the framework of this paper). 

The aim of this paper is to extend to the case of radially symmetric solutions in higher space dimensions the results (description of the global asymptotic behaviour) obtained in \cite{Risler_globalRelaxation_2016,Risler_globalBehaviour_2016} for bistable solutions when spatial domain is one-dimensional. Thus, the solutions that will be considered are solutions of system \cref{syst_rad_sym} that approach a stable homogeneous equilibrium as the radius $r$ goes to $+\infty$ (or equivalently radially symmetric solutions of system \cref{syst_higher_dim} that are stable at infinity, in space). The goal is to prove that, under generic assumptions on the potential, every such solution approaches a pattern made of a stacked family of (radially symmetric) bistable front going to infinity (a ``propagating terrace''), and around the origin a radially symmetric stationary solution (which may or not be spatially homogeneous).

In the scalar case $\dState$ equals $1$, the behaviour of solutions stable at infinity of reaction-diffusion equations in higher space dimension is the subject of a large amount of literature. For extinction/invasion (threshold) results in relation with the initial condition and the reaction term see for instance \cite{AronsonWeinberger_multidimNonlinDiffPopGen_1978,DuMatano_cvSharpThresholdNonlinDiff_2010,MuratovZhong_thresholdSymSol_2013,MuratovZhong_thresholdPhenSymDecrRadialSolRDEquations_2017,Zlatos_sharpTransitionExtinctPropag_2005}, for local convergence and quasi-convergence results see for instance \cite{DuMatano_cvSharpThresholdNonlinDiff_2010,DuPolacik_locUnifCVNLparabEquRN_2015,MatanoPolacik_dynNonnegSolOneDimRDI_2016,MatanoPolacik_dynNonnegSolOneDimRDII_2020,Polacik_cvQuasiCvSolParabEquRoverview_2017,HamelRossi_spreadSpeedsOneDimSymRDEqu_2021}, and for further estimates on the location and shape at large positive times of the level sets see for instance \cite{Jones_sphericallySymmetricSolutionsRDEquation_1983,Jones_asymptBehavRDEquHigherSpaceDim_1983,Uchiyama_asymptBehavSolRDEquaVaryingDrift_1985,RoussierMichon_stabilityRadiallySymmTWRDEqu_2004,RoquejoffreRoussier_sharpLargeTimeBehavNdimBistable_2021,HamelRossi_spreadSpeedsOneDimSymRDEqu_2021}. Recently, a result of global convergence towards a radial terrace of travelling fronts was proved by Y. Du and H. Matano \cite{DuMatano_radialTerraceSolutionsPropProfileRN_2020} (without any radial symmetry assumption on the solutions), and a rather complementary result of convergence/quasi-convergence (in $L^\infty_\text{loc}(\rr^{\dState},\rr)$) was proved by P. Poláčik \cite{Polacik_bdedRadialSolParabEquQuasiconvStableInfinity_2023} under very weak non-degeneracy assumptions on the nonlinearity; see also \cite{MatanoMoriNara_asymptBehavFrontsAnisotropic_2019,MatsuzawaNara_asymptBehavFrontsAnisotropic_2022} for results similar to those of \cite{DuMatano_radialTerraceSolutionsPropProfileRN_2020} when space is anisotropic. The present paper extends some of those results (in particular some of the results of \cite{DuMatano_radialTerraceSolutionsPropProfileRN_2020,Polacik_bdedRadialSolParabEquQuasiconvStableInfinity_2023}) to the more general setting of systems, but for radially symmetric solutions only. 

The path of the proof is very similar to the one used in the spatial dimension one case \cite{Risler_globalRelaxation_2016,Risler_globalBehaviour_2016}. It is based on a careful study of the relaxation properties of energy or $L^2$ functionals (localized in space by adequate weight functions), both in the laboratory frame and in frames travelling at various speeds. The differences are mainly of technical nature due to specific features of the (reduced) system \cref{syst_rad_sym}:
\begin{itemize}
\item the ``curvature'' term $(d-1)u_r/r$; 
\item the fact that space is reduced to the half-line $[0,+\infty)$ (thus is in this sense less ``spatially homogeneous'' than the full real line); 
\item the ``no invasion implies relaxation'' part of the argument, which does not call upon the radial symmetry and is therefore processed in the companion paper \cite{Risler_noInvasionCaseHigherSpace_2020};
\item the convergence behind the terrace of travelling front, which differs from the space dimension one case both regarding the arguments and the result (roughly speaking, again due to the curvature term).
\end{itemize}
\section{Assumptions, notation, and statement of the results}
This section presents strong similarities with \cite[\GlobalBehaviourSecAssumptionsNotationStatement]{Risler_globalBehaviour_2016} and \cite[\GlobalRelaxationSecAssumptionsNotationStatement]{Risler_globalRelaxation_2016}, where more details and comments can be found.

For the remaining of the paper it will be assumed than the space dimension $d$ is not smaller than $2$. Indeed the case $d=1$ was already treated in \cite{Risler_globalRelaxation_2016,Risler_globalBehaviour_2016}, and several definitions, estimates, and statements will turn out to be irrelevant without this assumption (see for instance the definition of the weight function $T_\rho\psi_0$ in \vref{subsubsec:def_fire_zero}). 
\subsection{Semi-flow and coercivity hypothesis}
\label{subsec:coerc_glob_exist}
Let us consider the following two Banach spaces of continuous and uniformly bounded functions equipped with the uniform norm: 
\[
\begin{aligned}
X &= \bigl(\CbRd,\norm{\dots}_{L^\infty(\rr^d,\rr^{\dState})}\bigr) \,, \\
\text{and}\quad
Y &= \Bigl(\cccb{0},\norm{\dots}_{L^\infty\bigl([0,+\infty),\rr^{\dState}\bigr)}\Bigr)
\,.
\end{aligned}
\]
System \cref{syst_higher_dim} defines a local semi-flow in $X$ (see for instance D. B. Henry's book \cite{Henry_geomSemilinParab_1981}). 

As in \cite{Risler_globalRelaxation_2016,Risler_globalBehaviour_2016}, let us assume that the potential function $V:\rr^{\dState}\to\rr$ is of class $\ccc^2$ and that this potential function is strictly coercive at infinity in the following sense: 
\begin{gather}
\tag{$\text{H}_\text{coerc}$}
\lim_{R\to+\infty}\quad  \inf_{\abs{u}\ge R}\ \frac{u\cdot \nabla V(u)}{\abs{u}^2} >0
\label{hyp_coerc}
\end{gather}
(or in other words there exists a positive quantity $\varepsilon$ such that the quantity $u\cdot \nabla V(u)$ is greater than or equal to $\varepsilon\abs{u}^2$ as soon as $\abs{u}$ is large enough). 

According to this hypothesis \cref{hyp_coerc}, the semi-flow of system \cref{syst_higher_dim} on $X$ is actually global (see \vref{prop:attr_ball}). As a consequence, considering the restriction of this semi-flow to radially symmetric functions, it follows that system \cref{syst_rad_sym} defines a global semi-flow on $Y$. Let us denote by $(S_t)_{t\ge0}$ this last semi-flow on $Y$. 

In the following, a \emph{solution of system \cref{syst_rad_sym}} will refer to a function 
\[
[0,+\infty)\times[0,+\infty)\to\rr^{\dState}\,, \quad (r,t)\mapsto u(r,t)
\,,
\]
such that the function $u_0:r\mapsto u(r,t=0)$ (initial condition) is in $Y$ and $u(\cdot,t)$ equals $(S_t u_0)(\cdot)$ for every nonnegative time $t$. 
\subsection{Minimum points and solutions stable at infinity}
\subsubsection{Minimum points}
Everywhere in this paper, the term ``minimum point'' denotes a point where a function --- namely the potential $V$ --- reaches a local \emph{or} global minimum. Let $\mmm$ denote the set of \emph{nondegenerate} minimum points:
\[
\mmm=\{u\in\rr^{\dState}: \nabla V(u)=0 
\quad\text{and}\quad 
D^2V(u)\text{ is positive definite}\}
\,.
\]
\subsubsection{Solutions stable at infinity}
\begin{definition}[solution stable at infinity]
A solution $(r,t)\mapsto u(r,t)$ of system \cref{syst_rad_sym} is said to be \emph{stable at infinity} if there exists a point $m$ in $\mmm$ such that
\[
\limsup_{r\to+\infty} \abs{(u(r,t)-m}\to 0
\quad\text{as}\quad
t\to+\infty
\,.
\]
More precisely, such a solution is said to be \emph{stable close to $m$ at infinity}. A function (initial condition) $u_0$ in $Y$ is said to be \emph{stable (close to $m$) at infinity} if the solution of system \cref{syst_rad_sym} corresponding to this initial condition is stable (close to $m$) at infinity. 
\end{definition}
\begin{notation}
Let 
\[
\YstabInfty(m)
\]
denote the subset of $Y$ made of initial conditions that are stable close to $m$ at infinity. 
\end{notation}
\subsection{Stationary solutions, travelling fronts, terraces, and asymptotic pattern}
\subsubsection{Radially symmetric stationary solutions}
A function
\[
\phi:[0,+\infty)\to\rr^{\dState},
\quad r\mapsto\phi(r)
\]
is a stationary solution of system \cref{syst_rad_sym} if $\phi$ is a solution of the differential system
\begin{equation}
\label{syst_rad_sym_stationary}
\phi'' + \frac{d-1}{r}\phi' = \nabla V(\phi)
\,,
\quad\text{with the boundary condition}\quad
\phi'(0) = 0
\,.
\end{equation}
\begin{notation}
If $m$ is in $\mmm$, let $\PhiZeroCentre(m)$ denote the set of solutions of system \cref{syst_rad_sym_stationary} approaching $m$ at infinity. With symbols, 
\begin{equation}
\label{definition_PhiZeroCentre_of_m}
\begin{aligned}
\PhiZeroCentre(m) = \bigl\{ &
\phi:[0,+\infty)\to\rr^{\dState} : \phi \text{ is a solution of system \cref{syst_rad_sym_stationary}}
\\
& \text{and}\quad\phi(r)\xrightarrow[r\to +\infty]{} m \bigr\}
\,.
\end{aligned}
\end{equation}
In this notation, 
\begin{itemize}
\item the index ``$0$'' refers to the ``zero speed'' of these solutions, by contrast with the nonzero speed of the travelling fronts considered below, 
\item the symbols $\Phi$ and $\phi$ have been chosen for homogeneity with the notation introduced below for travelling fronts, 
\item and the index ``centre'' refers to the fact that these solutions are stationary for system \cref{syst_rad_sym} (with the curvature term) and not for system \cref{syst_rad_sym_large_radius} introduced below. 
\end{itemize}
This set $\PhiZeroCentre(m)$ comprises the constant solution $\phi\equiv m$, by contrast with the sets introduced in the next two \cref{subsubsec:trav_fronts,subsubsec:breakup}. 
\end{notation}
A function $\phi$ belonging to $\PhiZeroCentre(m)$ for some $m$ in $\mmm$ is said to be \emph{stable at infinity}.
\begin{definition}[energy of a stationary solution stable at infinity]
If $m$ is a point in $\mmm$ and $\phi$ is a function in $\PhiZeroCentre(m)$, let us call \emph{energy of $\phi$}, and let us denote by $\eee[\phi]$, the quantity 
\[
\eee[\phi] = \int_0^{+\infty} r^{d-1}\Bigl(\frac{1}{2}\phi'(r)^2 + V\bigl(\phi(r)\bigr)-V(m)\Bigr)\, dr
\,.
\]
Since $\phi(r)$ goes to $m$ at an exponential rate as $r$ goes to $+\infty$, this integral converges. 
\end{definition}
It follows from Pokhozhaev's identity (\cite{Pokhozhaev_onEigenfunctionsDeltauPlusLambdaf_1965,BerestyckiLions_existenceGroundState_1983}) that
\begin{equation}
\label{Pokhozhaev}
\eee[\phi] = \frac{1}{d}\int_0^{+\infty} r^{d-1}\phi'(r)^2\, dr
\,,
\end{equation}
and this shows that $\eee[\Phi]$ is nonnegative (and even positive if $\phi$ is not identically equal to $m$). It turns out that the results of the companion paper \cite{Risler_noInvasionCaseHigherSpace_2020} provide another justification of the nonnegativity of this energy (see conclusion \cref{item:prop_no_invasion_implies_relaxation_asympt_energy} of \cref{prop:no_invasion_implies_relaxation}).
\begin{remark}
Let us denote by $S_{d-1}$ the surface area of the $d-1$-unit sphere in $\rr^d$, and let us introduce the function $\tilde\phi:\rr^d\to\rr$ defined as $\tilde{\phi}(x)=\phi(\abs{x})$; then, 
\[
S_{d-1} \,\eee[\phi] = \int_{\rr^d}\Bigl(\frac{1}{2}\abs{\nabla\tilde{\phi}(x)}^2 + V\bigl(\tilde{\phi}(x)\bigr)-V(m)\Bigr)\, dx
\]
(compare with equality \cref{factor_D_d_minus_one_in_definition_of_energy} in introduction). 
\end{remark}
\subsubsection{Large radius asymptotic form of the system governing radially symmetric solutions}
When the radius $r$ goes to $+\infty$, system \cref{syst_rad_sym} governing radially symmetric solutions takes the following asymptotic form:
\begin{equation}
\label{syst_rad_sym_large_radius}
u_t = -\nabla V(u) + u_{rr}
\,,
\end{equation}
on functions $(r,t)\mapsto u(r,t)$ defined on $\rr\times[0,+\infty)$ (here the radius $r$ is defined on the whole real line), with values in $\rr^{\dState}$. 
\subsubsection{Radially symmetric travelling fronts for the large radius limit}
\label{subsubsec:trav_fronts}
Let $c$ be a positive quantity. A function
\[
\phi:\rr\to\rr^{\dState},
\quad \rho\mapsto\phi(\rho)
\]
is the profile of a wave travelling at the speed $c$ for system \cref{syst_rad_sym_large_radius} if the function  $(r,t)\mapsto \phi(r-ct)$ is a solution of this system, that is if $\phi$ is a solution of the differential system
\begin{equation}
\label{syst_trav_front}
\phi''=-c\phi'+\nabla V(\phi) 
\,.
\end{equation}
\begin{notation}
If $m_-$ and $m_+$ are two points of $\mmm$ and $c$ is a positive quantity, let $\Phi_c(m_-,m_+)$ denote the set of \emph{nonconstant} global solutions of system \cref{syst_trav_front} connecting $m_-$ to $m_+$. With symbols, 
\[
\begin{aligned}
\Phi_c(m_-,m_+) = \bigl\{ &
\phi:\rr\to\rr^{\dState} : \phi \text{ is a \emph{nonconstant} global solution of system \cref{syst_trav_front}}
\\
& \text{and}\quad\phi(\rho)\xrightarrow[\rho\to -\infty]{} m_-
\quad\text{and}\quad
\phi(\rho)\xrightarrow[\rho\to +\infty]{} m_+
\bigr\}
\,.
\end{aligned}
\]
\end{notation}
If $\phi$ is an element of some set $\Phi_c(m_-,m_+)$ for some positive quantity $c$, then it follows from system \cref{syst_trav_front} that
\begin{equation}
\label{V_of_u_plus_minus_V_of_u_minus}
V(m_+)-V(m_-) = c \int_{\rr}\phi'(\xi)^2 \, d\xi
\,,
\end{equation}
so that $V(m_-)$ is less than $V(m_+)$ and $m_-$ and $m_+$ differ; in this case the function $\phi$ is thus the profile of a travelling front. Since its asymptotic values $m_-$ and $m_+$ belong to $\mmm$, this front is qualified as \emph{bistable}. 
\subsubsection{Propagating terraces of bistable travelling fronts}
\label{subsubsec:def_prop_terrace}
\begin{figure}[!htbp]
\centering
\includegraphics[width=.8\textwidth]{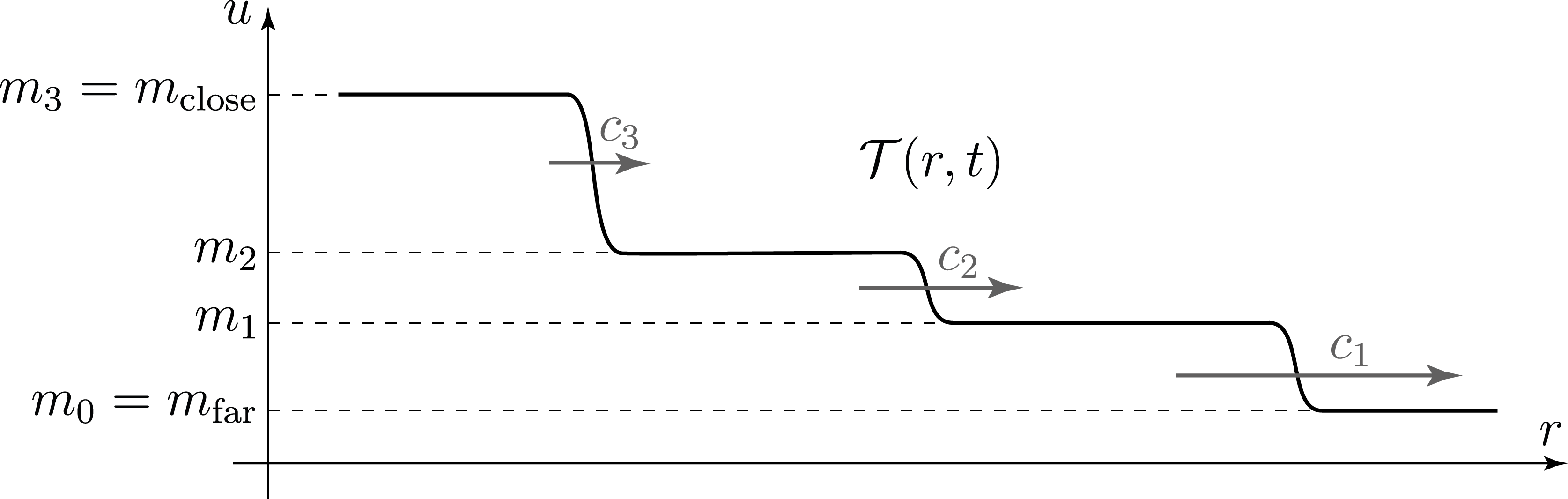}
\caption{Propagating terrace of (bistable) fronts (travelling to the right).}
\label{fig:prop_terrace}
\end{figure}
\begin{definition}[propagating terrace of bistable travelling fronts, \cref{fig:prop_terrace}]
\label{def:propagating_terrace}
Let $\mClose$ and $\mFar$ be two points of $\mmm$ (satisfying $V(\mClose)\le V(\mFar)$). A function 
\[
\ttt : \rr\times[0,+\infty)\to\rr^{\dState},\quad (r,t)\mapsto \ttt(r,t)
\]
is called a \emph{propagating terrace of bistable fronts travelling to the right, connecting $\mClose$ to $\mFar$,} if there exists a nonnegative integer $q$ such that:
\begin{enumerate}
\item if $q$ equals $0$, then $\mClose=\mFar$ and, for every real quantity $r$ and every nonnegative time $t$, 
\[
\ttt(r,t)=\mClose=\mFar
\,;
\]
\item if $q$ equals $1$, then there exist
\begin{itemize}
\item a positive quantity $c_1$
\item and a function $\phi_1$ in $\Phi_c(\mClose,\mFar)$ (that is, the profile of a bistable front travelling at the speed $c_1$ and connecting $\mClose$ to $\mFar$)
\item and a $\ccc^1$-function $[0,+\infty)\to[0,+\infty)$, $t\mapsto r_1(t)$, satisfying $r_1'(t)\to c_1$ as time goes to $+\infty$
\end{itemize}
such that, for every real quantity $r$ and every nonnegative time $t$, 
\[
\ttt(r,t)=\phi_1\bigl(r-r_1(t)\bigr)
\,;
\]
\label{item:def_propagating_terrace_q_equals_one}
\item if $q$ is not smaller than $2$, then there exists $q-1$  points $m_1,\dots,m_{q-1}$ of $\mmm$, satisfying (if $\mFar$ is denoted by $m_0$ and $\mClose$ by $m_q$)
\[
V(m_0)>V(m_1)>\dots>V(m_q)
\,,
\]
and there exist $q$ positive quantities $c_1$, …, $c_q$ satisfying:
\[
c_1\ge\dots\ge c_q
\,,
\]
and for every integer $i$ in $\{1,\dots,q\}$, there exist:
\begin{itemize}
\item a function $\phi_i$ in $\Phi_{c_i}(m_i,m_{i-1})$ (that is, the profile of a bistable front travelling at the speed $c_i$ and connecting $m_i$ to $m_{i-1}$)
\item and a $\ccc^1$-function $[0,+\infty)\to[0,+\infty)$, $t\mapsto r_i(t)$, satisfying $r_i'(t)\to c_i$ as time goes to $+\infty$
\end{itemize}
such that, for every integer $i$ in $\{1,\dots,q-1\}$, 
\[
r_{i+1}(t)-r_i(t)\to +\infty 
\quad\text{as}\quad
t\to +\infty
\,,
\]
and such that, for every real quantity $r$ and every nonnegative time $t$, 
\[
\ttt(r,t) = m_0 + \sum_{i=1}^q \Bigl[\phi_i\bigl(r-r_i(t)\bigr)-m_{i-1}\Bigr]
\,.
\]
\label{item:def_propagating_terrace_q_larger_than_one}
\end{enumerate}
\end{definition}
\begin{remarks}
\begin{enumerate}
\item Item \cref{item:def_propagating_terrace_q_equals_one} may have been omitted in this definition, since it boils down to item \cref{item:def_propagating_terrace_q_larger_than_one} with $q$ equals $1$. 
\item It would be interesting to investigate whether \cref{thm:main} (the main result of this paper, stated below) still holds with more refined estimates on the positions of the travelling fronts involved in \cref{def:propagating_terrace} above. In particular, beyond the convergence ``$r_i'(t)\to c_i$'' stated in this definition, taking into account the curvature term in the differential system \cref{syst_rad_sym} should lead to asymptotics of the form: 
\[
r_i(t) = c_i t - \frac{d-1}{c_i}\log(t) +\dots
\,,
\]
see for instance \cite[Theorem 1.1]{DuMatano_radialTerraceSolutionsPropProfileRN_2020} in the scalar case $\dState$ equals $1$. 
\end{enumerate}
\end{remarks}
The terminology ``propagating terrace'' was introduced by A. Ducrot, T. Giletti, and H. Matano in \cite{DucrotGiletti_existenceConvergencePropagatingTerrace_2014} (and subsequently used by several other authors \cite{Polacik_propagatingTerracesAsymptOneDimSym_2017,Polacik_propTerracesProofGibbonsConj_2016,GilettiRossi_pulsatingSolMultBistMultiStab_2019,MatanoPolacik_dynNonnegSolOneDimRDII_2020,Polacik_propagatingTerracesDynFrontLikeSolRDEquationsR_2020,GilettiMatano_existenceUniquenessPropTerr_2020,PauthierRademacherU_WeakStrongInteractKinks_2021}) to denote a stacked family (a layer) of travelling fronts in a (scalar) reaction-diffusion equation. This led the author to keep the same terminology in the present context. This terminology is convenient to denote objects that would otherwise require a long description. It is also used in the companion papers \cite{Risler_globalRelaxation_2016,Risler_globalBehaviour_2016,Risler_globalBehaviourHyperbolicGradient_2017}. Additional comments on this terminological choice are provided in \cite{Risler_globalBehaviour_2016}.
\subsubsection{Asymptotic pattern stable at infinity}
\label{subsubsec:def_asympt_patt}
\begin{figure}[!htbp]
\centering
\includegraphics[width=.6\textwidth]{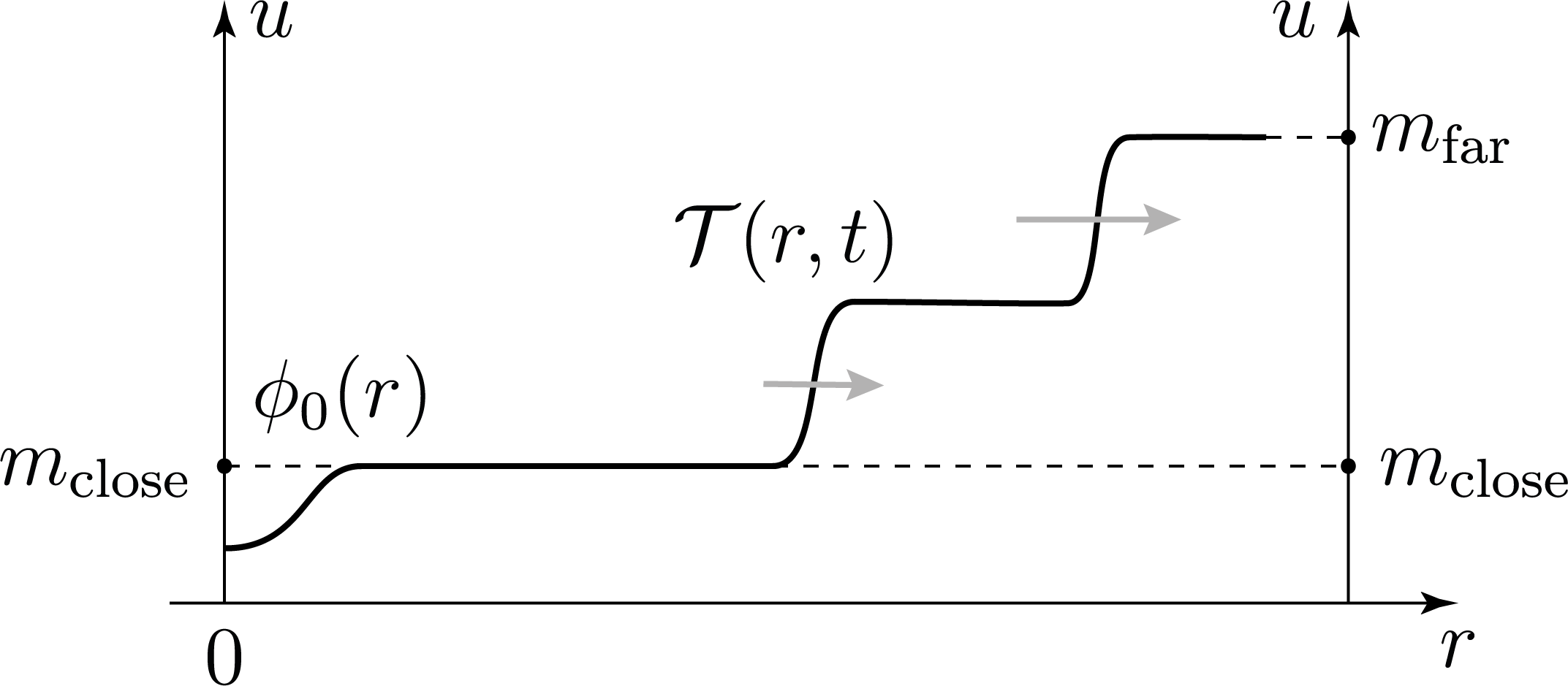}
\caption{Asymptotic pattern stable at infinity.}
\label{fig:bist_asympt_pattern}
\end{figure}
\begin{definition}[asymptotic pattern stable at infinity, \cref{fig:bist_asympt_pattern}]
Let $\mFar$ be a point of $\mmm$. A function
\[
\ppp : [0,+\infty)\times[0,+\infty)\to\rr^{\dState},\quad (r,t)\mapsto \ppp(r,t)
\]
is called an \emph{asymptotic pattern stable close to $\mFar$ at infinity} if there exists:
\begin{itemize}
\item a point $\mClose$ in $\mmm$, 
\item and a propagating terrace $\ttt$ of bistable fronts travelling to the right, connecting $\mClose$ to $\mFar$, 
\item and a stationary solution $\phi_0$ in $\PhiZeroCentre(\mClose)$,
\end{itemize} 
such that, for every nonnegative quantity $r$ and for every nonnegative time $t$, 
\[
\ppp(r,t) = \phi_0(r) + \bigl( \ttt(r,t) - \mClose \bigr)
\,.
\]
\end{definition}
\subsection{Generic hypotheses on the potential}
\label{subsec:generic_assupmt_pot}
\subsubsection{Escape distance of a minimum point}
\label{subsubsec:Escape_dist}
\begin{notation}
For every $u$ in $\rr^d$, let $\sigma\bigl(D^2V(u)\bigr)$ denote the spectrum (the set of eigenvalues) of the Hessian matrix of $V$ at $u$, and let $\eigVmin(u)$ denote the minimum of this spectrum:
\begin{equation}
\label{def_eigVmin_of_u}
\eigVmin(u) = \min \Bigl(\sigma\bigl(D^2V(u)\bigr)\Bigr)
\,.
\end{equation}
\end{notation}
\begin{definition}[Escape distance of a nondegenerate minimum point]
For every $m$ in $\mmm$, let us call \emph{Escape distance of $m$}, and let us denote by $\dEsc(m)$, the supremum of the set
\begin{equation}
\label{set_for_definition_Escape_distance}
\Bigl\{\delta \in[0,1]: \text{ for all } u \text{ in } \rr^d \text{ satisfying } \abs{u-m}_{\ddd}\le \delta, \quad\eigVmin(u) \ge\frac{1}{2} \eigVmin(m) \Bigr\}
\,.
\end{equation}
\end{definition}
Since the quantity $\eigVmin(u)$ varies continuously with $u$, this Escape distance $\dEsc(m)$ is positive (thus in $(0,1]$). In addition, for all $u$ in $\rr^d$ such that $\abs{u-m}_{\ddd}$ is not larger than $\dEsc(m)$, the following inequality holds:
\begin{equation}
\label{property_dEsc}
\eigVmin(u) \ge\frac{1}{2} \eigVmin(m)
\,.
\end{equation}
\begin{remark}
This notation $\dEsc(m)$ refers to the word ``distance'' (and ``Escape'') and should not be mingled with the space dimension $d$. 
\end{remark}
\subsubsection{Breakup of space translation invariance for travelling fronts}
\label{subsubsec:breakup}
For every ordered pair $(m_-,m_+)$ of points of $\mmm$, every positive quantity $c$, and every function $\phi$ in $\Phi_c(m_-,m_+)$, 
\[
\sup_{\rho\in\rr}\abs{\phi(\rho)-m_-}>\dEsc(m_-)
\quad\text{and}\quad
\sup_{\rho\in\rr}\abs{\phi(\rho)-m_+}>\dEsc(m_+)
\,,
\]
see \cref{fig:esc_distance}. For a proof of this standard result, see for instance \cite[\GlobalBehaviourLemTWApproachCriticalPoints]{Risler_globalBehaviour_2016}. 
\begin{figure}[!htbp]
\centering
\includegraphics[width=0.5\textwidth]{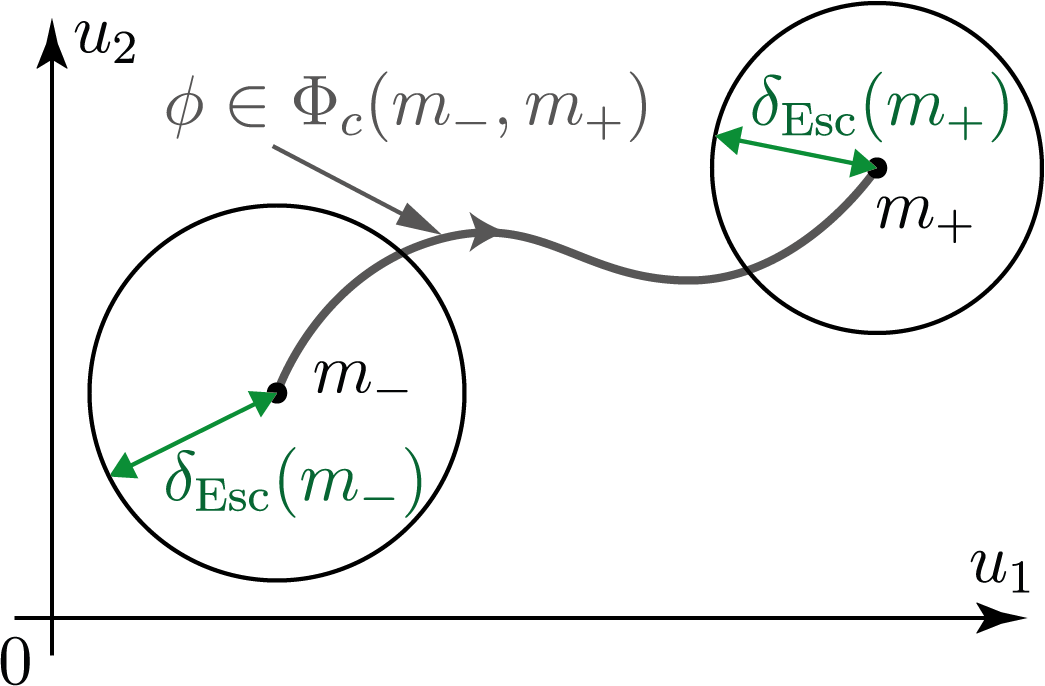}
\caption{Every function in $\Phi_{c}(m_-,m_+)$ escapes at least at distance $\dEsc(m_-)$ of $m_-$ and at distance $\dEsc(m_+)$ of $m_+$.}
\label{fig:esc_distance}
\end{figure}
Thus, for every positive quantity $c$ and every ordered pair $(m_-,m_+)$ of points of $\mmm$, let us introduce the set of \emph{normalized bistable fronts (travelling at the speed $c$) connecting $m_-$ to $m_+$}, defined as
\begin{figure}[!htbp]
\centering
\includegraphics[width=0.75\textwidth]{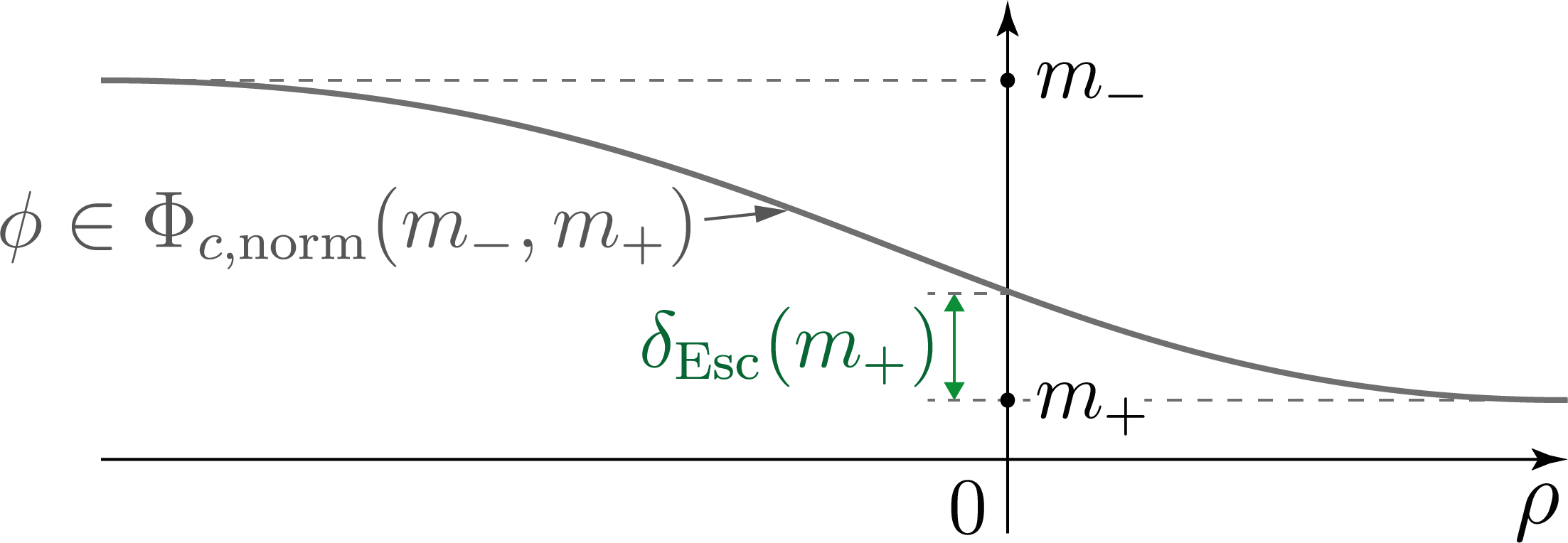}
\caption{Normalized bistable travelling front.}
\label{fig:norm_stat}
\end{figure}
\begin{equation}
\label{def_norm}
\begin{aligned}
\PhicNorm{c}(m_-,m_+) = \bigl\{ & \phi\in\Phi_{c}(m_-,m_+): \abs{\phi(0)-m_+} =\dEsc(m_+) \quad \text{and}\\
& 
\abs{\phi(\rho)-m_+} <\dEsc(m_+)\quad\text{for every positive quantity }\rho\bigr\} 
\,,
\end{aligned}
\end{equation}
 see \cref{fig:norm_stat}.
\subsubsection{Statement of the generic hypotheses}
\label{subsubsec:add_hyp_V}
The results of this paper require a number of generic hypotheses on the potential $V$, that will now be stated. 
\begin{notation}
If $m_+$ is a point in $\mmm$ and $c$ is a \emph{positive} quantity, let $\Phi_c(m_+)$ denote the set of bounded (thus globally defined) profiles of \emph{nonconstant} waves travelling at the speed $c$ and ``invading'' the homogeneous equilibrium $m_+$; with symbols, 
\[
\begin{aligned}
\Phi_c(m_+) = \bigl\{ &
\phi:\rr\to\rr^d : 
\phi \text{ is a \emph{nonconstant} global solution of system \cref{syst_trav_front}}
\\
& \text{and}\quad
\sup_{\rho\in\rr}\abs{\phi(\rho)}<+\infty
\quad\text{and}\quad
\phi(\rho)\xrightarrow[\rho\to +\infty]{} m_+
\bigr\}
\,,
\end{aligned}
\]
and let 
\[
\begin{aligned}
\PhicNorm{c}(m_+) = \bigl\{ & \phi\in\Phi_{c}(m_+): \abs{\phi(0)-m_+} =\dEsc(m_+)\quad \text{and}\\
& 
\abs{\phi(\rho)-m_+} <\dEsc(m_+)\quad\text{for all}\quad \rho \text{ in } (0,+\infty)\bigr\} 
\,. \\
\end{aligned}
\]
\end{notation}
Here are the six generic hypotheses that will be required. 
\begin{description}
\item[\hypOnlyBistLabel]\hypertarget{hypOnlyBist} Every nonconstant bounded wave travelling at a nonzero speed and invading a stable equilibrium (a point of $\mmm$) is a bistable travelling front. With symbols, for every $m_+$ in $\mmm$ and every positive quantity $c$, 
\[
\begin{aligned}
\Phi_c(m_+) &= \bigcup_{m_-\in\mmm} \Phi_c(m_-,m_+) \,,\\
\text{or equivalently}\quad
\PhicNorm{c}(m_+) &= \bigcup_{m_-\in\mmm} \PhicNorm{c}(m_-,m_+)
\,.
\end{aligned}
\]
\item[\hypDiscVelLabel]\hypertarget{hypDiscVel} For every $m_+$ in $\mmm$, the set:
\[
\bigl\{ c\text{ in }[0,+\infty) : \Phi_c(m_+)\not=\emptyset \bigr\} 
\]
has an empty interior. 
\item[\hypDiscFrontLabel]\hypertarget{hypDiscFront} For every point $m_+$ in $\mmm$ and every positive quantity $c$, the set
\[
\bigl\{ \bigl(\phi(0),\phi'(0)\bigr) : \phi\in\PhicNorm{c}(m_+) \bigr\}
\]
is totally discontinuous --- if not empty --- in $\rr^{2\dState}$. That is, its connected components are singletons. Equivalently, the set $\PhicNorm{c}(m_+)$ is totally disconnected for the topology of compact convergence (uniform convergence on compact subsets of $\rr$).
\end{description}
In these two last definitions, the subscript ``disc'' refers to the concept of ``discontinuity'' or ``discreteness''. The following hypothesis will be required to ensure that the number of travelling fronts involved in the asymptotic behaviour of a radially symmetric solution stable at infinity is finite: 
\begin{description}
\item[\hypCriticalValuesLabel]\hypertarget{hypCriticalValues} The set of \emph{critical values} of $V$, that is the set
\[
\bigl\{V(u) : u\in\rr^d\text{ and }\nabla V(u)=0 \bigr\} 
\,,
\]
is finite. 
\end{description}
The next hypothesis is the analogue of \textup{(\hyperlink{hypDiscFront}{\hypDiscFrontRef})} for radially symmetric stationary solutions. 
\begin{description}
\item[\hypDiscStationaryLabel]\hypertarget{hypDiscStationary} For every point $m$ in $\mmm$, the set
\[
\bigl\{ \phi(0) : \phi\in\PhiZeroCentre(m) \bigr\}
\]
is totally discontinuous in $\rr^{\dState}$. That is, its connected components are singletons. Equivalently, the set $\PhiZeroCentre(m)$ is totally disconnected for the topology of compact convergence (uniform convergence on compact subsets of $[0,+\infty)$).
\end{description}
Finally, let us us call \cref{hyp_gen} the union of these five generic hypotheses:
\begin{gather}
\tag{G}
\text{\textup{(\hyperlink{hypOnlyBist}{\hypOnlyBistRef})} and \textup{(\hyperlink{hypDiscVel}{\hypDiscVelRef})} and \textup{(\hyperlink{hypDiscFront}{\hypDiscFrontRef})} and \textup{(\hyperlink{hypCriticalValues}{\hypCriticalValuesRef})} and \textup{(\hyperlink{hypDiscStationary}{\hypDiscStationaryRef})}}.
\label{hyp_gen}
\end{gather}
A formal proof of the genericity of these hypotheses is provided in \cite{JolyRisler_genericTransversalityTravStandFrontsPulses_2023} (for \textup{(\hyperlink{hypOnlyBist}{\hypOnlyBistRef})}, \textup{(\hyperlink{hypDiscVel}{\hypDiscVelRef})}, \textup{(\hyperlink{hypDiscFront}{\hypDiscFrontRef})}, and \textup{(\hyperlink{hypCriticalValues}{\hypCriticalValuesRef})}) and in \cite{Risler_genericTransvRadSymStatSol_2023} (for \textup{(\hyperlink{hypDiscStationary}{\hypDiscStationaryRef})}). 
\subsection{Main result}
\begin{theorem}[global asymptotic behaviour]
\label{thm:main}
Let $V$ denote a function in $\ccc^2(\rr^{\dState},\rr)$ satisfying the coercivity hypothesis \cref{hyp_coerc} and the generic hypotheses \cref{hyp_gen}. Then, for every solution stable at infinity $(r,t)\mapsto u(r,t)$ of system \cref{syst_rad_sym}, there exists an asymptotic pattern $\ppp$ stable at infinity such that
\[
\sup_{r\in[0,+\infty)}\abs{u(r,t)-\ppp(r,t)}\to 0
\quad\text{as}\quad
t\to + \infty
\,.
\]
\end{theorem}
\subsection{Additional results}
%
\subsubsection{Residual asymptotic energy}
%
Here is an additional conclusion to \cref{thm:main}. 
\begin{proposition}[residual asymptotic energy]
\label{prop:resid_asympt_energy}
Assume that the assumptions of \cref{thm:main} hold. With the notation of this theorem, if $\mClose$ and $\mFar$ denote the two points of $\mmm$ such that the propagating terrace $\ttt$ involved in the asymptotic pattern of the solution connects $\mClose$ to $\mFar$, and if $\phi_0$ denotes the function of $\PhiZeroCentre(\mClose)$ involved in this asymptotic pattern, then, for every small enough positive quantity $\varepsilon$, 
\[
\int_{0}^{\varepsilon t} r^{d-1}\Bigl( \frac{1}{2}u_r(r,t)^2 + V\bigl(u(r,t)\bigr) - V(\mClose)\Bigr) \, dr \to \eee[\phi_0]
\quad\text{as}\quad
t\to+\infty
\,.
\]
\end{proposition}
The quantity $\eee[\phi_0]$ may be called the \emph{residual asymptotic energy} of the solution. 
\subsubsection{``Mountain pass'' existence of a ``ground state''}
%
Assume that $V$ satisfies hypothesis \cref{hyp_coerc}. 
\begin{notation}
If $m$ is a point in $\mmm$, let $\basatt(m)$ denote the basin of attraction (for the semi-flow of system \cref{syst_rad_sym}) of the homogeneous equilibrium $m$:
\[
\basatt(m)=\bigl\{ u_0\in Y: (S_t u_0)(r) \to m \text{ , uniformly with respect to } r, \text{ as } t\to+\infty \bigr\} \,,
\]
and let $\partial\basatt(m)$ denote the topological border, in $Y$, of $\basatt(m)$. 
\end{notation}
The following statement can be seen as the ``semi-flow'' version of a standard result ensuring the existence of a ``ground state'' for system \cref{syst_rad_sym_stationary}. Variants of this existence result have been established in numerous references, for instance in \cite{BerestyckiLionsPeletier_odeExistPosSolSemilin_1981} (by a direct ``shooting'' method probably specific to the scalar case where $\dState$ is equal to $1$), and in \cite{BerestyckiLions_existenceGroundState_1983} (by a more general variational method).
\begin{proposition}[``mountain pass'' existence of a ``ground state'' and attractor of the border of the basin of attraction of a stable homogeneous equilibrium]
\label{prop:mountain_pass_existence_ground_state}
Assume that $V$ satisfies hypothesis \cref{hyp_coerc} and let $m$ be a point in $\mmm$ which is not a global minimum point of $V$. Then the following conclusions hold. 
\begin{enumerate}
\item There exists at least one nonconstant function in $\PhiZeroCentre(m)$.
\item The set $\partial\basatt(m)\cap\YstabInfty(m)$ is nonempty.
\item For every solution $(r,t)\mapsto u(r,t)$ of system \cref{syst_rad_sym} in this set $\partial\basatt(m)\cap\YstabInfty(m)$, there exists a function $\phi$ in $\PhiZeroCentre(m)$ such that $\phi$ is not identically equal to $m$ and such that 
\[
\abs{u(r,t)-\phi(r)}\to 0
\quad\text{as}\quad 
t\to+\infty
\,,
\]
uniformly with respect to $r$ in $[0,+\infty)$. 
\end{enumerate}
\end{proposition}
\section{Preliminaries}
\label{sec:preliminaries}
As everywhere else, let us consider a function $V$ in $\ccc^2(\rr^{\dState},\rr)$ satisfying the coercivity hypothesis \cref{hyp_coerc}. 
\subsection{Global existence of solutions and attracting ball for the semi-flow}
\label{subsec:glob_exist}
\begin{proposition}[global existence of solutions and attracting ball]
\label{prop:attr_ball}
For every function $u_0$ in $Y$, system \cref{syst_rad_sym} has a unique globally defined solution $t\mapsto S_t u_0$ in $\ccc^0([0,+\infty),Y)$ with initial condition $u_0$. In addition, there exist a positive quantity $\Rattinfty$ (radius of attracting ball for the $L^\infty$-norm), depending only on $V$, such that, for every large enough positive time $t$, 
\[
\begin{aligned}
\norm{r\mapsto(S_t u_0)(r)}_{\Linfty} &\le \Rattinfty \,, \\
\text{and}\qquad
\norm{r\mapsto(S_t u_0)(r)}_{\Honeul} &\le \RHoneul 
\,.
\end{aligned}
\]
\end{proposition}
\begin{proof}
For a proof of this rather standard result, see for instance \cite[\InvasionRelaxationPropAttBall]{Risler_noInvasionCaseHigherSpace_2020}, which provides identical conclusions for system \cref{syst_higher_dim} (without the radial symmetry hypothesis). 
\end{proof}
In addition, system~\cref{syst_rad_sym} has smoothing properties (Henry \cite{Henry_geomSemilinParab_1981}). Due to these properties, since $V$ is of class $\ccc^2$, for every quantity $\alpha$ in the interval $(0,1)$, every solution $t\mapsto S_t u_0$ in $\ccc^0([0,+\infty),Y)$ actually belongs to
\[
\ccc^0\left((0,+\infty),\cccb{2,\alpha}\right) \cap \ccc^1\left((0,+\infty),\cccb{0,\alpha}\right),
\]
and, for every positive quantity $\varepsilon$, the quantities
\begin{equation}
\label{bound_u_ut_ck}
\sup_{t\ge\varepsilon}\norm{S_t u_0}_{\cccb{2,\alpha}}
\quad\text{and}\quad
\sup_{t\ge\varepsilon}\norm{\frac{d(S_t u_0)}{dt}(t)}_{\cccb{0,\alpha}}
\end{equation}
are finite. 
\subsection{Asymptotic compactness of solutions}
\label{subsec:compactness}
The next two lemmas will be used in the proofs of \cref{prop:inv_cv,prop:relax_implies_cv}. 
\begin{lemma}[asymptotic compactness in the infinite radius limit]
\label{lem:compactness}
For every solution $(r,t)\mapsto u(r,t)$ of system \cref{syst_rad_sym}, and for every sequence $(r_n,t_n)_{n\in\nn}$ in $[0,+\infty)^2$ such that $r_n$ and $t_n$ go to $+\infty$ as $n$ goes to $+\infty$, there exists a entire solution $\widebar{u}$ of system \cref{syst_rad_sym_large_radius} in 
\[
\ccc^0\left(\rr,\cccbR{2}\right)\cap \ccc^1\left(\rr,\cccbR{0}\right)
\,,
\]
such that, up to replacing the sequence $(r_n,t_n)_{n\in\nn}$ by a subsequence, 
\begin{equation}
\label{compactness}
D^{2,1}u(r_n+\cdot,t_n+\cdot)\to D^{2,1}\widebar{u}
\quad\text{as}\quad
n\to+\infty
\,,
\end{equation}
uniformly on every compact subset of $\rr^2$, where the symbol $D^{2,1}v$ stands for $(v,v_r,v_{rr},v_t)$ (for $v$ equal to $u$ or $\widebar{u}$). 
\end{lemma}
\begin{lemma}[asymptotic compactness close to the origin]
\label{lem:compactness_close_to_origin}
For every solution $(r,t)\mapsto u(r,t)$ of system \cref{syst_rad_sym}, and for every sequence $(t_n)_{n\in\nn}$ in $[0,+\infty)$ such that $t_n\to+\infty$ as $n\to+\infty$, there exists a entire solution $\widebar{u}$ of system \cref{syst_rad_sym} in 
\[
\ccc^0\left(\rr,\cccb{2}\right)\cap \ccc^1\left(\rr,\cccb{0}\right)
\,,
\]
such that, up to replacing the sequence $(t_n)_{n\in\nn}$ by a subsequence, 
\begin{equation}
\label{compactness_close_to_origin}
D^{2,1}u(x_n+\cdot,t_n+\cdot)\to D^{2,1}\widebar{u}
\quad\text{as}\quad
n\to+\infty
\,,
\end{equation}
uniformly on every compact subset of $[0,+\infty)\times\rr$
\end{lemma}
\begin{proof}[Proofs of \cref{lem:compactness,lem:compactness_close_to_origin}]
See \cite[1963]{MatanoPolacik_entireSolutionBistableParabEquTwoCollidingPulses_2017} or the proof of \cite[\GlobalRelaxationLemAsymptCompactness]{Risler_globalRelaxation_2016}. 
\end{proof}
\subsection{Time derivative of (localized) energy and \texorpdfstring{$L^2$}{L2}-norm of a solution}
\label{subsec:1rst_ord}
Let $(r,t)\mapsto u(r,t)$ be a solution of system \cref{syst_rad_sym} and $m$ be a point of $\mmm$. Let us assume, in the next calculations, that $t$ is positive, so that according to \cref{bound_u_ut_ck} the regularities of $u$ and $u_t$ ensure that all integrals converge. 
\subsubsection{Standing frame}
Let $r\mapsto \psi(r)$ denote a function in the space $W^{1,1}\bigl([0,+\infty),\rr\bigr)$ (that is a function belonging to $L^1\bigl([0,+\infty)\bigr)$ together with its first derivative), and let us introduce the energy (Lagrangian) and the $L^2$-norm of the distance to $m$, localized by the weight function $\psi$:
\[
\int_0^{+\infty} \psi(r) \Bigl(\frac{1}{2}u_r(r,t)^2 + V\bigl(u(r,t)-V(m)\bigr)\Bigr) \, dr
\quad\text{and}\quad
\int_0^{+\infty} \psi(r) \frac{1}{2}\bigl(u(r,t)-m\bigr)^2\, dr
\,.
\]
Let us assume that $\psi(0) = 0$, and, to simplify the presentation, let us assume that
\[
m = 0_{\rr^{\dState}}
\quad\text{and}\quad
V(m) = V(0_{\rr^{\dState}}) = 0
\,.
\]
Then the time derivatives of the two integrals above read:
\begin{equation}
\label{ddt_loc_en_stand_fr}
\frac{d}{dt} \int_0^{+\infty} \psi\Bigl(\frac{1}{2}u_r^2 + V(u)\Bigr) \, dr =  \int_0^{+\infty} \biggl[ - \psi u_t^2 
+ \Bigl( \frac{d-1}{r} \psi - \psi' \Bigl) u_t \cdot u_r \biggr] \, dr
\,,
\end{equation}
and
\begin{equation}
\label{ddt_loc_L2_stand_fr}
\frac{d}{dt} \int_0^{+\infty} \psi\frac{1}{2}u^2\, dr = \int_0^{+\infty} \biggl[ \psi \bigl( - u \cdot \nabla V(u) - u_r^2 \bigr)
+ \Bigl( \frac{d-1}{r} \psi - \psi' \Bigl) u \cdot u_r \biggr] \, dr
\,.
\end{equation}
In both expressions, the border term at $r$ equals $0$ coming from the integration by parts vanishes since $\psi(0) = 0$. In both expressions again, the last term disappears on every domain where $\psi(r)$ is proportional to $r^{d-1}$ (this corresponds to a uniform weight for the Lebesgue measure on $\rr^d$).  

More comments on these expressions are provided in \cite{Risler_globalRelaxation_2016}. The sole difference with the one-dimensional space case treated in \cite{Risler_globalRelaxation_2016} is the ``$(d-1)/r$'' curvature terms on the right-hand side of these expressions. Fortunately, this additional term will not induce many changes with respect to the arguments developed in \cite{Risler_globalRelaxation_2016}, since:
\begin{itemize}
\item close to the origin $r=0$, the weight function $\psi$ can be chosen proportional to $r^{d-1}$,
\item far away from the origin $r=0$, this curvature term is just small. 
\end{itemize}
\subsubsection{Travelling frame}
Now let us introduce nonnegative quantities $c$ and $\tInit$ and $\rInit$ (the speed, origin of time, and initial origin of space for the travelling frame respectively, see \vref{fig:trav_fr}). For every nonnegative quantity $s$, let us introduce the interval:
\[
I(s) = [-\rInit-cs,+\infty)
\,,
\]
and, for every $\rho$ in $I(s)$, let
\[
v(\rho,s) = u(r,t)
\quad\text{where}\quad
r = \rInit + cs +\rho 
\quad\text{and}\quad
t = \tInit+s
\]
denote the same solution viewed in a referential travelling at the speed $c$. This function $(\rho,s)\mapsto v(\rho,s)$ is a solution of the system:
\[
v_s - c v_\rho = -\nabla V(v) + \frac{d-1}{\rInit + cs + \rho} v_\rho + v_{\rho\rho}
\,.
\]
This time, let us assume that the weight function $\psi$ is a function of the two variables $\rho$ and $s$, defined on the domain
\[
\bigl\{(\rho,s)\in\rr\times[0,+\infty) : \rho \in I(s) \bigr\}
\]
and such that, for all $s$ in $[0,+\infty)$, the function $\rho\mapsto \psi(\rho,s)$ belongs to $W^{2,1}\bigl(I(s),\rr\bigr)$ and the time derivative $\rho\mapsto \psi_s(\rho, s)$ is defined and belongs to $L^1\bigl(I(s),\rr\bigr)$. Again, let us introduce the energy (Lagrangian) and the $L^2$-norm of the solution, localized by the weight function $\psi$:
\[
\int_{I(s)} \psi(\rho,s) \Bigl(\frac{1}{2}v_\rho(\rho,s)^2 + V\bigl(v(\rho,s)\bigr)\Bigr) \, d\rho
\quad\text{and}\quad
\int_{I(s)}  \psi(\rho,s)\frac{1}{2}v(\rho,s)^2\, d\rho
\,.
\]
Let us assume in addition that, for all $s$ in $[-\tInit,+\infty)$, the functions $\rho\mapsto\psi(\rho,s)$ and $\rho\mapsto\psi_\rho(\rho,s)$ vanish at $\rho = -\rInit -cs$ (at the left end of its domain of definition). Then the time derivatives of these two quantities read:
\begin{equation}
\label{ddt_loc_en_trav_fr}
\begin{aligned}
\frac{d}{ds} \int_{I(s)} \psi \Bigl( \frac{1}{2}v_\rho^2 + V(v)\Bigr) \, d\rho = & \int_{I(s)} \biggl[ - \psi v_s^2 
+ \psi_s \Bigl( \frac{1}{2}v_\rho^2 + V(v) \Bigr)  \\
& + \Bigl( \frac{d-1}{\rInit + cs + \rho} \psi + c \psi - \psi_\rho \Bigl) v_s \cdot v_\rho \biggr] \, d\rho
\,,
\end{aligned}
\end{equation}
and
\begin{equation}
\label{ddt_loc_L2_trav_fr}
\begin{aligned}
\frac{d}{ds} \int_{I(s)} \psi\frac{1}{2}v^2\, d\rho = & \int_{I(s)} \biggl[ \psi \bigl( - v \cdot \nabla V(v) - v_\rho^2 \bigr) 
+ \psi_s \frac{1}{2}v^2 \\
& + \Bigl( \frac{d-1}{\rInit + cs + \rho} \psi + c \psi - \psi_\rho \Bigl) v \cdot v_\rho \biggr] \, d\rho \\
= &  \int_{I(s)} \biggl[ \psi \bigl( - v \cdot \nabla V(v) - v_\rho^2 \bigr) + \frac{\psi_s}{2}v^2 \\
& + \frac{1}{2}(\psi_{\rho\rho} - c \psi_\rho)v^2 +   \frac{(d-1)\psi}{\rInit + cs + \rho}  v \cdot v_\rho \biggr] \, d\rho
\,.
\end{aligned}
\end{equation}
In these expressions again, the integration by part border terms at $\rho = -\rInit-cs$ vanish, and some terms simplify where the quantity
\begin{equation}
\label{ode_weight_tf}
\frac{d-1}{\rInit + cs + \rho} \psi + c \psi - \psi_\rho
\end{equation}
vanishes, that is where $\psi$ is proportional to the expression 
\[
(\rInit + cs + \rho)^{d-1} \exp (c\rho)
\]
(combining the Lebesgue measure and the exponential weight $\exp (c\rho)$). For the time derivative of the $L^2$-functional, a second expression (after integrating by parts the factor $c\psi-\psi_\rho$) is given (it is actually this second expression that will turn out to be the most appropriate for the calculations and estimates to come). 

More comments on these expressions are provided in \cite{Risler_globalBehaviour_2016}. As in the laboratory frame case, the sole difference with the one-dimensional space case treated in \cite{Risler_globalBehaviour_2016} is the ``$(d-1)/(\rInit + cs + \rho)$'' curvature terms on the right-hand side of these expressions. Fortunately, this additional term does not induce many changes with respect to the arguments of \cite{Risler_globalBehaviour_2016}, since:
\begin{itemize}
\item close to the ``origin'' $\rho=-\rInit-cs$, the weight function can be chosen in such a way that the quantity \cref{ode_weight_tf} (involving this curvature term) vanishes or remains small,
\item far away from the origin, this curvature term is just small. 
\end{itemize}
\subsection{Miscellanea}
\label{subsec:misc}
\subsubsection{Second order estimates for the potential around a minimum point}
\begin{lemma}[second order estimates for the potential around a minimum point]
\label{lem:estim_from_def_escape}
For every $m$ in $\mmm$ and every $u$ in $\rr^{\dState}$ satisfying $\abs{u-m}\le\dEsc(m)$, the following estimates hold:
\begin{align}
\label{posit_pot_around_loc_min}
V(u)-V(m) &\ge \frac{\eigVmin(m)}{4} (u-m)^2 \,, \\
\text{and}\qquad
\label{v_nablaV_controls_square_around_loc_min}
(u-m)\cdot \nabla Vu) &\ge \frac{\eigVmin(m)}{2} (u-m)^2 \,, \\
\text{and}\qquad
\label{v_nablaV_controls_pot_around_loc_min}
(u-m)\cdot \nabla V(u) &\ge V(u)-V(m)
\,.
\end{align}
\end{lemma}
\begin{proof}
The three inequalities follow from inequality \vref{property_dEsc} defining $\dEsc(m)$ and from three variants of Taylor's Theorem with Lagrange remainder applied to the function $f$ defined on $[0,1]$ by $f(\theta)=V\bigl(m+\theta (u-m)\bigr)$ (see \cite[\GlobalRelaxationLemEstimFromEscDist]{Risler_globalRelaxation_2016}).
\end{proof}
\subsubsection{Lower quadratic hulls of the potential at minimum points}
\label{subsubsec:low_quad_hull}
As in \cite{Risler_globalRelaxation_2016,Risler_globalBehaviour_2016}, let 
\[
\qLowHull = \min_{m\in\mmm}\ \inf_{u\in\rr^{\dState}\setminus \{m\}}\ \frac{V(u)-V(m)}{(u-m)^2}
\,,
\]
see \cref{fig:low_quad_hull}.
\begin{figure}[!htbp]
\centering
\includegraphics[width=0.5\textwidth]{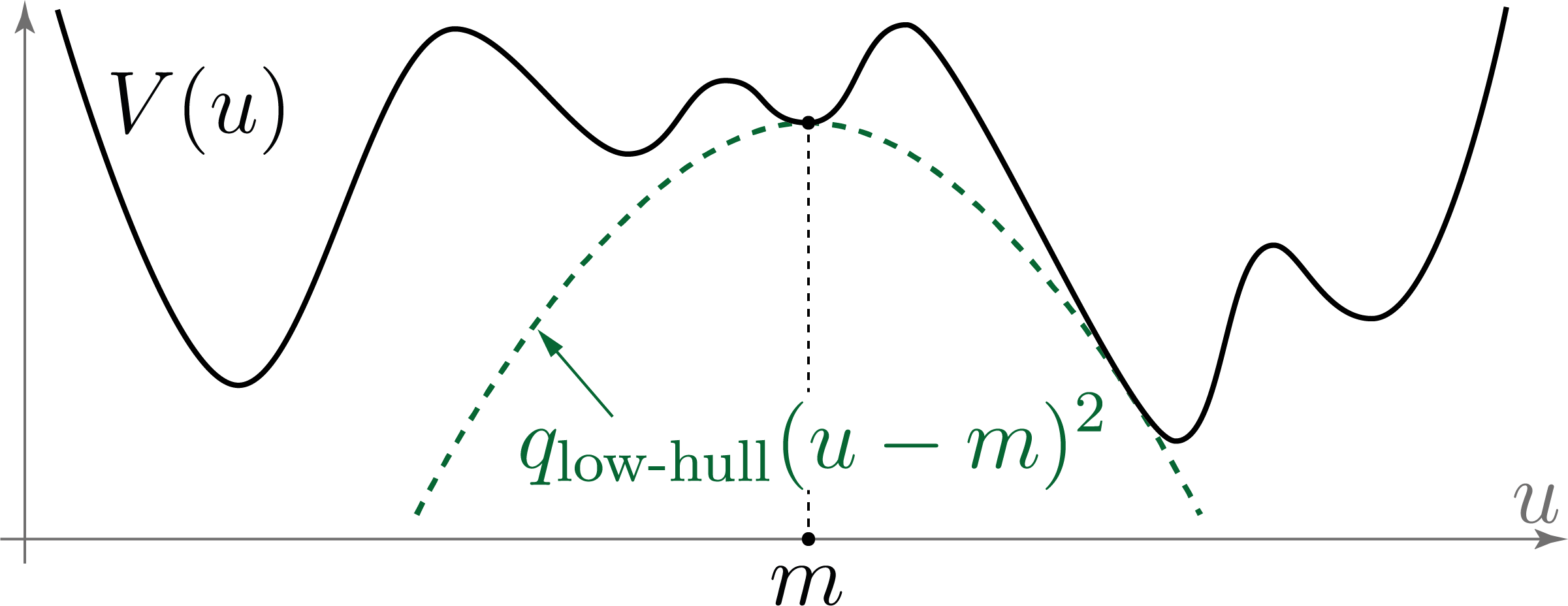}
\caption{Lower quadratic hull of the potential at a minimum point (definition of the quantity $\qLowHull$).}
\label{fig:low_quad_hull}
\end{figure}
and let
\begin{equation}
\label{def_coeffEnZero}
\coeffEnZero=\frac{1}{\max(1, - 4\, \qLowHull)}
\,.
\end{equation}
It follows from this definition that, for every $m$ in the set $\mmm$ and for all $u$ in $\rr^{\dState}$,
\begin{equation}
\label{property_weight_en}
\coeffEnZero \, V(u) +  \frac{1}{4}(u-m)^2\ge 0
\,.
\end{equation}
\section{Invasion implies convergence}
\label{sec:inv_impl_cv}
\subsection{Definitions and hypotheses}
\label{subsec:inv_cv_def_hyp}
As everywhere else, let us consider a function $V$ in $\ccc^2(\rr^{\dState},\rr)$ satisfying the coercivity hypothesis \cref{hyp_coerc}. Let us consider a point $m$ in $\mmm$, a function (initial condition) $u_0$ in $Y$, and the corresponding solution $(r,t)\mapsto u(r,t) = (S_t u_0)(r)$ defined on $[0,+\infty)^2$. 

It will not be assumed that this solution is stable at infinity, but instead, as stated by the next hypothesis \textup{(\hyperlink{hypHom}{\hypHomRef})}, that there exists a growing interval, travelling at a positive speed, where the solution is close to $m$ (the subscript ``hom'' in the definitions below refers to this ``homogeneous'' area), see \cref{fig:inv_cv}.
\begin{figure}[!htbp]
\centering
\includegraphics[width=.8\textwidth]{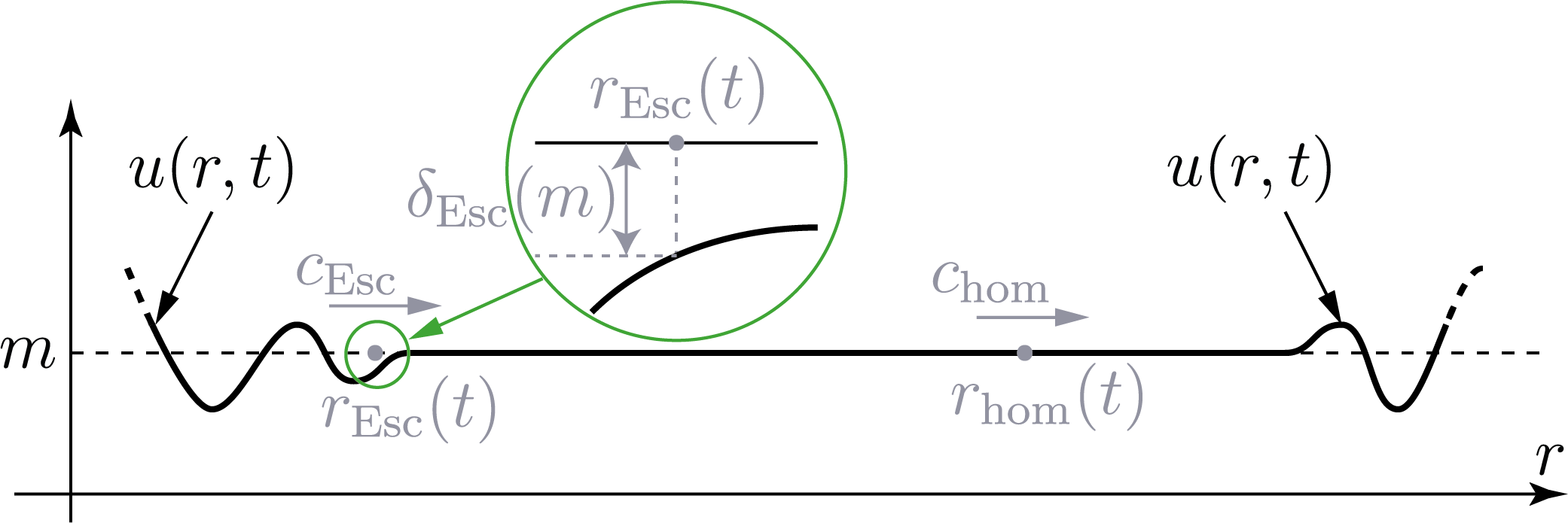}
\caption{Illustration of hypotheses \textup{(\hypHomRef)} and \textup{(\hypInvRef)}.}
\label{fig:inv_cv}
\end{figure}
\begin{description}
\item[\hypHomLabel]\hypertarget{hypHom} There exists a positive quantity $\cHom$ and a $\ccc^1$-function 
\[
\rHom : [0,+\infty)\to \rr\,,
\quad\text{satisfying}\quad
\rHom'(t) \to \cHom
\quad\text{as}\quad
t\to + \infty
\,,
\]
such that, for every positive quantity $L$, 
\[
\sup_{\rho\in[-L,L]}\abs{u\bigl( \rHom(t) + \rho, t\bigr) - m } \to 0
\quad\text{as}\quad
t\to + \infty
\,.
\]
\end{description}
For every $t$ in $[0+\infty)$, let us denote by $\rEsc(t)$ the supremum of the set:
\[
\Bigl\{ r\in \bigl[0,\rHom(t)\bigr] : \abs{u(r,t) - m} = \dEsc(m) \Bigr\}
\,,
\]
with the convention that $\rEsc(t)$ equals $-\infty$ if this set is empty. In other words, $\rEsc(t)$ is the first point at the left of $\rHom(t)$ where the solution ``Escapes'' at the distance $\rEsc$ from the stable homogeneous equilibrium $m$. This point will be referred to the ``Escape point'' (with an upper-case ``E'', by contrast with another ``escape point'' that will be introduced later, with a lower-case ``e'' and a slightly different definition). Let us consider the upper limit of the mean speeds between $0$ and $t$ of this Escape point:
\[
\cEsc = \limsup_{t\to +\infty} \frac{\rEsc(t)}{t}
\,,
\]
and let us make the following hypothesis, stating that the area around $\rHom(t)$ where the solution is close to $m$ is ``invaded'' from the left at a nonzero (mean) speed.
\begin{description}
\item[\hypInvLabel]\hypertarget{hypInv} The quantity $\cEsc$ is positive. 
\end{description}
\subsection{Statement}
\label{subsec:inv_cv_stat}
The aim of \cref{sec:inv_impl_cv} is to prove the following proposition, which is the main step in the proof of \cref{thm:main}. The proposition is illustrated by \cref{fig:inv_cv_bis}.
\begin{figure}[!htbp]
	\centering
    \includegraphics[width=\textwidth]{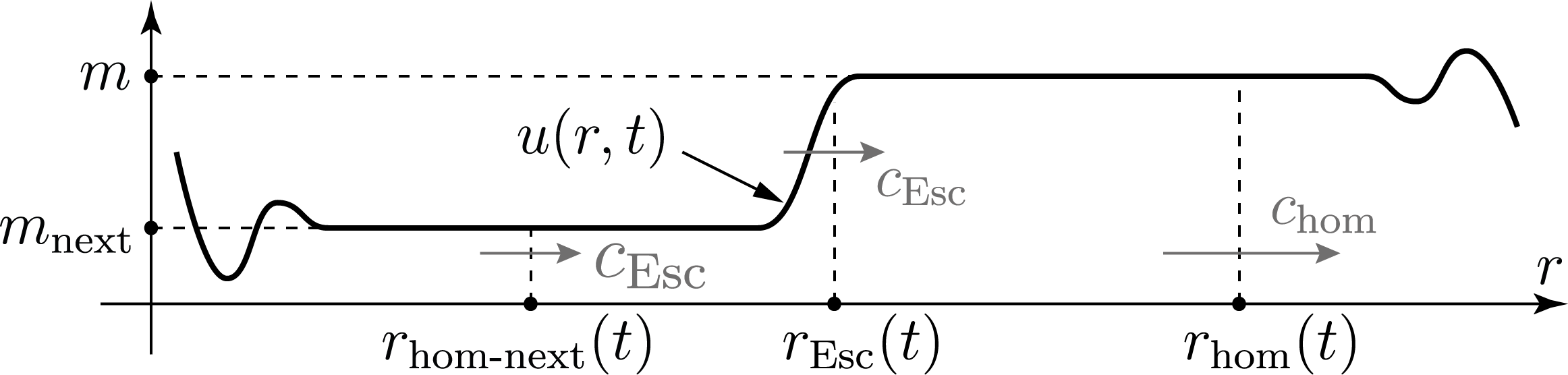}
    \caption{Illustration of \cref{prop:inv_cv}.}
    \label{fig:inv_cv_bis}
\end{figure}
\begin{proposition}[invasion implies convergence]
\label{prop:inv_cv}
Assume that $V$ satisfies the coercivity hypothesis \cref{hyp_coerc} and the generic hypotheses \textup{(\hyperlink{hypOnlyBist}{\hypOnlyBistRef})} and \textup{(\hyperlink{hypDiscVel}{\hypDiscVelRef})} and \textup{(\hyperlink{hypDiscFront}{\hypDiscFrontRef})}, and, keeping the definitions and notation above, let us assume that for the solution under consideration hypotheses \textup{(\hyperlink{hypHom}{\hypHomRef})} and \textup{(\hyperlink{hypInv}{\hypInvRef})} hold. Then the following conclusions hold. 
\begin{itemize}
\item For $t$ large enough positive, the function $t\mapsto \rEsc(t)$ is of class $\ccc^1$ and 
\[
\rEsc'(t)\to \cEsc
\quad\text{as}\quad
t\to +\infty
\,.
\]
\item There exist:
\begin{itemize}
\item a point $\mNext$ in $\mmm$ satisfying $V(\mNext)<V(m)$,
\item a profile of travelling front $\phi$ in $\PhicNorm{\cEsc}(\mNext,m)$,
\item a $\ccc^1$-function $[0,+\infty)\to\rr$, $t\mapsto \rHomNext(t)$,
\end{itemize}
such that, as time goes to $+\infty$, the following limits hold:
\[
\rEsc(t)-\rHomNext(t)\to +\infty
\quad\text{and}\quad
\rHomNext'(t)\to \cEsc
\]
and
\[
\sup_{r\in[\rHomNext(t) \, ,\ \rHom(t)]} \abs{u(r,t) - \phi\bigl( r-\rEsc(t) \bigr)} \to 0 
\]
and, for every positive quantity $L$, 
\[
\sup_{\rho\in[-L,L]}\abs{u\bigl( \rHomNext(t)+\rho,t\bigr) - \mNext}\to 0 
\,.
\]
\end{itemize}
\end{proposition}
\subsection{Set-up for the proof, 1}
\label{subsec:inv_cv_set_pf}
Let us keep the notation and assumptions of \cref{subsec:inv_cv_def_hyp}, and let us assume that the hypotheses \cref{hyp_coerc} and \textup{(\hyperlink{hypOnlyBist}{\hypOnlyBistRef})} and \textup{(\hyperlink{hypDiscVel}{\hypDiscVelRef})} and \textup{(\hyperlink{hypDiscFront}{\hypDiscFrontRef})} and \textup{(\hyperlink{hypHom}{\hypHomRef})} and \textup{(\hyperlink{hypInv}{\hypInvRef})} of \cref{prop:inv_cv} hold. 
\subsubsection{Assumptions holding up to changing the origin of time}
Without loss of generality, up to changing the origin of time, it may be assumed that the following properties hold. 
\begin{itemize}
\item According to \vref{prop:attr_ball} (``global existence of solutions and attracting ball''), it may be assumed that, for all $t$ in $[0,+\infty)$, 
\end{itemize}
\begin{equation}
\label{hyp_attr_ball}
\sup_{r\in[0,+\infty)} \, \abs{u(r,t)} \le \Rattinfty
\,.
\end{equation}
\begin{itemize}
\item According to the bounds \vref{bound_u_ut_ck}, it may be assumed that 
\end{itemize}
\begin{equation}
\label{bound_u_ut_ck_bis}
\begin{aligned}
\sup_{t\ge0} \, \norm{r\mapsto u(r,t)}_{\cccb{2}} &< +\infty \\
\text{and}\quad
\sup_{t\ge0} \, \norm{r\mapsto u_t(r,t)}_{\cccb{0}} &< +\infty
\,.
\end{aligned}
\end{equation}
\begin{itemize}
\item According to \textup{(\hyperlink{hypHom}{\hypHomRef})}, it may be assumed that, for all $t$ in $[0,+\infty)$, 
\end{itemize}
\begin{equation}
\label{hyp_xHom_prime_pos}
\rHom'(t)\ge 0
\,.
\end{equation}
\subsubsection{Normalized potential and corresponding solution}
For notational convenience, let us introduce the ``normalized potential'' $V^\dag$ and the ``normalized solution'' $u^\dag$ defined as
\begin{equation}
\label{def_Vdag_udag}
V^\dag(v)=V(m +v)-V(m)
\quad\text{and}\quad
u^\dag(r,t) = u(r,t)-m
\,.
\end{equation}
Thus the origin $0_{\rr^d}$ of $\rr^d$ is to $V^\dag$ what $m$ is to $V$, it is a nondegenerate minimum point for $V^\dag$ (with $V^\dag(0_{\rr^d})=0$), and $u^\dag$ is a solution of system \cref{syst_rad_sym} with potential $V^\dag$ instead of $V$; and, for all $(r,t)$ in $[0,+\infty)^2$, 
\[
V^\dag\bigl(u^\dag(r,t)\bigr) = V\bigl(u(r,t)\bigr)-V(m)
\,.
\]
It follows from inequality \cref{property_weight_en} satisfied by $\coeffEnZero$ that, for all $v$ in $\rr^d$, 
\begin{equation}
\label{def_weight_en_V_dag}
\coeffEnZero \, V^\dag(v) +  \frac{1}{4}v^2\ge 0
\,,
\end{equation}
and it follows from inequalities \cref{posit_pot_around_loc_min,v_nablaV_controls_square_around_loc_min,v_nablaV_controls_pot_around_loc_min} that, for all $v$ in $\rr^d$ satisfying $\abs{v}\le\dEsc(m)$, 
\begin{align}
\label{v_nablaV_controls_square_around_loc_min_dag}
v\cdot \nabla V^\dag(v) &\ge \frac{\eigVmin(m)}{2} v^2 \,, \\
\text{and}\qquad
\label{v_nablaV_controls_pot_around_loc_min_dag}
v\cdot \nabla V^\dag(v) &\ge V^\dag(v)
\,.
\end{align}
\subsection{Firewall function in the laboratory frame}
\subsubsection{Definition} 
\label{subsubsec:def_fire_zero}
%
Let $\kappa_0$ and $\rSmallCurv$ denote two positive quantities, with $\kappa_0$ small enough and $\rSmallCurv$ large enough so that
\begin{equation}
\label{def_kappa_0_rSmallCurv}
\frac{\coeffEnZero}{4}\Bigl( \frac{d-1}{\rSmallCurv} + \kappa_0 \Bigr)^2 + \frac{1}{4}\Bigl( \frac{d-1}{\rSmallCurv} + \kappa_0 \Bigr) \le \frac{1}{2}
\quad\text{and}\quad
\frac{d-1}{\rSmallCurv} + \kappa_0 \le \frac{\eigVmin(m)}{8}
\end{equation}
(those properties will be used to prove inequality \cref{dt_fire_spat_as} below). Since according to its definition \vref{def_coeffEnZero} the quantity $\coeffEnZero$ is not larger than $1$, these quantities may be chosen as
\begin{equation}
\label{def_kappa_zero}
\kappa_0 = \min \Bigl( \frac{1}{2}\, , \, \frac{\eigVmin(m)}{16} \Bigr)
\end{equation}
and
\begin{equation}
\label{def_r_s_c}
\rSmallCurv = \max \Bigl( 2(d-1) \, , \, \frac{16(d-1)}{\eigVmin(m)} \Bigr)
\,.
\end{equation}
Let us consider the weight function $\psi_0$ defined as
\begin{equation}
\label{def_psi_zero}
\psi_0(r) = \exp(-\kappa_0\abs{r})
\,,
\end{equation}
and, for every quantity $\bar{r}$ greater than or equal to $\rSmallCurv$, let $T_{\bar{r}}\psi_0$ denote the function $[0,+\infty)\to \rr$ defined as
\begin{equation}
\label{def_Tbarr_psi_zero}
T_{\bar{r}}\psi_0 (r) =
\left\{
\begin{aligned}
\psi_0(r-\bar{r})\Bigl(\frac{r}{\rSmallCurv}\Bigr)^{d-1} &\quad\text{if}\quad 0\le r\le \rSmallCurv \,, \\
\psi_0(r-\bar{r}) &\quad\text{if}\quad r\ge \rSmallCurv \,,
\end{aligned}
\right.
\end{equation}
see \cref{fig:graph_weight_firezero}.
\begin{figure}[!htbp]
\centering
\includegraphics[width=\textwidth]{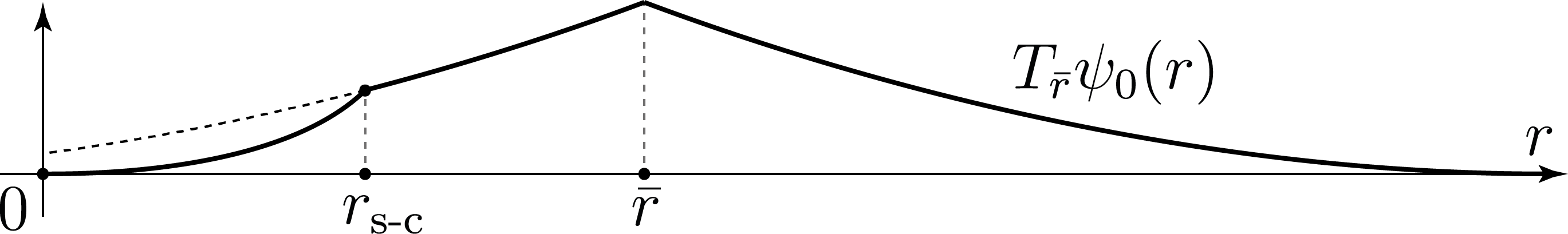}
\caption{Graph of the weight function $r\mapsto T_{\bar{r}}\psi_0(r)$ used to define the firewall function $\fff_0(\bar{r},t)$. The quantity $\rSmallCurv$ is large, and, according to the definition of $\kappa_0$, the slope is small.}
\label{fig:graph_weight_firezero}
\end{figure}

As the following computations will show, for $r$ greater than this quantity $\rSmallCurv$, the ``curvature terms'' that take place in the time derivatives of energy and $L^2$ functionals (see expressions \vref{ddt_loc_en_stand_fr,ddt_loc_L2_stand_fr}) will be small enough for the desired estimates to hold. The subscript ``s-c'' thus refers to ``small curvature'' (or equivalently, ``large radius''). 

Thus, the function $T_{\bar{r}}\psi_0$ defined above is:
\begin{itemize}
\item a translate of the function $\psi_0$ far from the origin (for $r$ greater than $\rSmallCurv$),
\item the same translate multiplied by a factor proportional to the ``Lebesgue measure'' weight $r^{d-1}$ close to the origin (for $r$ smaller than $\rSmallCurv$), this factor being equal to $1$ at $r=\rSmallCurv$ to ensure the continuity of the function. 
\end{itemize}

One purpose of this definition is to control the last terms in the expressions \cref{ddt_loc_en_stand_fr,ddt_loc_L2_stand_fr} for the time derivatives of the energy and $L^2$ functionals. For all nonnegative quantities $\bar{r}$ and $r$ with $\bar{r}$ not smaller than $\rSmallCurv$, 
\[
\frac{d-1}{r} T_{\bar{r}}\psi_0 (r) - T_{\bar{r}}\psi_0' (r) = 
\left\{
\begin{aligned}
&  - \kappa_0 T_{\bar{r}}\psi_0 (r) && \text{if}\quad r<\rSmallCurv \,, \\
& \Bigl( \frac{d-1}{r} - \kappa_0 \Bigr)T_{\bar{r}}\psi_0 (r) && \text{if}\quad \rSmallCurv < r < \bar{r} \,, \\
& \Bigl( \frac{d-1}{r} + \kappa_0 \Bigr)T_{\bar{r}}\psi_0 (r) && \text{if}\quad \bar{r} < r \,, \\
\end{aligned}
\right.
\]
thus, in all three cases,
\begin{equation}
\label{upp_bd_der_T_bar_r_psi_zero}
\abs{\frac{d-1}{r} T_{\bar{r}}\psi_0 (r) - T_{\bar{r}}\psi_0' (r)} \le \Bigl( \frac{d-1}{\rSmallCurv} + \kappa_0  \Bigr) T_{\bar{r}}\psi_0 (r) 
\,.
\end{equation}
For all nonnegative quantities $r$ and $t$, let us introduce the quantities
\begin{equation}
\label{def_Edag_Fdag}
E^\dag(r,t) = \frac{1}{2}u^\dag_r(r,t)^2 + V^\dag \bigl(u^\dag(r,t)\bigr)
\quad\text{and}\quad
F^\dag(r,t) = \coeffEnZero E^\dag(r,t) + \frac{1}{2}u^\dag(r,t)^2
\,,
\end{equation}
and, for all nonnegative quantities $\bar{r}$ and $t$ with $\bar{r}$ not smaller than $\rSmallCurv$, let us introduce the ``firewall'' $\fff_0(\bar{r},t)$ defined as
\begin{equation}
\label{def_fffZero}
\fff_0(\bar{r},t) = \int_0^{+\infty} T_{\bar{r}}\psi_0(r)F^\dag(r,t) \, dr
\,.
\end{equation}
\subsubsection{Coercivity} 
%
\begin{lemma}[firewall coercivity]
\label{lem:coerc_fire_sf}
For every nonnegative time $t$ and nonnegative radius $r$, 
\begin{equation}
\label{coerc_Fdag_sf}
F^\dag(r,t) \ge \min\left(\frac{\coeffEnZero}{2},\frac{1}{4}\right) \bigl( u^\dag_r(r,t)^2 + u^\dag(r,t)^2 \bigr)
\,,
\end{equation}
so that, for every $\bar{r}$ greater than or equal to $\rSmallCurv$, 
\begin{equation}
\label{coerc_fire_sf}
\fff_0(\bar{r},t) \ge \min\left(\frac{\coeffEnZero}{2},\frac{1}{4}\right) \int_{\rr} T_{\bar{r}}\psi_0(r) \bigl( u^\dag_r(r,t)^2 + u^\dag(r,t)^2 \bigr)\, dr
\,.
\end{equation}
\end{lemma}
\begin{proof}
Inequality \cref{coerc_Fdag_sf} follows from inequality \vref{def_weight_en_V_dag} satisfied by $\coeffEnZero$, and \cref{coerc_fire_sf} follows from \cref{coerc_Fdag_sf}. 
\end{proof}
\subsubsection{Linear decrease up to pollution} 
%
For every nonnegative time $t$, let us introduce the following set (the set of radii where the solution ``Escapes'' at a certain distance from $0_{\rr^{\dState}}$):
\begin{equation}
\label{def_Sigma_Esc_0}
\SigmaEscZero(t)=\bigl\{r\in[0,+\infty) : \abs{u^\dag(r,t)} >\dEsc(m)\bigr\}
\,,
\end{equation}
\begin{lemma}[firewall linear decrease up to pollution]
\label{lem:dt_fire_0}
There exist positive quantities $\nuFZero$ and $\KFZero$, depending only on $V$, such that, for all nonnegative quantities $\bar{r}$ and $t$, 
\begin{equation}
\label{dt_fire_0}
\partial_t \fff_0(\bar{r},t)\le-\nuFZero \, \fff_0(\bar{r},t) + \KFZero\, \int_{\SigmaEscZero(t)} T_{\bar{r}}\psi_0(r)\, dr
\,.
\end{equation}
\end{lemma}
\begin{proof}
It follows from expressions \vref{ddt_loc_en_stand_fr,ddt_loc_L2_stand_fr} for the time derivatives of localized energy and $L^2$-functionals that
\[
\begin{aligned}
\partial_t\fff_0(\bar{r},t) = & \int_0^{+\infty} \biggl( T_{\bar{r}}\psi_0 \bigl( - \coeffEnZero \, (u^\dag_t)^2 - u^\dag\cdot\nabla V^\dag(u^\dag) - (u^\dag_r)^2 \bigr) \\
& + \Bigl( \frac{d-1}{r} T_{\bar{r}}\psi_0  - T_{\bar{r}}\psi_0'  \Bigr) \bigl( \coeffEnZero \, u^\dag_t\cdot u^\dag_r + u^\dag\cdot u^\dag_r \bigr) \biggr) \, dr 
\,.
\end{aligned}
\]
Thus, according to the upper bound \cref{upp_bd_der_T_bar_r_psi_zero}, 
\[
\begin{aligned}
\partial_t\fff_0(\bar{r},t) & \le \int_0^{+\infty} T_{\bar{r}}\psi_0 \biggl( - \coeffEnZero \, (u^\dag_t)^2 - u^\dag\cdot\nabla V^\dag(u^\dag) - (u^\dag_r)^2  \\
& + \left( \frac{d-1}{\rSmallCurv} + \kappa_0  \right) \Bigl( \coeffEnZero \, \abs{u^\dag_t\cdot u^\dag_r} + \abs{u^\dag\cdot u^\dag_r} \Bigr) \biggr) \, dr
\,,
\end{aligned}
\]
thus, using the inequalities
\[
\begin{aligned}
\left( \frac{d-1}{\rSmallCurv} + \kappa_0  \right) \abs{u^\dag_t\cdot u^\dag_r} &\le (u^\dag_t)^2 + \frac{1}{4}\left( \frac{d-1}{\rSmallCurv} + \kappa_0  \right)^2 (u^\dag_r)^2 \,, \\
\text{and}\quad
\abs{u^\dag_t\cdot u^\dag_r} &\le \frac{1}{4}(u^\dag_r)^2 + (u^\dag)^2 
\,,
\end{aligned}
\]
it follows that
\[
\begin{aligned}
\partial_t\fff_0(\bar{r},t)\le & \int_0^{+\infty} T_{\bar{r}}\psi_0\Biggl[ \biggl( \frac{\coeffEnZero}{4}\Bigl( \frac{d-1}{\rSmallCurv} + \kappa_0 \Bigr)^2 + \frac{1}{4}\Bigl( \frac{d-1}{\rSmallCurv} + \kappa_0 \Bigr)-1\biggr)(u^\dag_r)^2 \\
& - u^\dag\cdot\nabla V^\dag(u^\dag) + \Bigl( \frac{d-1}{\rSmallCurv} + \kappa_0 \Bigr)(u^\dag)^2 \Biggr] \, dr 
\,,
\end{aligned}
\]
and according to inequalities \cref{def_kappa_0_rSmallCurv} satisfied by the quantities $\kappa_0$ and $\rSmallCurv$, 
\begin{equation}
\label{dt_fire_spat_as}
\partial_t\fff_0(\bar{r},t) \le \int_0^{+\infty} T_{\bar{r}}\psi_0 \Bigl( -\frac{1}{2}(u^\dag_r)^2 - u^\dag\cdot\nabla V^\dag(u^\dag) + \frac{\eigVmin(m)}{8} (u^\dag)^2 \Bigr) \, dr
\,.
\end{equation}
Let $\nuFZero$ be a positive quantity to be chosen below. It follows from the previous inequality and from the definition \cref{def_fffZero} of $\fff_0(\bar{r},t)$ that
\begin{equation}
\label{dt_fire_spat_as_2}
\begin{aligned}
\partial_t\fff_0(\bar{r},t) + \nuFZero \fff_0(\bar{r},t) \le \int_0^{+\infty} T_{\bar{r}}\psi_0 \biggl[ & -\frac{1}{2}(1-\nuFZero\, \coeffEnZero)(u^\dag_r)^2 - u^\dag\cdot\nabla V^\dag(u^\dag) \\
&+ \nuFZero\coeffEnZero V^\dag(u^\dag) + \Bigl(\frac{\eigVmin(m)}{8} + \frac{\nuFZero}{2}\Bigr)(u^\dag)^2\biggr] \, dr
\,.
\end{aligned}
\end{equation}
In view of this inequality and inequalities \vref{v_nablaV_controls_square_around_loc_min_dag,v_nablaV_controls_pot_around_loc_min_dag}, let us assume that $\nuFZero$ is small enough so that
\begin{equation}
\label{def_nu_spat_as}
\nuFZero \, \coeffEnZero \le 1 
\quad\text{and}\quad
\nuFZero \, \coeffEnZero \le \frac{1}{2} 
\quad\text{and}\quad
\frac{\nuFZero}{2}\le \frac{\eigVmin(m)}{8}
\,;
\end{equation}
the quantity $\nuFZero$ may be chosen as
\begin{equation}
\label{def_nuZero}
\nuFZero = \min\left(\frac{1}{2\coeffEnZero}, \frac{\eigVmin(m)}{4}\right)
\,.
\end{equation}
Then, it follows from \cref{dt_fire_spat_as_2,def_nu_spat_as} that
\begin{equation}
\label{dt_fire_spat_as_3}
\partial_t\fff_0(\bar{r},t) + \nuFZero \fff_0(\bar{r},t) \le \int_0^{+\infty} T_{\bar{r}}\psi_0 \Bigl[  - u^\dag\cdot\nabla V^\dag(u^\dag) + \frac{1}{2} \abs{V^\dag(u^\dag)} + \frac{\eigVmin(m)}{4} (u^\dag)^2\Bigr] \, dr
\,.
\end{equation}
According to \cref{v_nablaV_controls_square_around_loc_min_dag,v_nablaV_controls_pot_around_loc_min_dag}, the integrand of the integral at the right-hand side of this inequality is nonpositive as long as $r$ is \emph{not} in $\SigmaEscZero(t)$. Therefore this inequality still holds if the domain of integration of this integral is changed from $[0,+\infty)$ to $\SigmaEscZero(t)$. 
Besides, observe that, in terms of the ``initial'' potential $V$ and solution $u(r,t)$, the factor of $T_{\bar{x}} \psi$ under the integral of the right-hand side of this last inequality reads
\[
 - (u-m)\cdot \nabla V(u)+\frac{1}{2} \bigl(V(u)-V(m)\bigr)+ \frac{\eigVmin(m)}{4} (u-m)^2
\,.
\]
Thus, if $\KFZero$ denotes the maximum of this expression over all possible values for $m$ and $u$, that is the (positive) quantity
\begin{equation}
\label{def_KFZero}
\KFZero = \max_{v\in\rr^{\dState},\ \abs{v}\le \Rattinfty}\left[ - (u-m)\cdot \nabla V(u)+ \frac{1}{2} \abs{V(u)-V(m)} + \frac{\eigVmin(m)}{4} (u-m)^2\right]
\,,
\end{equation}
then inequality~\cref{dt_fire_0} follows from inequality \cref{dt_fire_spat_as_3} (with the domain of integration of the integral on the right-hand side restricted to $\SigmaEscZero(t)$). Observe that $\KFZero$ depends only on $V$. This finishes the proof of \cref{lem:dt_fire_0}.
\end{proof}
\subsection{Upper bound on the invasion speed}
\label{subsec:upp_bd_inv_vel}
Let
\begin{equation}
\label{def_delta_esc_fire}
\desc(m) = \dEsc(m)\sqrt{\frac{\min\Bigl(\frac{\coeffEnZero}{2},\frac{1}{4}\Bigr)}{\kappa_0 + 1}}
\,.
\end{equation}
As the quantity $\dEsc(m)$ defined in \vref{subsubsec:breakup}, this quantity $\desc(m)$ will provide a way to measure the vicinity of the solution $u$ to the point $m$, this time in terms of the firewall function $\fff_0$. The value chosen for $\desc(m)$ depends only on $V$ and ensures the validity of the following lemma. 
\begin{lemma}[escape/Escape]
\label{lem:esc_Esc}
For all $(\bar{r},t)$ in $[\rSmallCurv,+\infty)\times[0,+\infty)$, the following assertion holds:
\begin{equation}
\label{ineq_esc_Esc}
\fff_0(\bar{r},t) \le \desc(m)^2
\implies
\abs{u^\dag(\bar{r},t)} \le \dEsc(m)
\,.
\end{equation}
\end{lemma}
\begin{proof}
Let $v$ be a function $Y$, and assume in addition that $v$ is of class $\ccc^1$ and that its derivative is uniformly bounded on $[0,+\infty)$. Then, for all $\bar{r}$ in $[\rSmallCurv,+\infty)$,
\[
\begin{aligned}
v(\bar{r})^2 & = T_{\bar{r}}\psi_0(\bar{r}) v(\bar{r})^2 \\
& \le \int_{\bar{r}}^{+\infty} \abs{ \frac{d}{dr}\bigl( T_{\bar{r}}\psi_0(r) v(r)^2 \bigr) } \, dr \\
& \le \int_{\bar{r}}^{+\infty} \bigl( \abs{T_{\bar{r}}\psi_0'(r)} \,  v(r)^2  + 2 T_{\bar{r}}\psi_0(r) \, v(r) \cdot v'(r) \bigr) \, dr \\
& \le \int_{\bar{r}}^{+\infty} T_{\bar{r}}\psi_0(r) \bigl( (\kappa_0 + 1) v(r)^2 + v'(r)^2 \bigr) \, dr \\
& \le (\kappa_0 + 1) \int_0^{+\infty} T_{\bar{r}}\psi_0(r) \bigl( v(r)^2 + v'(r)^2 \bigr) \, dr 
\,.
\end{aligned}
\]
Thus it follows from inequality \cref{coerc_fire_sf} on the coercivity of $\fff_0(\cdot,\cdot)$ that, for all $\bar{r}$ in $[\rSmallCurv,+\infty)$ and $t$ in $[0,+\infty)$, 
\[
u^\dag(\bar{r},t)^2 \le \frac{\kappa_0 + 1}{\min\Bigl(\frac{\coeffEnZero}{2},\frac{1}{4}\Bigr)} \fff_0(\bar{r},t)
\,,
\]
and this ensures the validity of implication \cref{ineq_esc_Esc} with the value of $\desc(m)$ chosen in definition \cref{def_delta_esc_fire}.
\end{proof}
Let $L$ be a positive quantity, large enough so that
\[
2 \KFZero \frac{\exp(-\kappa_0 L)}{\kappa_0} \le \nuFZero \frac{\desc(m)^2}{8}
\,,
\quad\text{namely}\quad
L = \frac{1}{\kappa_0}\log\Bigl(\frac{16 \,\KFZero}{\nuFZero\, \desc(m)^2\, \kappa_0}\Bigr)
\]
(this quantity depends only on $V$), 
let $\hullnoesc:\rr\to\rr\cup\{+\infty\}$ (``no-escape hull'') be the function defined as
\begin{equation}
\label{def_h_noesc}
\hullnoesc(\rho) = 
\left\{
\begin{aligned}
& +\infty & \quad\text{for}\quad & \rho<0 \,, \\
& \frac{\desc(m)^2}{2}\Bigl(1-\frac{\rho}{2\,L}\Bigr) & \quad\text{for}\quad & 0\le \rho\le L \,, \\
& \frac{\desc(m)^2}{4} & \quad\text{for}\quad & \rho\ge L \,,
\end{aligned}
\right.
\end{equation}
\begin{figure}[!htbp]
	\centering
    \includegraphics[width=0.6\textwidth]{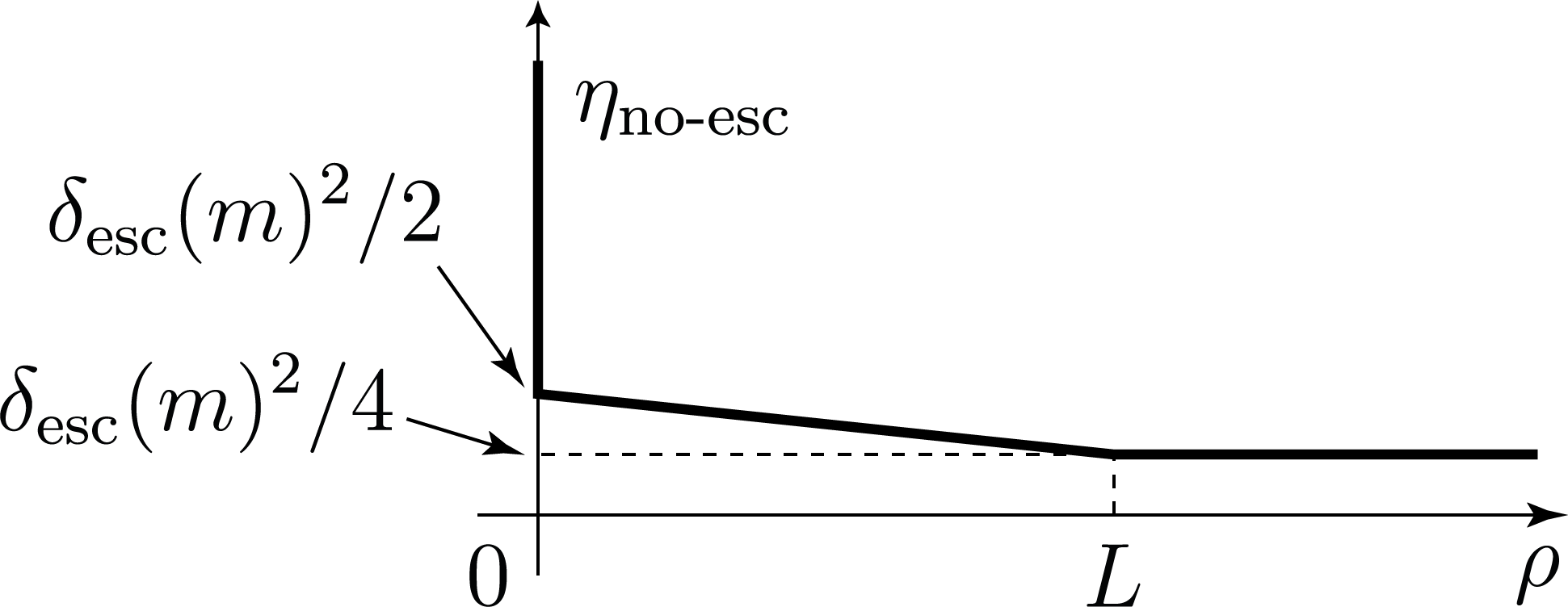}
    \caption{Graph of the hull function $\hullnoesc$.}
    \label{fig:graph_hull}
\end{figure}
see \cref{fig:graph_hull}; and let  $\cnoesc$ (``no-escape speed'') be a positive quantity, large enough so that
\[
\cnoesc \frac{\desc(m)^2}{4L} \ge 2 \frac{\KFZero}{\kappa_0}
\,,
\quad\text{namely}\quad
\cnoesc = \frac{8\,\KFZero \, L}{\kappa_0 \, \desc(m)^2}
\]
(this quantity depends only on $V$ and $m$). 
The following lemma, illustrated by \cref{fig:hull_invasion}, is a variant of \cite[\GlobalRelaxationLemBoundInvasionSpeed]{Risler_globalRelaxation_2016}. 
\begin{figure}[!htbp]
\centering
\includegraphics[width=\textwidth]{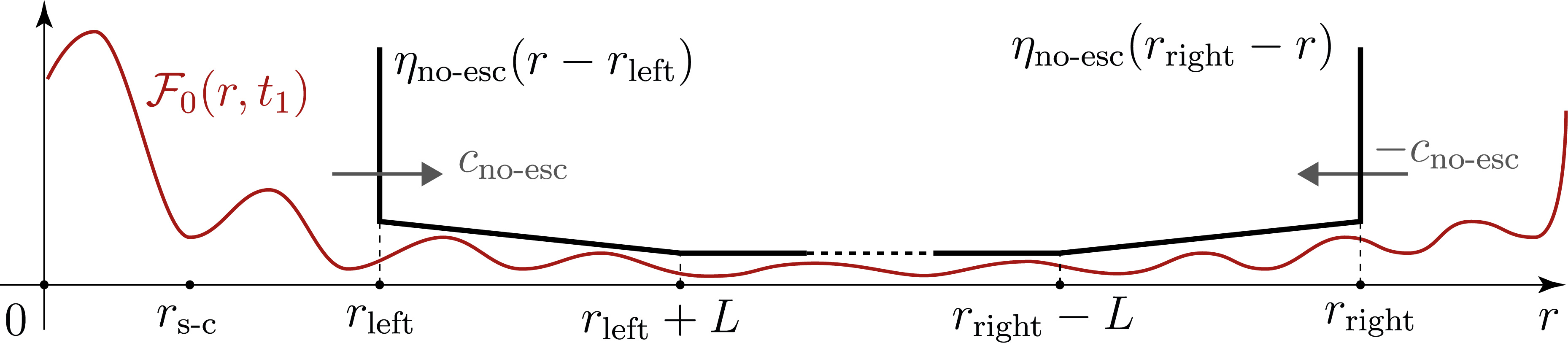}
\caption{Illustration of \cref{lem:inv}; if the firewall function is below the maximum of two mirror hulls at a certain time $t_0$ and if these two hulls travel at opposite speeds $\pm\cnoesc$, then the firewall will remain below the maximum of those travelling hulls in the future (note that after they cross this maximum equals $+\infty$ thus the assertion of being ``below'' is empty).}
\label{fig:hull_invasion}
\end{figure}
\begin{lemma}[bound on invasion speed]
\label{lem:inv}
For every ordered pair $(\rLeft,\rRight)$ of points in the interval $[\rSmallCurv,+\infty)$ and every nonnegative time $t_1$, if
\[
\fff_0(r,t_1) \le \max\bigl( \hullnoesc(r-\rLeft) , \hullnoesc(\rRight - r) \bigr)
\quad\text{for all } r \text{ in }[\rSmallCurv,+\infty) 
\,,
\]
then, for every time $t$ greater than or equal to $t_1$ and all $r$ in $[\rSmallCurv,+\infty)$, 
\[
\fff_0(r,t) \le \max\Bigl( \hullnoesc\bigl(\rLeft-\cnoesc\, (t-t_1)\bigr), \hullnoesc\bigl(\rRight + \cnoesc\, (t-t_1)- r\bigr) \Bigr)
\,.
\]
\end{lemma}
\begin{proof}
See \cite[\GlobalRelaxationLemBoundInvasionSpeed]{Risler_globalRelaxation_2016}. 
\end{proof}
\subsection{Set-up for the proof, 2: escape point and associated speeds}
\label{subsec:inv_cv_set_pf_cont}
According to hypothesis \textup{(\hyperlink{hypHom}{\hypHomRef})} and to the bounds \cref{bound_u_ut_ck_bis} on the solution, it may be assumed, up to changing the origin of time, that, for all $t$ in $[0,+\infty)$ and for all $r$ in $[\rSmallCurv,+\infty)$,
\begin{equation}
\label{hyp_for_def_resc}
\begin{aligned}
\rSmallCurv &\le \rHom(t) - 1 \,, \\
\text{and}\quad
\fff_0(r,t) &\le \max\biggl( \hullnoesc\Bigl( r-\bigl( \rHom(t)-1\bigr) \Bigr) , \hullnoesc \bigl( \rHom(t) - r \bigr) \biggr)
\,.
\end{aligned}
\end{equation}
As a consequence, for all $t$ in $[0,+\infty)$, the set
\[
\begin{aligned}
\IHom(t) = \Bigl\{ & r_{\ell} \in [\rSmallCurv , \rHom(t)] : \text{ for all } r \text{ in } [\rSmallCurv , +\infty) \,, 
\\
& \fff_0(r,t) \le \max\Bigl( \hullnoesc( r-r_{\ell}) , \hullnoesc \bigl( \rHom(t) - r \bigr) \Bigr)
\Bigr\}
\end{aligned}
\]
is a nonempty interval (containing $[\rHom(t)-1,\rHom(t)]$), see \cref{fig:def_resc}. 
\begin{figure}[!htbp]
	\centering
    \includegraphics[width=\textwidth]{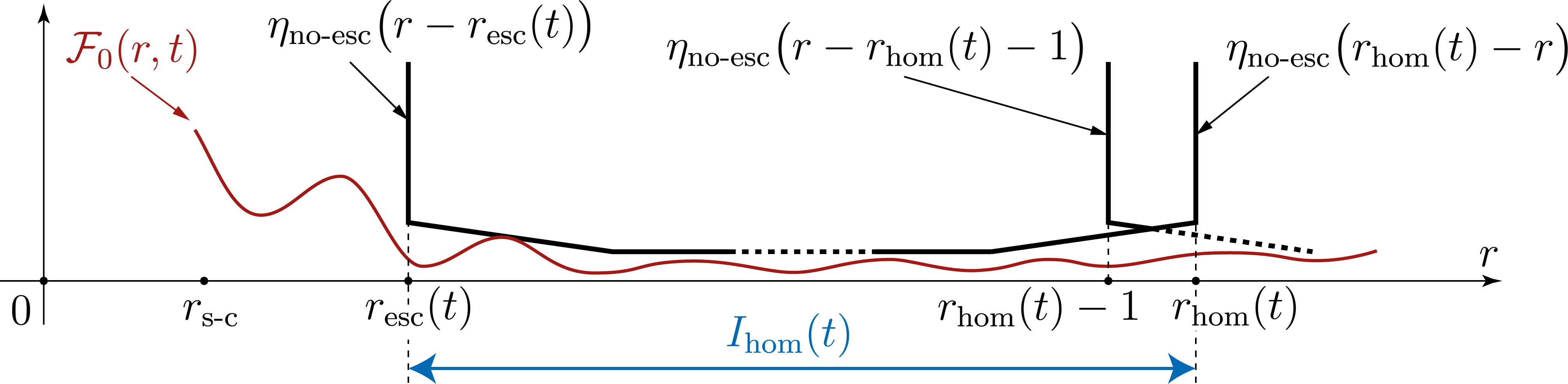}
    \caption{Interval $\IHom(t)$ and definition of $\resc(t)$.}
    \label{fig:def_resc}
\end{figure}
For all $t$ in $[0,+\infty)$, let 
\begin{equation}
\label{def_resc}
\resc(t) = \inf \bigl( \IHom(t) \bigr) 
\quad
\text{(thus } \resc(t)\in[\rSmallCurv , \rHom(t)-1] \text{ ).}
\end{equation}
Somehow like $\rEsc(t)$, this point represents the first point at the left of $\rHom(t)$ where the solution $u$ (respectively $u^\dag$) ``escapes'' (in a sense defined by the firewall function $\fff_0$ and the no-escape hull $\hullnoesc$) at a certain distance from $m$ (respectively from $0_{\rr^{\dState}}$) --- except if $\IHom(t)$ is the whole interval $[\rSmallCurv,\rHom(t)]$, in this case this ``escape'' does not occur. In the following, this point $\resc(t)$ will be called the ``escape point'' (by contrast with the ``Escape point'' $\rEsc(t)$ defined before). According to the ``hull inequality'' \cref{hyp_for_def_resc} and \cref{lem:esc_Esc} (``escape/Escape''), for all $t$ in $[0,+\infty)$,
\begin{equation}
\label{rEsc_resc_rHom}
\rEsc(t) \le \resc(t) \le \rHom(t)-1 
\quad\text{and}\quad
\SigmaEscZero(t) \cap [\rEsc(t),\rHom(t)] = \emptyset
\,,
\end{equation}
and, according to hypothesis \textup{(\hyperlink{hypHom}{\hypHomRef})} and to the bounds \cref{bound_u_ut_ck_bis} on the solution, 
\begin{equation}
\label{rHom_minus_resc}
\rHom(t) - \resc(t) \to +\infty
\quad\text{as}\quad
t\to +\infty
\,.
\end{equation}
The big advantage of $\resc(\cdot)$ with respect to $\rEsc(\cdot)$ is that, according to \cref{lem:inv} (``bound on invasion speed''), the growth of $\resc(\cdot)$ is more under control. More precisely, according to this lemma, for all nonnegative quantities $t$ and $s$, 
\begin{equation}
\label{control_escape}
\resc(t+s)\le \resc(t) + \cnoesc \, s
\,.
\end{equation}
For every $s$ in $[0,+\infty)$, let us consider the ``upper and lower bounds of the variations of $\resc(\cdot)$ over all time intervals of length $s$'':
\begin{figure}[!htbp]
	\centering
    \includegraphics[width=\textwidth]{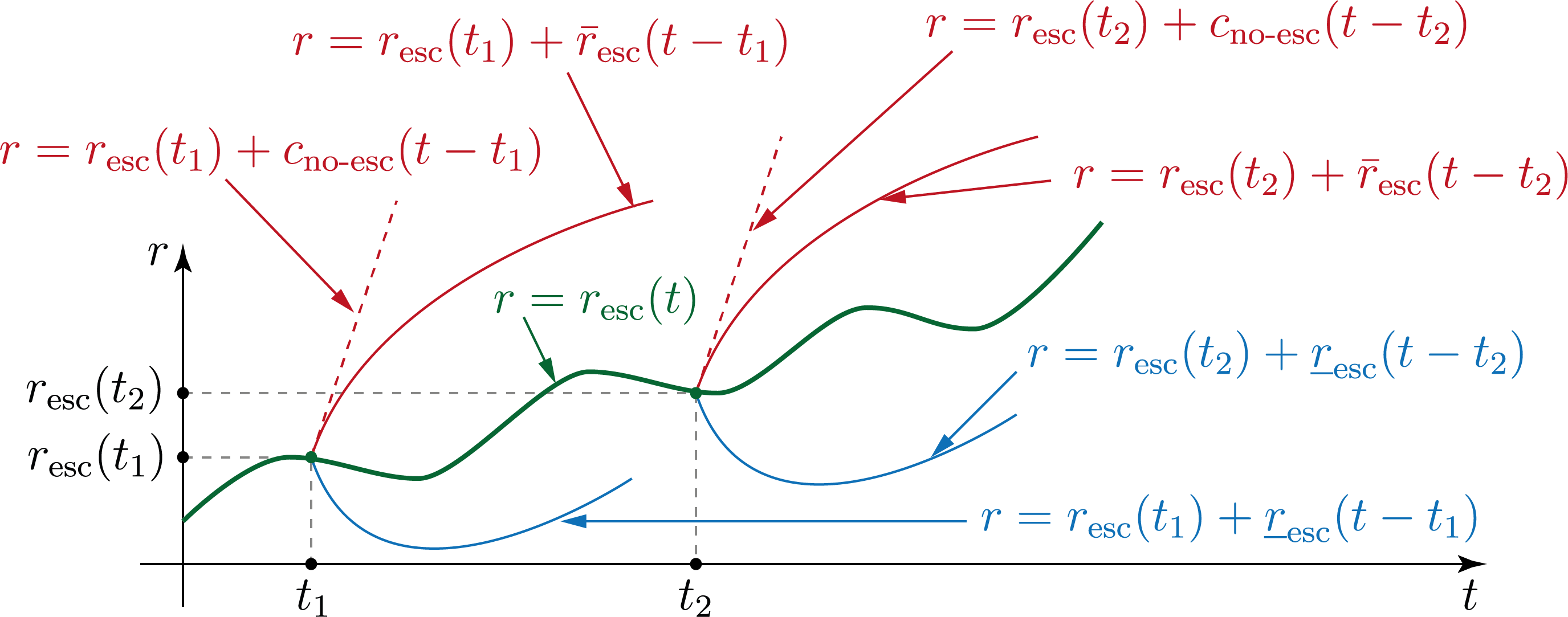}
    \caption{Illustration of the bounds \cref{various_bounds_r_esc}.}
    \label{fig:def_bar_underbar_resc}
\end{figure}
\[
\barresc(s) = \sup_{t\in[0,+\infty)} \resc(t+s) - \resc(t)
\quad\text{and}\quad
\underbarresc(s) = \inf_{t\in[0,+\infty)} \resc(t+s) - \resc(t)
\,,
\]
see \cref{fig:def_bar_underbar_resc}. According to these definitions and to inequality \cref{control_escape} above, for all $t$ and $s$ in $[0,+\infty)$,
\begin{equation}
\label{various_bounds_r_esc}
-\infty\le\underbarresc(s)\le \resc(t+s)-\resc(t)\le \barresc(s) \le \cnoesc\, s
\,.
\end{equation}
Let us consider the four limit mean speeds:
\[
\cescinf = \liminf_{t\to+\infty}\frac{\resc(t)}{t}
\quad\text{and}\quad
\cescsup = \limsup_{t\to+\infty}\frac{\resc(t)}{t}
\]
and
\[
\underbarcescinf = \liminf_{s\to+\infty}\frac{\underbarresc(s)}{s}
\quad\text{and}\quad
\barcescsup = \limsup_{s\to+\infty}\frac{\barresc(s)}{s}
\,.
\]
The following inequalities follow from these definitions and from hypothesis \textup{(\hyperlink{hypInv}{\hypInvRef})}:
\[
-\infty \le \underbarcescinf \le \cescinf \le \cescsup \le \barcescsup \le \cnoesc
\quad\text{and}\quad
0 < \cEsc \le \cescsup
\,.
\]
The four limit mean speeds defined just above will turn out to be equal. The proof of this equality is based of the ``relaxation scheme'' set up in the next \namecref{subsec:relax_sch_tr_fr}. 
\subsection{Relaxation scheme in a travelling frame}
\label{subsec:relax_sch_tr_fr}
The aim of this \namecref{subsec:relax_sch_tr_fr} is to set up an appropriate relaxation scheme in a travelling frame. This means defining an appropriate localized energy and controlling the ``flux'' terms occurring in the time derivative of this localized energy. The considerations made in \vref{subsec:1rst_ord} will be put in practice. 
\subsubsection{Preliminary definitions}
\label{subsubsec:def_trav_f}
Let us introduce the following real quantities that will play the role of ``parameters'' for the relaxation scheme below (see \cref{fig:trav_fr}):
\begin{figure}[!htbp]
\centering
\includegraphics[width=.8\textwidth]{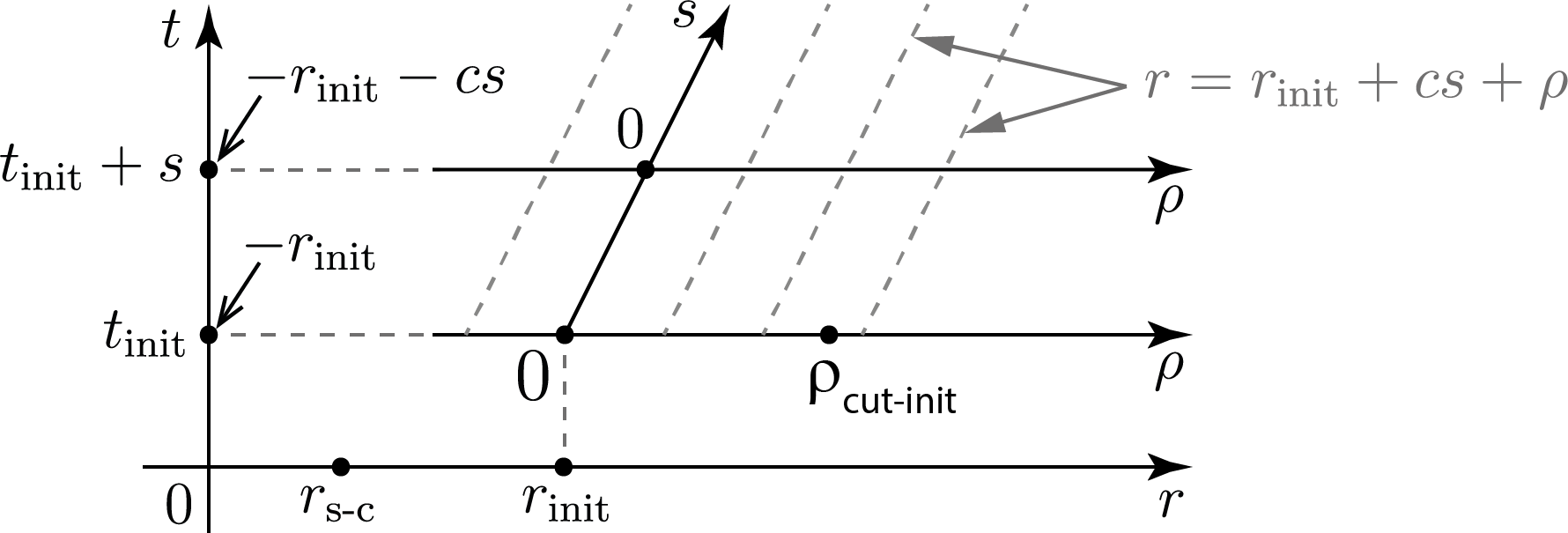}
\caption{Space coordinate $\rho$ and time coordinate $s$ in the travelling frame, and parameters $\tInit$ and $\rInit$ and $c$ and $\initCut$.}
\label{fig:trav_fr}
\end{figure}
\begin{itemize}
\item the ``initial time'' $\tInit$ of the time interval of the relaxation;
\item the position $\rInit$ of the origin of the travelling frame at initial time $t=\tInit$ (in practice it will be chosen equal to $\resc(\tInit)$);
\item the speed $c$ of the travelling frame;
\item a quantity $\initCut$ that will be the the position of the maximum point of the weight function $\rho\mapsto\chi(\rho,\tInit)$ localizing energy at initial time $t=\tInit$ (this weight function is defined below); the subscript ``cut'' refers to the fact that this weight function displays a kind of ``cut-off'' on the interval between this maximum point and $+\infty$. Thus the maximum point is in some sense the point ``where the cut-off begins''. 
\end{itemize}
Let us make on these parameters the following hypotheses:
\begin{equation}
\label{hyp_param_relax_sch}
0\le \tInit
\quad\text{and}\quad
0 < c \le \cnoesc
\quad\text{and}\quad
0 \le \initCut
\quad\text{and}\quad
\rInit \ge \rSmallCurv
\,.
\end{equation}
For all $\rho$ in $[-\rInit - cs,+\infty)$ and $s$ in $[0,+\infty)$, let
\[
v(\rho,s) = u^\dag(r,t)
\quad\text{where}\quad
r = \rInit + cs + \rho
\quad\text{and}\quad
t = \tInit + s
\,.
\]
This function satisfies the differential system
\begin{equation}
\label{syst_rad_sym_tf}
v_s - c v_\rho = -\nabla V^\dag(v) + \frac{d-1}{\rInit + cs + \rho} v_\rho + v_{\rho\rho}
\,.
\end{equation}
Let $\kappa$ (rate of decrease of the weight functions), $\cCut$ (speed of the cutoff point in the travelling frame), and $\coeffEn$ (coefficient of energy in the ``firewall'' function) be three positive quantities, small enough so that
\begin{equation}
\label{constraints_quantities_relax_sch}
\begin{aligned}
\coeffEn\Bigl(\frac{\cCut(c+\kappa)}{2} + \frac{(c+\kappa+\frac{1}{2})^2}{4}\Bigr) &\le \frac{1}{4}
\quad\text{and}\quad
\coeffEn \cCut (c+\kappa)  \le \frac{1}{4}\\
\text{and}\quad
\frac{(\cCut+\kappa)(c+\kappa)}{2} &\le \frac{\eigVmin(m)}{16}
\end{aligned}
\end{equation}
(these conditions will be used to prove inequality \vref{ds_fire_2}), and so that
\begin{equation}
\label{coeffEn_smaller_than_coeffEnZero}
\coeffEn \le \coeffEnZero
\,.
\end{equation}
These quantities may be chosen as follows (first choose $\kappa$ and $\cCut$ so that the third inequality of \cref{constraints_quantities_relax_sch} be fulfilled, and then choose $\coeffEn$ according to the first two inequalities of \cref{constraints_quantities_relax_sch} and to \cref{coeffEn_smaller_than_coeffEnZero}):
\[
\begin{aligned}
\kappa &= \min\Bigl( \sqrt{\frac{\eigVmin(m)}{32}} , \frac{\eigVmin(m)}{32\cnoesc} \Bigr) 
\quad\text{and}\quad
\cCut = \frac{\eigVmin(m)}{16(\cnoesc+\kappa)} \,, \\
\text{and}\quad\coeffEn &= \min\Bigl(\frac{1}{2\cCut(\cnoesc + \kappa) + (\cnoesc + \kappa+1/2)^2} , \frac{1}{4\cCut(\cnoesc + \kappa)},  \coeffEnZero\Bigr) 
\,.
\end{aligned}
\]
Conditions \cref{constraints_quantities_relax_sch,coeffEn_smaller_than_coeffEnZero} are very similar to those stated in \cite{Risler_globalBehaviour_2016}, although slightly more stringent due to the curvature terms. 
\subsubsection{Localized energy}
\label{subsubsec:def_loc_en}
For every nonnegative quantity $s$, let us introduce the intervals
\[
\begin{aligned}
\iLeft(s) &= \bigl[- \rInit - cs , - \rInit - cs + \rSmallCurv \bigr] \,, \\
\text{and}\qquad
\iMain(s) &= \bigl[ - \rInit - cs + \rSmallCurv , \initCut + \cCut s\bigr] \,, \\
\text{and}\qquad
\iRight(s) &= [ \initCut + \cCut s , +\infty) \,, \\
\text{and}\qquad
\iTot(s) &= \bigl[- \rInit - cs , +\infty\bigr) = \iLeft(s) \cup \iMain(s) \cup \iRight(s) 
\,,
\end{aligned}
\]
see \cref{fig:chi_psi}.
\begin{figure}[!htbp]
	\centering
    \includegraphics[width=\textwidth]{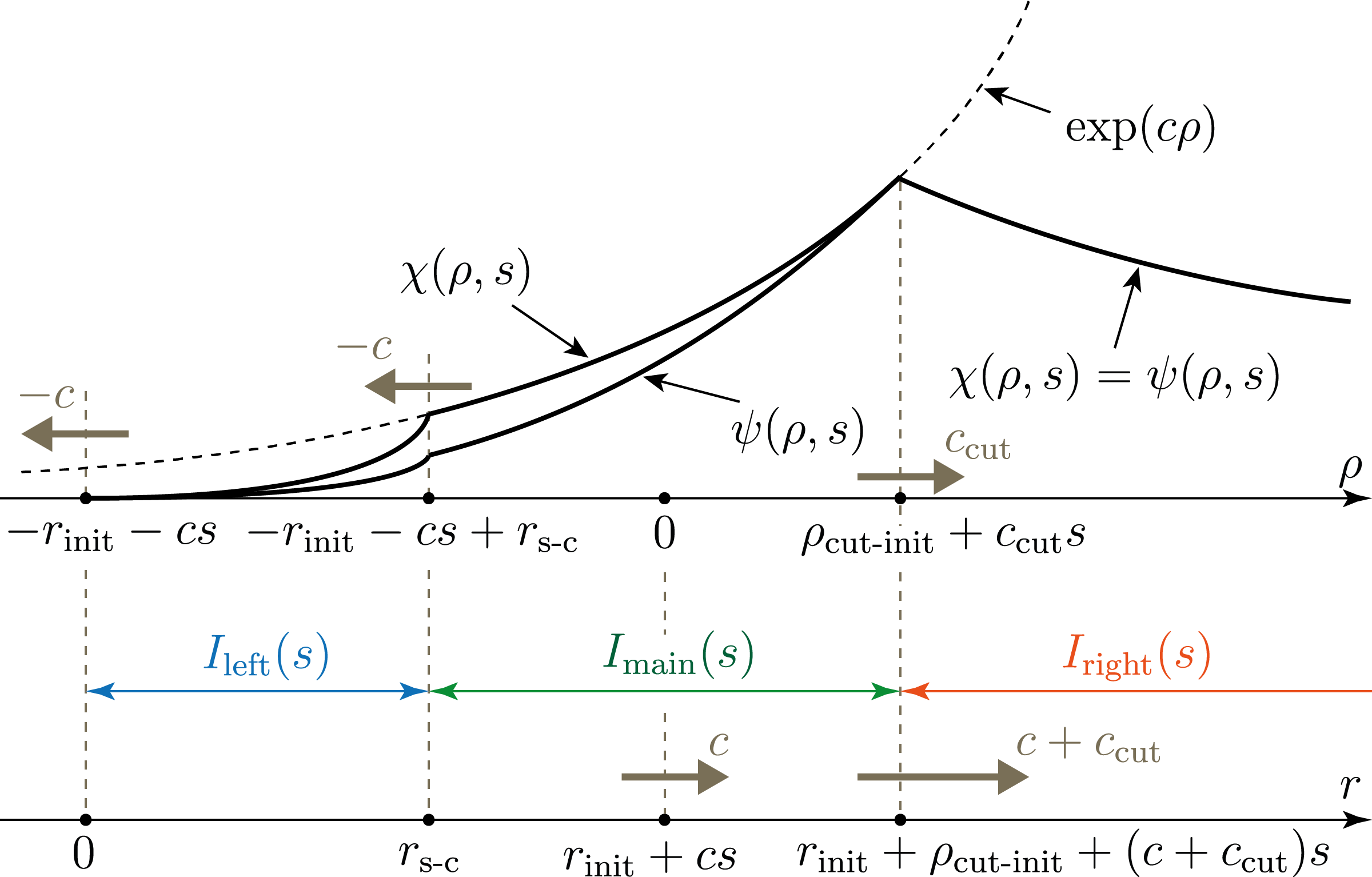}
    \caption{Intervals $\iLeft(s)$ and $\iMain(s)$ and $\iRight(s)$ and graphs of the weight functions $\chi(y,s)$ and $\psi(y,s)$.}
    \label{fig:chi_psi}
\end{figure}
Observe that, since according to hypotheses \cref{hyp_param_relax_sch} the quantity $\rInit$ is greater than or equal to $\rSmallCurv$, the interval $\iMain(s)$ is nonempty. 
Let us introduce the function $\chi(\rho,s)$ (weight function for the localized energy) defined as
\[
\chi(\rho,s) =
\left\{
\begin{aligned}
& \exp(c\rho)  \Bigl( \frac{ \rInit + cs + \rho }{\rSmallCurv} \Bigr)^{d-1} 
& & \text{if}\quad\rho \in \iLeft(s) \,, \\
& \exp(c\rho) 
& & \text{if}\quad\rho \in \iMain(s) \,, \\
& \exp\Bigl[c(\initCut + \cCut s) - \kappa\bigl(\rho-(\initCut + \cCut s)\bigr)\Bigr]
& & \text{if}\quad\rho \in \iRight(s) \,. \\
\end{aligned}
\right.
\]
For all $s$ in $[0,+\infty)$, let us define the ``energy function'' $\eee(s)$ by
\[
\eee(s) = \int_{\iTot(s)} \chi(\rho,s)\Bigl( \frac{1}{2}v_\rho(\rho,s)^2 + V^\dag\bigl(v(\rho,s)\bigr) \Bigr) \, d\rho
\,.
\]
\subsubsection{Time derivative of the localized energy}
\label{subsubsec:der_loc_en}
For every nonnegative quantity $s$, let 
\[
\ddd(s) = \int_{\iTot(s)} \chi(\rho,s)\, v_s(\rho,s)^2 \, d\rho
\,.
\]
The aim of this \namecref{subsubsec:der_loc_en} is to prove the following lemma. 
\begin{lemma}[upper bound on time derivative of energy, first version]
\label{lem:ds_en_before_fire_last}
There exist positive quantities $\KELeft$ and $\KEMain$, depending only on $V$ and $d$, such that, for every nonnegative quantity $s$, the following inequality holds:
\begin{equation}
\label{ds_en_before_fire_last}
\begin{aligned}
\eee'(s) \le & -\frac{1}{2}\ddd(s) + \KELeft \exp \bigl( - c \rInit \bigr) \\
& + \KEMain \biggl( \exp\Bigl(-\frac{c\rInit}{2} \Bigr)+ \frac{2}{\rInit} \exp\bigl( c(\initCut + \cCut s)\bigr)\biggr) \\
& + \int_{\iRight(s)} \chi \biggl( \cCut (c+\kappa) \Bigl( \frac{1}{2}v_\rho^2 + V^\dag(v) \Bigr) + \frac{(1/2 + c + \kappa)^2}{2}v_\rho^2  \biggr) \, d\rho 
\,.
\end{aligned}
\end{equation}
\end{lemma}
\begin{proof}
\renewcommand{\qedsymbol}{} 
It follows from expression \vref{ddt_loc_en_trav_fr} for the derivative of a localized energy that
\[
\begin{aligned}
\eee'(s) = & - \ddd(s) + \int_{\iTot(s)} \chi_s \Bigl( \frac{1}{2}v_\rho^2 + V^\dag(v) \Bigr) \, d\rho \\
& + \int_{\iTot(s)} \Bigl( \frac{d-1}{\rInit + cs + \rho} \chi+c\chi-\chi_\rho \Bigr) v_s \cdot v_\rho \, d\rho 
\,.
\end{aligned}
\]
It follows from the definition of $\chi$ that, for every real quantity $\rho$, 
\[
\chi_s(\rho,s) =
\left\{
\begin{aligned}
& \frac{c(d-1)}{\rInit + cs + \rho} \chi(\rho,s)
& & \text{if}\quad\rho \in \iLeft(s) \,, \\
& 0
& & \text{if}\quad\rho \in \iMain(s) \,, \\
& \cCut (c+\kappa) \chi(\rho,s)
& & \text{if}\quad\rho \in \iRight(s) \,, \\
\end{aligned}
\right.
\]
and
\[
\chi_\rho(\rho,s) = \left\{
\begin{aligned}
& c \chi(\rho,s) + \frac{d-1}{\rInit + cs + \rho} \chi(\rho,s)
& & \text{if}\quad\rho \in \iLeft(s) \,, \\
& c \chi(\rho,s)
& & \text{if}\quad\rho \in \iMain(s) \,, \\
& -\kappa \chi(\rho,s) 
& & \text{if}\quad\rho \in \iRight(s) \,. \\
\end{aligned}
\right.
\]
thus
\begin{equation}
\label{pollution_factor_chi}
\begin{aligned}
\frac{d-1}{\rInit + cs + \rho} \chi(\rho,s) + c\chi(\rho,s) & -\chi_\rho(\rho,s)  = \\
& \left\{ \begin{aligned}
& 0 
& & \text{if}\quad\rho \in \iLeft(s) \,, \\
& \frac{d-1}{\rInit + cs + \rho} \chi(\rho,s)
& & \text{if}\quad\rho \in \iMain(s) \,, \\
& \Bigl(c+\kappa + \frac{d-1}{\rInit + cs + \rho} \Bigr) \chi(\rho,s) 
& & \text{if}\quad\rho \in \iRight(s) \,. \\
\end{aligned}
\right.
\end{aligned}
\end{equation}
As a consequence,
\[
\begin{aligned}
\eee'(s) & = - \ddd(s) \\
& + \int_{\iLeft(s)} \chi \frac{c(d-1)}{\rInit + cs + \rho} \Bigl( \frac{1}{2}v_\rho^2 + V^\dag(v)\Bigr) \, d\rho \\
& + \int_{\iMain(s)} \chi \frac{d-1}{\rInit + cs + \rho} v_s \cdot v_\rho \, d\rho \\
& + \int_{\iRight(s)} \chi \biggl( \cCut (c+\kappa) \Bigl( \frac{1}{2}v_\rho^2 + V^\dag(v) \Bigr) + \Bigl( \frac{d-1}{\rInit + cs + \rho} + c + \kappa \Bigr) v_s \cdot v_\rho \biggr) \, d\rho 
\,. 
\end{aligned}
\]
Polarizing the scalar products $v_s \cdot v_\rho$, it follows that
\begin{equation}
\label{ds_en_before_fire}
\begin{aligned}
\eee'(s) \le & - \frac{1}{2}\ddd(s) \\
& + \int_{\iLeft(s)} \chi \frac{c(d-1)}{\rInit + cs + \rho}  \Bigl( \frac{1}{2}v_\rho^2 + V^\dag(v)\Bigr)  \, d\rho \\
& + \int_{\iMain(s)} \chi \frac{1}{2}\Bigl(\frac{d-1}{\rInit + cs + \rho}\Bigr)^2 v_\rho^2  \, d\rho \\
& + \int_{\iRight(s)} \chi \biggl( \cCut (c+\kappa) \Bigl( \frac{1}{2}v_\rho^2 + V^\dag(v) \Bigr) + \frac{1}{2}\Bigl( \frac{d-1}{\rInit + cs + \rho} + c + \kappa \Bigr)^2 v_\rho^2  \biggr) \, d\rho 
\,.
\end{aligned}
\end{equation}
Let us make a brief comment on this inequality, in comparison with the (simpler) case $d=1$ (see \cite[\GlobalBehaviourSubSubSecTimeDerivativeLocEn]{Risler_globalBehaviour_2016}). 

Observe that the last term of this inequality (the integral over $\iRight(s)$) is very similar to the $d=1$ case. As in the $d=1$ case, its control will require the definition of a ``firewall function'' that will be defined in the next \cref{subsubsec:def_firewall}. Thus the main novelty with respect to the $d=1$ case is the existence of the two other integrals over $\iLeft(s)$ and $\iMain(s)$ (according to the calculations above, the integral over $\iLeft(s)$ follows from the fact that $\chi_s(\rho,s)$ is positive when $\rho$ belongs to this interval, and the integral over $\iMain(s)$ comes from the curvature term in system \cref{syst_rad_sym_tf}). 

Unfortunately, the firewall function that will be defined in the next \namecref{subsubsec:def_firewall} will be of no help to control these two terms, since the weight function $\psi(\rho,s)$ involved in its definition will have to be chosen much smaller than $\chi(\rho,s)$ on both intervals $\iLeft(s)$ and $\iMain(s)$. As a consequence, these two terms need to be treated separately. The aim of the following two lemmas is to do this job, that is to provide appropriate upper bounds for these two terms (proof \cref{lem:ds_en_before_fire_last} will follow afterwards). The sole required feature of these bounds is that they should be small if the positive quantity $\rInit$ is large. 
\end{proof} 
\begin{lemma}[upper bound for curvature term on $\iLeft(s)$]
\label{lem:bound_ds_en_iLeft}
There exists a positive quantity $\KELeft$, depending only on $V$ and $d$, such that, for every nonnegative quantity $s$, the following estimate holds: 
\begin{equation}
\label{bound_ds_en_iLeft}
\int_{\iLeft(s)} \chi \frac{c(d-1)}{\rInit + cs + \rho}  \Bigl( \frac{1}{2}v_\rho^2 + V^\dag(v)\Bigr)  \, d\rho 
\le 
\KELeft \exp \bigl( - c \rInit \bigr)
\,.
\end{equation}
\end{lemma}
\begin{proof}[Proof of \cref{lem:bound_ds_en_iLeft}]
For every nonnegative quantity $s$ and every $\rho$ in $\iLeft(s)$, 
\[
\begin{aligned}
\chi(\rho,s) \frac{c(d-1)}{\rInit + cs + \rho} & = \frac{c(d-1)\exp(c\rho)}{\rSmallCurv} \Bigl(\frac{\rInit + cs + \rho}{\rSmallCurv}\Bigr)^{d-2} \\
&\le \frac{c(d-1)\exp(c\rho)}{\rSmallCurv}
\end{aligned}
\]
(this inequality still holds if $d=1$, however recall that for clarity the $d$ equals $1$ case was excluded, thus $d$ is assumed to be not smaller than $2$). Thus,
\[
\begin{aligned}
\int_{\iLeft(s)} \chi \frac{c(d-1)}{\rInit + cs + \rho} \, d\rho & \le \frac{d-1}{\rSmallCurv} \exp\bigl( c(-\rInit - cs + \rSmallCurv\bigr) \\
& \le \Bigl( \frac{d-1}{\rSmallCurv} \exp(c\rSmallCurv)\Bigr) \exp(-c\rInit)
\,,
\end{aligned}
\]
thus inequality \cref{bound_ds_en_iLeft} follows from the bound \cref{hyp_param_relax_sch} on the speed $c$ and the bounds \vref{hyp_attr_ball} for the solution.
\Cref{lem:bound_ds_en_iLeft} is proved.
\end{proof}
Let us make the following additional hypothesis on the parameter $\rInit$:
\begin{equation}
\label{hyp_rInit_larger_twice_rsc}
\rInit \ge 2 \rSmallCurv
\,.
\end{equation}
\begin{lemma}[upper bound for curvature term on $\iMain(s)$]
\label{lem:bound_ds_en_iMain}
There exists a positive quantity $\KEMain$, depending only on $V$ and $d$, such that, for every nonnegative quantity $s$, the following estimate holds: 
\begin{equation}
\label{bound_ds_en_iMain}
\begin{aligned}
\int_{\iMain(s)} \chi \frac{1}{2} \Bigl(\frac{d-1}{\rInit + cs + \rho}\Bigr)^2 & v_\rho^2 \, d\rho
\le \\
\KEMain & \biggl( \exp\Bigl(-\frac{c\rInit}{2} \Bigr)+ \frac{2}{\rInit} \exp\bigl( c(\initCut + \cCut s)\bigr)\biggr)
\,.
\end{aligned}
\end{equation}
\end{lemma}
\begin{proof}[Proof of \cref{lem:bound_ds_en_iMain}]
Let us introduce the integral:
\[
\begin{aligned}
J &= \int_{\iMain(s)}\frac{\chi}{(\rInit + cs + \rho)^2} \, d\rho \\
& = \int_{-\rInit - cs + \rSmallCurv}^{\initCut + \cCut s} \frac{\exp(c \, \rho)}{(\rInit + cs + \rho)^2} \, d\rho \\
& = \exp \bigl( -c \rInit - c^2s \bigr) \int_{\rSmallCurv}^{\rInit + \initCut + (c+\cCut) s} \frac{\exp(c \, r)}{r^2} \, dr
\,.
\end{aligned}
\]
To bound from above this expression, the integral may be cut into two pieces, namely:
\[
J = \exp \bigl( -c \rInit - c^2s \bigr) \biggl( \int_{\rSmallCurv}^{\rInit/2} \frac{\exp(c \, r)}{r^2} \, dr + \int_{\rInit/2}^{\rInit + \initCut + (c+\cCut) s} \frac{\exp(c \, r)}{r^2} \, dr \biggr) 
\,;
\]
observe that according to hypothesis \cref{hyp_rInit_larger_twice_rsc} the quantity $\rInit/2$ is not smaller than $\rSmallCurv$. Thus, bounding from above the two quantities $\exp(cr)$ in this expression (by replacing the quantity $r$ by the upper bound of the respective integration domain), it follows that
\[
J  \le \frac{\exp\Bigl(-c\rInit/2 - c^2 s\Bigr)}{\rSmallCurv} + \frac{2}{\rInit} \exp\bigl( c(\initCut + \cCut s)\bigr)
\,,
\]
and since according to its definition \vref{def_r_s_c} the quantity $\rSmallCurv$ is not smaller than $1$, it follows that
\[
J \le \exp\Bigl(-\frac{c\rInit}{2} \Bigr)+ \frac{2}{\rInit} \exp\bigl( c(\initCut + \cCut s)\bigr)
\,.
\]
Thus inequality \cref{bound_ds_en_iMain} follows from the bounds \vref{hyp_attr_ball} for the solution. \Cref{lem:bound_ds_en_iMain} is proved.
\end{proof}
\begin{proof}[End of the proof of \cref{lem:ds_en_before_fire_last}]
Observe that, according to the definition \vref{def_r_s_c} of $\rSmallCurv$, the quantity $(d-1)/(\rInit + cs + \rho)$ is not larger than $1/2$ as soon as $\rho$ is in $\iRight(s)$ (actually even in $\iMain(s)$). Thus inequality \cref{ds_en_before_fire_last} of \cref{lem:ds_en_before_fire_last} follows from inequality \cref{ds_en_before_fire} and from \cref{lem:bound_ds_en_iLeft,lem:bound_ds_en_iMain}. \Cref{lem:ds_en_before_fire_last} is proved. 
\end{proof}
\subsubsection{Firewall function} 
\label{subsubsec:def_firewall}
A second function (the ``firewall'') will now be defined, to get some control over the last term of the right-hand side of inequality \cref{ds_en_before_fire}.
Let us introduce the function $\psi(y,s)$ (weight function for the firewall function) defined as follows (for every nonnegative quantity $s$ and every quantity $\rho$ in $\iTot(s)$):
\[
\psi(\rho,s) = \left\{
\begin{aligned}
& \exp\Bigl[\kappa\bigl(\rho - ( \initCut + \cCut s)\bigr)  \Bigr]\chi(\rho,s) && \text{if} \quad \rho \in \iLeft(s)\cup \iMain(s) \,, \\
& \chi(\rho,s) && \text{if} \quad \rho \in \iRight(s)\,,
\end{aligned}
\right.
\]
see \cref{fig:chi_psi}; and, for every nonnegative quantity $s$, let us define the ``firewall'' function by
\begin{equation}
\label{def_fff}
\fff(s) = \int_{\iTot(s)} \psi(\rho,s)\biggl( \coeffEn \Bigl( \frac{1}{2}v_\rho(\rho,s)^2 + V^\dag\bigl(v(\rho,s)\bigr) \Bigr) + \frac{1}{2}v(\rho,s)^2 \biggr) \, d\rho
\,.
\end{equation}
\subsubsection{Energy decrease up to firewall}
\label{subsubsec:loc_energy_decrease}
\begin{lemma}[energy decrease up to firewall]
\label{lem:ds_en_trav_f}
There exists a positive quantity $\KERight$, depending only on $V$, such that for every nonnegative quantity $s$, 
\begin{equation}
\label{ds_en_trav_f}
\begin{aligned}
\eee'(s) & \le -\frac{1}{2}\ddd(s) + \KELeft \exp \bigl( - c \rInit \bigr) \\
& + \KEMain \biggl( \exp\Bigl(-\frac{c\rInit}{2} \Bigr)+ \frac{2}{\rInit} \exp\bigl( c(\initCut + \cCut s)\bigr)\biggr) + \KERight \fff(s)
\,.
\end{aligned}
\end{equation}
\end{lemma}
\begin{proof}
Let us introduce the following positive quantity (depending only on $V$): 
\[
\KERight = \frac{\cCut(\cnoesc+\kappa) + (1/2 + \cnoesc + \kappa)^2}{\coeffEn}
\,.
\]
Inequality \cref{ds_en_trav_f} follows \cref{{lem:ds_en_before_fire_last}} (upper bound \cref{ds_en_before_fire_last} on $\eee'(s)$). For a detailed justification, see \cite[\GlobalBehaviourSubSubSecDefinitionFirewall]{Risler_globalBehaviour_2016}. 
\end{proof}
\subsubsection{Relaxation scheme inequality, 1}
Let $\sFin$ be a nonnegative quantity (denoting the length of the time interval on which the relaxation scheme will be applied), and let us introduce the expression:
\[
\begin{aligned}
\KECurv(r,s,c) = &
\KELeft \, s \exp \bigl( - c r \bigr) \\
& + \KEMain \, s \biggl( \exp\Bigl(-\frac{c r}{2} \Bigr)+ \frac{2}{\rInit} \exp\bigl( c(\initCut + \cCut s)\bigr)\biggr)
\,.
\end{aligned}
\]
It follows from the previous inequality that
\begin{equation}
\label{relax_sch_int_fff}
\frac{1}{2} \int_0^{\sFin} \ddd(s) \, ds  \le  \ \eee(0) - \eee(\sFin) + \KECurv(\rInit,\sFin,c) + \KERight \int_0^{\sFin}\fff(s) \, ds
\,.
\end{equation}
This ``relaxation scheme inequality'' is the core of the arguments carried out through this \cref{sec:inv_impl_cv} to prove \cref{prop:inv_cv}. The crucial property of the ``curvature term'' $\KECurv(r,s,c)$ is that this quantity goes to $0$ as $r$ goes to $+\infty$, uniformly with respect to $s$ bounded and $c$ bounded away from $0$ and $+\infty$. The next goal is to gain some control over the firewall function (and as a consequence over the last term of this inequality). 
\subsubsection{Firewall linear decrease up to pollution}
\label{subsubsec:der_fire}
For every nonnegative quantity $s$, let us introduce the set (the domain of space where $v(\cdot,s)$ ``Escapes'' at distance $\dEsc(m)$ from $0_{\rr^{\dState}}$) defined as
\[
\SigmaEsc(s)=\{\rho\in \iMain(s) \cup \iRight(s) : \abs{v(\rho,s)} > \dEsc(m) \} 
\,.
\]
To make the connection with definition \vref{def_Sigma_Esc_0} of the related set $\SigmaEscZero(t)$, observe that, for every quantity $\rho$ in $\iTot(s)$, 
\[
\rho \in \SigmaEsc(s) \iff 
\rInit + cs + \rho \in \SigmaEscZero(\tInit+s)
\,.
\]
The next step is the following lemma (observe the strong similarity with \vref{lem:dt_fire_0}).
\begin{lemma}[firewall linear decrease up to pollution]
\label{lem:decrease_firewall}
There exist positive quantities $\nuF$ and $\KF$ and $\KFLeft$ such that, for every nonnegative quantity $s$,
\begin{equation}
\label{ds_fire}
\fff'(s) \le - \nuF \fff(s) + \KF \int_{\SigmaEsc(s)} \psi(\rho,s) \, d\rho + \KFLeft \, \exp(-c \, \rInit)
\,.
\end{equation}
The quantities $\nuF$ and $\KF$ depend only on $V$ and $m$, whereas $\KFLeft$ depends additionally on $d$. 
\end{lemma}
\begin{proof}
\renewcommand{\qedsymbol}{} 
According to expressions \vref{ddt_loc_en_trav_fr,ddt_loc_L2_trav_fr} for the time derivatives of a localized energy and a localized $L^2$ functional, for all $s$ in $[0,+\infty)$, 
\begin{equation}
\label{ds_fire_1}
\begin{aligned}
\fff'(s) =  \int_{\iTot(s)} \Biggl[ &
\psi \Bigl( - \coeffEn v_s^2 - v \cdot \nabla V^\dag(v) - v_\rho^2 \Bigr) + \psi_s \biggl( \coeffEn \Bigl( \frac{1}{2}v_\rho^2 + V^\dag(v) \Bigr) + \frac{1}{2}v^2 \biggr) \\
& + \Bigl( \frac{d-1}{\rInit + cs + \rho} \psi + c\psi - \psi_\rho \Bigr) ( \coeffEn v_s \cdot v_\rho + v \cdot v_\rho )
\Biggr] \, d\rho 
\,.
\end{aligned}
\end{equation}
(this makes use of the ``first'' version of the time derivative of the $L^2$-functional written in \cref{ddt_loc_L2_trav_fr}, without the additional integration by parts of $c\psi - \psi_\rho$). The aim of the next calculations is to control the two last terms below this integral. 

It follows from the definition of $\psi$ that, for every nonnegative quantity $s$, 
\begin{equation}
\label{ds_psi}
\psi_s(\rho,s) =
\left\{
\begin{aligned}
& \Bigl( - \kappa\cCut + \frac{c (d-1)}{\rInit + cs + \rho}\Bigr) \psi(\rho,s)
& & \text{if}\quad\rho \in \iLeft(s) \,, \\
& -\kappa\cCut \psi(\rho,s)
& & \text{if}\quad\rho \in \iMain(s) \,, \\
& \cCut (c+\kappa) \psi(\rho,s)
& & \text{if}\quad\rho \in \iRight(s) \,, \\
\end{aligned}
\right.
\end{equation}
and
\[
\psi_\rho(\rho,s) = 
\left\{
\begin{aligned}
& \Bigl( \frac{d-1}{\rInit + cs + \rho} +c+\kappa \Bigr) \psi(\rho,s)
& & \text{if}\quad\rho \in \iLeft(s) \,, \\
&  (c+\kappa) \psi(\rho,s)
& & \text{if}\quad\rho \in \iMain(s) \,, \\
& -\kappa \psi(\rho,s) 
& & \text{if}\quad\rho \in \iRight(s) \,, \\
\end{aligned}
\right.
\]
thus
\begin{equation}
\label{c_psi_minus_psi_rho}
\begin{aligned}
\frac{d-1}{\rInit + cs + \rho}  \psi(\rho,s) + & c\psi(\rho,s)-\psi_\rho(\rho,s) = \\
& \left\{ 
\begin{aligned}
& - \kappa \psi(\rho,s)
& & \text{if}\quad\rho \in \iLeft(s) \,, \\
& \Bigl(\frac{d-1}{\rInit + cs + \rho} -\kappa\Bigr) \psi(\rho,s)
& & \text{if}\quad\rho \in \iMain(s) \,, \\
& \Bigl( \frac{d-1}{\rInit + cs + \rho} + c+\kappa \Bigr) \psi(\rho,s) 
& & \text{if}\quad\rho \in \iRight(s) \,. \\
\end{aligned}
\right.
\end{aligned}
\end{equation}
As in the case $d=1$ (see \cite{Risler_globalBehaviour_2016}), the sole problematic term in the right-hand side of expression \cref{ds_fire_1} (with respect to the conclusions of \cref{lem:decrease_firewall}) is the product
\[
(c\psi-\psi_\rho) \, v \cdot v_\rho
\]
on the interval $\iRight(s)$. As in \cite{Risler_globalBehaviour_2016}, this term can be integrated by parts one more time to take advantage of the smallness of $\psi_{\rho\rho}-c\psi_\rho$ on $\iRight(s)$. There are several ways to proceed, since the integration by parts may be performed either only on $\iMain(s) \cup \iRight(s)$ or on the whole interval $\iTot(s)$. Since the first option would create a border term at the left of $\iMain(s)$ let us go on with the second option. Doing so, it follows from \cref{ds_fire_1} that
\begin{equation}
\label{ds_fire_1_bis}
\begin{aligned}
\fff'(s) =  \int_{\iTot(s)} \Biggl[ &
\psi \Bigl( - \coeffEn v_s^2 - v \cdot \nabla V^\dag(v) - v_\rho^2 \Bigr) + \psi_s \biggl( \coeffEn \Bigl( \frac{1}{2}v_\rho^2 + V^\dag(v) \Bigr) + \frac{1}{2}v^2 \biggr) \\
& + \coeffEn \Bigl( \frac{d-1}{\rInit + cs + \rho} \psi + c\psi - \psi_\rho \Bigr) v_s \cdot v_\rho + \frac{d-1}{\rInit + cs + \rho} \psi v \cdot v_\rho \\
& + \frac{\psi_{\rho\rho} - c \psi_\rho}{2}v^2
\Biggr] \, d\rho 
\,.
\end{aligned}
\end{equation}
It follows from the expression of $\psi_\rho$ above that, for every nonnegative quantity $s$, 
\begin{equation}
\label{psi_rho_rho_minus_c_psi_rho}
\psi_{\rho\rho}(\rho,s) - c\psi_\rho(\rho,s) \le \theta(\rho,s)
\quad\text{for all}\quad
\rho \in \iTot(s) 
\,,
\end{equation}
where
\[
\theta(\rho,s) = 
\left\{
\begin{aligned}
& \biggl(
\kappa(c+\kappa) + \frac{(c+2\kappa) (d-1)}{\rInit + cs + \rho} + \frac{(d-1)(d-2)}{(\rInit + cs + \rho)^2} 
\biggr) \psi(\rho,s)
&\quad\text{if}\quad\rho \in \iLeft(s)  \,, &\ \\
&  \kappa(c+\kappa)\psi(\rho,s)
&\quad\text{if}\quad\rho \not\in \iLeft(s) \,. &\
\end{aligned}
\right.
\]
Indeed, $\psi_{\rho\rho} - c\psi_\rho$ equals $\theta$ plus two Dirac masses of negative weight (one at the junction between $\iLeft(s)$ and $\iMain(s)$, and one at the junction between $\iMain(s)$ and $\iRight(s)$).

Observe that for every $\rho$ in the interval $\iMain(s)\cup\iRight(s)$, the quantity $\rInit + cs + \rho$ is not smaller than $\rSmallCurv$. As a consequence, it follows from equality \cref{ds_fire_1_bis} that, for every nonnegative quantity $s$, 
\begin{equation}
\label{ds_fire_1_ter}
\begin{aligned}
\fff'(s) \le & \int_{\iTot(s)}  \psi \Biggl[ 
- \coeffEn v_s^2 - v \cdot \nabla V^\dag(v) - v_\rho^2 + \cCut(c+\kappa) \biggl( \coeffEn \Bigl( \frac{1}{2}v_\rho^2 + V^\dag(v) \Bigr) + \frac{1}{2}v^2 \biggr) \\
& + \coeffEn \Bigl( \frac{d-1}{\rSmallCurv} + c + \kappa \Bigr) \abs{v_s \cdot v_\rho} + \frac{d-1}{\rSmallCurv} \abs{v \cdot v_\rho} +  \frac{\kappa(c+\kappa)}{2}v^2
\Biggr] \, d\rho \\
& + \FResidueLeft(s)
\end{aligned}
\end{equation}
where
\[
\begin{aligned}
\FResidueLeft(s) = & \int_{\iLeft(s)} \frac{d-1}{\rInit + cs + \rho} \psi \Biggl[ c \biggl( \coeffEn \Bigl( \frac{1}{2}v_\rho^2 + V^\dag(v) \Bigr) + \frac{1}{2}v^2 \biggr) 
 \\
& +  v \cdot v_\rho + \frac{1}{2}\biggl((c+2\kappa) + \frac{(d-2)}{\rInit + cs + \rho} \biggr) v^2
\Biggr] \, d\rho
\,.
\end{aligned}
\]
The following lemma deals with the ``residual'' term $\FResidueLeft(s)$.
\end{proof} 
\begin{lemma}[control on the residual integral over $\iLeft(s)$]
\label{lem:bd_psi_s_iLeft}
There exists a positive quantity $\KFLeft$, depending only on $V$ and $d$, such that, for every nonnegative quantity $s$, the following estimate holds:
\[
\FResidueLeft(s) \le \KFLeft \, \exp(-c\rInit)
\,.
\]
\end{lemma}
\begin{proof}[Proof of \cref{lem:bd_psi_s_iLeft}]
Since $\psi$ is smaller than $\chi$ on the interval $\iLeft(s)$, the proof is identical to that of \vref{lem:bound_ds_en_iLeft} (observe the vanishing term in $\FResidueLeft(s)$ if $d=2$).
\end{proof}
\begin{proof}[End of the proof of \cref{lem:decrease_firewall}]
Using the inequalities
\[
\begin{aligned}
\Bigl( \frac{d-1}{\rSmallCurv} + c + \kappa \Bigr) \abs{v_s \cdot v_\rho} &\le v_s^2 + \frac{\Bigl( \frac{d-1}{\rSmallCurv} + c + \kappa \Bigr)^2}{4}  v_\rho^2 \\
\quad\text{and}\quad
\abs{v \cdot v_\rho} &\le \frac{1}{2}v^2 + \frac{1}{2}v_\rho^2
\,,
\end{aligned}
\]
it follows from inequality \cref{ds_fire_1_ter} and from \cref{lem:bd_psi_s_iLeft} that, for every nonnegative quantity $s$, 
\[
\begin{aligned}
\fff'(s) \le & \int_{\iTot(s)}\psi\Biggl[ 
\biggl( -1 + \coeffEn\Bigl(\frac{\cCut(c+\kappa)}{2} + \frac{(c+\kappa+\frac{d-1}{\rSmallCurv})^2}{4}\Bigr) + \frac{d-1}{2\rSmallCurv} \biggr) v_\rho^2 \\
& - v \cdot \nabla V^\dag(v) + \coeffEn\cCut(c+\kappa) \abs{V^\dag(v)}
+ \Bigl( \frac{d-1}{2\rSmallCurv} + \frac{(\cCut+\kappa)(c+\kappa)}{2} \Bigr) v^2
\Biggr] \, d\rho \\
& + \KFLeft \, \exp(-c\rInit)  
\,.
\end{aligned}
\]
Since according to the definition \vref{def_r_s_c} for $\rSmallCurv$ the quantity $(d-1)/\rSmallCurv$ is smaller than $1/2$ and than $\eigVmin(m)/8$, it follows that
\[
\begin{aligned}
\fff'(s) \le & \int_{\iTot(s)}\psi\Biggl[ 
\biggl( -1 + \coeffEn\Bigl(\frac{\cCut(c+\kappa)}{2} + \frac{(c+\kappa+\frac{1}{2})^2}{4}\Bigr) + \frac{1}{4} \biggr) v_\rho^2 \\
& - v \cdot \nabla V^\dag(v) + \coeffEn\cCut(c+\kappa) \abs{V^\dag(v)}
+ \Bigl( \frac{\eigVmin(m)}{16} + \frac{(\cCut+\kappa)(c+\kappa)}{2} \Bigr) v^2
\Biggr] \, d\rho \\
& + \KFLeft \, \exp(-c\rInit)
\,,
\end{aligned}
\]
and according to the properties \vref{constraints_quantities_relax_sch} satisfied by the quantities $\kappa$ and $\cCut$ and $\coeffEn$, it follows that
\begin{equation}
\label{ds_fire_2}
\begin{aligned}
\fff'(s) \le & \int_{\iTot(s)}\psi\Bigl( 
-\frac{1}{2}v_\rho^2 - v \cdot \nabla V^\dag(v) + \frac{1}{4} \abs{V^\dag(v)} + \frac{\eigVmin(m)}{8} v^2 \Bigr) \, d\rho \\
& + \KFLeft \, \exp(-c\rInit)
\,.
\end{aligned}
\end{equation}
Let $\nuF$ be a positive quantity to be chosen below. It follows from the previous inequality and from the definition \cref{def_fff} of $\fff(s)$ that
\begin{equation}
\label{ds_fire_3}
\begin{aligned}
\fff'(s) + \nuF \fff(s) - \KFLeft \, &\exp(-c\rInit) \le \int_{\iTot(s)}\psi\biggl[ -\frac{1}{2}(1-\nuF\, \coeffEn)v_\rho^2 - v \cdot \nabla V^\dag(v) \\
& +  \Bigl(\frac{1}{4}+\nuF\coeffEn\Bigr) \abs{V^\dag(v)} + \Bigl(\frac{\eigVmin(m)}{8} + \frac{\nuF}{2}\Bigr)v^2 \biggr] \, d\rho 
\,.
\end{aligned}
\end{equation}
In view of this inequality and of inequalities \vref{v_nablaV_controls_square_around_loc_min_dag,v_nablaV_controls_pot_around_loc_min_dag}, let us assume that $\nuF$ is small enough so that
\begin{equation}
\label{def_nu_fire}
\nuF\, \coeffEn \le 1
\quad\text{and}\quad
\nuF\, \coeffEn \le \frac{1}{4}
\quad\text{and}\quad
\frac{\nuF}{2} \le \frac{\eigVmin(m)}{8}
\,;
\end{equation}
the quantity $\nuF$ may be chosen as
\[
\nuF = \min \Bigl( \frac{1}{4\coeffEn} , \frac{\eigVmin(m)}{4} \Bigr)
\,.
\]
Then, it follows from \cref{ds_fire_3,def_nu_fire} that
\begin{equation}
\label{ds_fire_4}
\fff'(s) + \nuF \fff(s) - \KFLeft \, \exp(-c\rInit) \le \int_{\iTot(s)}\psi\Bigl[- v \cdot \nabla V^\dag(v) + \frac{1}{2} \abs{V^\dag(v)} + \frac{\eigVmin(m)}{4} v^2 \Bigr]\, d\rho 
\,.
\end{equation}
According to \cref{v_nablaV_controls_square_around_loc_min_dag,v_nablaV_controls_pot_around_loc_min_dag}, the integrand of the integral at the right-hand side of this inequality is nonpositive as long as $\rho$ is \emph{not} in $\SigmaEsc(s)$.
Therefore this inequality still holds if the domain of integration of this integral is changed from $\iTot(s)$ to $\SigmaEsc(s)$. 
Thus, introducing the quantity 
\[
\KF = \max_{u\in\rr^{\dState}, \ \abs{u}\le \Rattinfty}\Bigl[- (u-m)\cdot \nabla V(u) + \frac{1}{2}\abs{V(u)-V(m)}+ \frac{\eigVmin(m)}{4}(u-m)^2 \Bigr]
\]
(which is positive and depends only on $V$ and $m$), inequality \cref{ds_fire} follows from inequality \cref{ds_fire_4}. \Cref{lem:decrease_firewall} is proved.
\end{proof}
\subsubsection{Nonnegativity of the firewall}
For all $s$ in $[0,+\infty)$,
\begin{equation}
\label{nonnegativity_of_fff}
\fff(s) \ge 0
\,.
\end{equation}
Indeed, in view of the property \vref{property_weight_en} concerning $\coeffEnZero$ and since $\coeffEn$ is not larger than $\coeffEnZero$, for all $s$ in $[0,+\infty)$ the following stronger coercivity property holds:
\[
\fff(s) \ge \min\Bigl(\frac{\coeffEn}{2},\frac{1}{4}\Bigr) \int_{\iTot(s)} \psi(\rho,s) \bigl( v_\rho(\rho,s)^2 + v(\rho,s)^2 \bigr) \, d\rho
\,.
\]
\subsubsection{Relaxation scheme inequality, 2}
For every nonnegative quantity $s$, let
\[
\mathcal{G}(s) = \int_{\SigmaEsc(s) } \psi(\rho,s) \, d\rho
\,.
\]
Integrating inequality \cref{ds_fire} between $0$ and a nonnegative quantity $\sFin$ yields, since according to \cref{nonnegativity_of_fff} $\fff(\sFin)$ is nonnegative, 
\[
\int_0^{\sFin} \fff(s) \, ds \le \frac{1}{\nuF} \Bigl( \fff(0) + \KF \int_0^{\sFin} \mathcal{G} (s) \, ds 
+ \KFLeft \, \sFin \, \exp(-c\, \rInit) \Bigr)
\,.
\]
Thus, introducing the expression
\[
\tidleKECurv(r,s,c) = \KECurv(r,s,c) + \frac{\KERight \, \KFLeft \, \sFin}{\nuF} \exp(-c\, \rInit)
\,,
\]
the ``relaxation scheme'' inequality \vref{relax_sch_int_fff} becomes
\begin{equation}
\label{relax_scheme_ggg}
\begin{aligned}
\frac{1}{2} \int_0^{\sFin} \ddd(s) \, ds \le & \eee(0) - \eee(\sFin)  + \tidleKECurv \bigl(\rInit,\sFin,c\bigr) \\
& + \frac{\KERight}{\nuF}  \Bigl(\fff(0) + \KF \int_0^{\sFin} \mathcal{G} (s) \, ds\Bigr)
\,.
\end{aligned}
\end{equation}
Observe that, as was the case for $\KECurv(r,s,c)$, the ``curvature term'' $\tidleKECurv(r,s,c)$ (still) goes to $0$ as $r$ goes to $+\infty$, uniformly with respect to $s$ bounded and $c$ bounded away from $0$ and $+\infty$.
The next step is to gain some control over the quantity $\mathcal{G}(s)$.
\subsubsection{Control over pollution in the time derivative of the firewall function}
\label{subsubsec:flux_der_fire}
For every nonnegative quantity $s$ let
\[
\begin{aligned}
\rhoHom(s) &= \rHom (\tInit + s) - \rInit - cs \,, \\
\text{and}\quad\rhoesc(s) &= \resc (\tInit + s) - \rInit - cs
\,.
\end{aligned}
\]
According to properties \vref{rEsc_resc_rHom} for the set $\SigmaEscZero(t)$, 
\[
\SigmaEsc(s) \subset (-\infty, \rhoesc(s)] \cup [\rhoHom(s),+\infty)
\,.
\]
Let us introduce the quantities 
\[
\Gback(s) = \int_{-\rInit - cs}^{\rhoesc(s)} \psi (\rho,s)\, d\rho
\quad\text{and}\quad
\Gfront(s) = \int_{\rhoHom(s)}^{+\infty} \psi (\rho,s)\, d\rho
\,;
\]
observe that, by definition --- see \vref{def_resc} --- the quantity $\rhoesc(s)$ is greater than or equal to $\rSmallCurv -\rInit - cs$, and is therefore greater than $-\rInit - cs$. Then,
\[
\mathcal{G} (s) \le \Gback(s) + \Gfront(s)
\,.
\]
Let us make the following hypothesis (required for the next lemma to hold):
\begin{equation}
\label{hyp_c_close_to_barcescsup}
(c+\kappa)(\barcescsup - c) \le \frac{\kappa \cCut}{4}
\end{equation}
(this hypothesis is satisfied as soon as $c$ is close enough to $\barcescsup$).
\begin{lemma}[upper bounds on pollution terms in the derivative of the firewall]
\label{lem:bound_Gback_Gfront}
There exists \\
a positive quantity $K[u_0]$, depending only on $V$ and on the initial condition $u_0$ (but not on the parameters $\tInit$ and $\rInit$ and $c$ and $\initCut$ of the relaxation scheme) such that, for every nonnegative quantity $s$, 
\begin{equation}
\label{up_bd_S_G_back_G_front}
\begin{aligned}
\Gback(s) & \le K[u_0] \exp(-\kappa \, \initCut) \exp\Big( -\frac{\kappa\, \cCut}{2}s\Bigl) \\
\Gfront(s) & \le \frac{1}{\kappa}\exp\bigl( (\cnoesc+1)\,  \initCut \bigr) \exp \bigl( (\cnoesc + \kappa) (\cCut + \kappa)  s \bigr) \exp \bigl( -\kappa \, \rhoHom(0)\bigr)
\,.
\end{aligned}
\end{equation}
\end{lemma}
\begin{proof}
The proof is identical to that of \cite[\GlobalBehaviourLemBoundGbackGfront]{Risler_globalBehaviour_2016}.
\end{proof}
\subsubsection{Relaxation scheme inequality, final}
\label{subsubsec:fin_relax}
Let us introduce the quantity
\[
\KGback[u_0] = \frac{2 \, \KERight \, \KF \, K[u_0]}{ \nuF \, \kappa \cCut }
\,,
\]
and, for every nonnegative quantity $s$, the quantity
\[
\KGfront (s) = \frac{ \KERight \, \KF}{\nuF \, \kappa \, (\cnoesc + 1) (\cCut + 1)}
\exp \bigl( (\cnoesc + 1) (\cCut + 1)  s \bigr)
\,.
\]
Then, for every nonnegative quantity $\sFin$, according to inequalities \cref{up_bd_S_G_back_G_front}, the ``relaxation scheme'' inequality \cref{relax_scheme_ggg} can be rewritten as
\begin{equation}
\label{relax_scheme_final}
\begin{aligned}
\frac{1}{2} \int_0^{\sFin} \ddd(s) \, ds \le & \eee(0) - \eee(\sFin) + \frac{\KERight}{\nuF} \fff(0) 
+ \KGback[u_0] \exp(-\kappa \, \initCut) \\
& + \KGfront (\sFin)\exp\bigl( (\cnoesc+1)\,  \initCut \bigr) \exp \bigl( -\kappa \, \rhoHom(0)\bigr) \\
& + \KECurv\bigl(\rInit,\sFin,c\bigr)
\,.
\end{aligned}
\end{equation}
Recall that the ``curvature term'' $\KECurv(r,s,c)$ goes to $0$ as $r$ goes to $+\infty$, uniformly with respect to $s$ bounded and $c$ bounded away from $0$ and $+\infty$. Recall by the way that this last inequality requires the additional hypothesis \vref{hyp_rInit_larger_twice_rsc} made on the quantity $\rInit$ (namely, $\rInit$ should not be smaller than $2\rSmallCurv$). 
\subsection{Convergence of the mean invasion speed}
\label{subsec:cv_mean_inv_vel}
The aim of this \namecref{subsec:cv_mean_inv_vel} is to prove the following proposition. 
\begin{proposition}[mean invasion speed]
\label{prop:cv_mean_inv_vel}
The following equalities hold:
\[
\cescinf = \cescsup = \barcescsup
\,.
\]
\end{proposition}
\begin{proof}
Let us proceed by contradiction and assume that
\[
\cescinf < \barcescsup
\,.
\]
Then, let us take and fix a positive quantity $c$ satisfying the following conditions:
\begin{equation}
\label{hyp_c_cv_mean_vel}
\cescinf < c < \barcescsup \le c + \frac{\kappa \cCut}{4(\cnoesc+\kappa)}
\quad\text{and}\quad
\Phi_c(m) = \emptyset
\,.
\end{equation}
The first condition is satisfied as soon as $c$ is smaller than and close enough to $\barcescsup$, thus existence of a quantity $c$ satisfying the two conditions follows from hypothesis \textup{(\hyperlink{hypDiscVel}{\hypDiscVelRef})}.

The contradiction will follow from the relaxation scheme set up in \cref{subsec:relax_sch_tr_fr}. The main ingredient is: since the set $\Phi_c(m)$ is empty, some dissipation must occur permanently around the escape point in a referential travelling at the speed $c$. This is stated by the following lemma. 
\begin{lemma}[nonzero dissipation in the absence of travelling front]
\label{lem:dissip_no_tf_vel}
There exist positive quantities $L$ and $\epsDissip$ such that, for every $t$ in $[0,+\infty)$, if the quantity $\resc(t)$ is greater than or equal to $L$, then the following inequality holds:
\[
\norm{\rho \mapsto u_t\bigl( \resc (t) + \rho, t\bigr) + c u_r \bigl( \resc (t) + \rho, t\bigr) }_{L^2([-L,L],\rr^{\dState})} \ge \epsDissip
\,.
\]
\end{lemma}
\begin{proof}[Proof of \cref{lem:dissip_no_tf_vel}]
Let us proceed by contradiction and assume that the converse is true. Then, for every positive integer $n$, there exists $t_n$ in $[0,+\infty)$ such that the quantity $\resc(t_n)$ is greater than or equal to $n$ and such that
\begin{equation}
\label{proof_lem_dissip}
\norm{\rho \mapsto u_t\bigl( \resc (t_n) + \rho, t_n \bigr) + c u_r \bigl( \resc (t_n) + \rho, t_n \bigr) }_{L^2([-p,p],\rr^{\dState})} \le \frac{1}{n}
\,.
\end{equation}
By compactness (\cref{lem:compactness}), up to replacing the sequence $(t_n)_{n\in\nn}$ by a subsequence, there exists an entire solution $\widebar{u}$ of system \cref{syst_rad_sym_large_radius} in 
\[
\ccc^0\left(\rr,\cccbR{2}\right)\cap \ccc^1\left(\rr,\cccbR{0}\right)
\,,
\]
such that, with the notation of \cref{compactness}, 
\[
D^{2,1}u(\resc (t_n)+\cdot,t_n+\cdot)\to D^{2,1}\widebar{u}
\quad\text{as}\quad
n\to+\infty
\,,
\]
uniformly on every compact subset of $\rr^2$. According to inequality \cref{proof_lem_dissip}, the function $\rho\mapsto\widebar{u}_t(\rho,0)+c\widebar{u}_r(\rho,0)$ vanishes identically, so that the function $\rho\mapsto\widebar{u}(\rho,0)$ is a solution of system \cref{syst_trav_front} governing the profiles of waves travelling at the speed $c$ for system \cref{syst_rad_sym_large_radius}. According to the properties of the escape point \vref{rEsc_resc_rHom,rHom_minus_resc}, 
\[
\sup_{\rho \in[0,+\infty)} \abs{\phi(\rho)-m} \le \dEsc(m)
\,,
\]
thus it follows from \cite[\GlobalBehaviourLemTWApproachCriticalPoints]{Risler_globalBehaviour_2016} that $\phi(\rho)$ goes to $m$ as $\rho$ goes to $+\infty$. On the other hand, according to the bound \cref{hyp_attr_ball} on the solution, $\abs{\phi(\cdot)}$ is bounded (by $\Rattinfty$), and since $\Phi_c(m)$ is empty, it follows from \cite[\GlobalBehaviourLemTWApproachCriticalPoints]{Risler_globalBehaviour_2016} that $\phi(\cdot)$ is identically equal to $m$, a contradiction with the definition of $\resc(\cdot)$. 
\end{proof}
The remaining of the proof of \cref{prop:cv_mean_inv_vel} is identical to that of \cite[\GlobalBehaviourPropMeanInvasionSpeed]{Risler_globalBehaviour_2016}, therefore it will not be reproduced here. 
\end{proof}
According to \cref{prop:cv_mean_inv_vel}, the three quantities $\cescinf$ and $\cescsup$ and $\barcescsup$ are equal; let
\[
\cesc
\]
denote their common value.
\subsection{Further control on the escape point}
\label{subsec:further_control}
\begin{proposition}[mean invasion speed, further control]
\label{prop:further_control}
The following equality holds:
\[
\underbarcescinf = \cesc
\,.
\]
\end{proposition}
\begin{proof}
The proof is identical to that of \cite[\GlobalBehaviourPropInvasionSpeedFurtherControl]{Risler_globalBehaviour_2016}.
\end{proof}
\subsection{Dissipation approaches zero at regularly spaced times}
\label{subsec:dissip_zero_some_times}
For every $t$ in $[1,+\infty)$, the following set
\[
\left\{\varepsilon \text{ in } (0,+\infty) : \int_{-1/\varepsilon}^{+1/\varepsilon} 
\Bigl( u_t\bigl(\resc(t)+\rho,t\bigr)+\cesc u_r\bigl(\resc(t)+\rho,t\bigr) \Bigr)^2\, d\rho  \le \varepsilon \right\}
\]
is (according to the bounds \vref{hyp_attr_ball} for the solution) a nonempty interval (which by the way is unbounded from above). Let 
\[
\deltaDissip(t)
\] 
denote the infimum of this interval. This quantity measures to what extent the solution is, at time $t$ and around the escape point $\resc(t)$, close to be stationary in a frame travelling at the speed $\cesc$. The next goal is to prove that
\[
\deltaDissip(t) \to 0
\quad\text{as}\quad
t\to +\infty
\,.
\]
\Cref{prop:dissp_zero_some_times} below can be viewed as a first step towards this goal. 
\begin{proposition}[regular occurrence of small dissipation]
\label{prop:dissp_zero_some_times}
For every positive quantity $\varepsilon$, there exists a positive quantity $T(\varepsilon)$ such that, for every $t$ in $[0,+\infty)$, 
\[
\inf_{t'\in[t,t+T(\varepsilon)]} \deltaDissip(t') \le \varepsilon
\,.
\]
\end{proposition}
\begin{proof}
The proof is identical to that of \cite[\GlobalBehaviourPropRegularSmallDissipation]{Risler_globalBehaviour_2016}.
\end{proof}
\subsection{Relaxation}
\label{subsec:relax}
\begin{proposition}[relaxation]
\label{prop:dissip_app_zero}   
The following assertion holds:
\[
\deltaDissip(t)\to0\quad\text{as}\quad
 t\to+\infty \,. 
\]
\end{proposition}
\begin{proof}
The proof is identical to that of \cite[\GlobalBehaviourPropRelaxation]{Risler_globalBehaviour_2016}.
\end{proof}
\subsection{Convergence}
\label{subsec:convergence}
The end of the proof of \vref{prop:inv_cv} (``invasion implies convergence'') is a straightforward consequence of \cref{prop:dissip_app_zero}, and is identical to the case of space dimension one treated in \cite[\GlobalBehaviourSubSecConvergence{} and \GlobalBehaviourSubSecHomPoint]{Risler_globalBehaviour_2016}. As mentioned above, the definition of the quantity $\deltaDissip(t)$ is slightly different from that of \cite{Risler_globalBehaviour_2016}, however since this quantity goes to $0$ as time goes to $+\infty$, limits of the profiles of the solution around the escape point $\resc(t)$ must (still with this new definition of $\deltaDissip(t)$) necessarily be solutions of system \cref{syst_trav_front} satisfied by the profiles of travelling fronts. Thus all arguments remain the same, and details will not be reproduced here. \Cref{prop:inv_cv} is proved. 
\section{No invasion implies relaxation}
\label{sec:non_invasion_implies_relaxation}
\subsection{Definitions and hypotheses}
\label{subsec:def_hyp_dichot}
As everywhere else, let us consider a function $V$ in $\ccc^2(\rr^{\dState},\rr)$ satisfying the coercivity hypothesis \cref{hyp_coerc}. As in \cref{sec:inv_impl_cv}, let us consider a point $m$ in $\mmm$ and a solution $(r,t)\mapsto u(r,t)$ of system \cref{syst_rad_sym}, and let us make the following hypothesis (which is identical to the one made in \cref{subsec:inv_cv_def_hyp}):
\begin{description}
\item[\hypHomLabel]\hypertarget{hypHomBis} There exists a positive quantity $\cHom$ and a $\ccc^1$-function 
\[
\rHom : [0,+\infty)\to \rr\,,
\quad\text{satisfying}\quad
\rHom'(t) \to \cHom
\quad\text{as}\quad
t\to + \infty
\,,
\]
such that, for every positive quantity $L$, 
\[
\sup_{\rho\in[-L,L]}\abs{u\bigl( \rHom(t) + \rho, t\bigr) - m}  \to 0
\quad\text{as}\quad
t\to + \infty
\,.
\]
\end{description}
Let us define the function $t\mapsto \rEsc(t)$ and the quantity $\cEsc$ exactly as in \cref{subsec:inv_cv_def_hyp}. By contrast with \cref{subsec:inv_cv_def_hyp} and the hypothesis \textup{(\hyperlink{hypInv}{\hypInvRef})} made there, let us introduce the following (converse) hypothesis. 
\begin{description}
\item[\hypNoInvLabel]\hypertarget{hypNoInv} The quantity $\cEsc$ is nonpositive. 
\end{description}
\subsection{Statement and proof}
\label{subsec:statement_dichot}
\begin{proposition}[no invasion implies relaxation]
\label{prop:no_invasion_implies_relaxation}
Assume that $V$ satisfies the coercivity hypothesis \cref{hyp_coerc} and, keeping the definitions and notation above, let us assume that the solution under consideration satisfies hypotheses \textup{(\hyperlink{hypHomBis}{\hypHomRef})} and \textup{(\hyperlink{hypNoInv}{\hypNoInvRef})}. Then the following conclusions hold. 
\begin{enumerate}
\item There exists a nonnegative quantity $\eeeResAsympt[u]$ (``residual asymptotic energy'') such that, for every quantity $c$ in the interval $(0,\cHom)$,
\[
\int_0^{ct} r^{d-1} \left(\frac{1}{2} u_r(r,t)^2 + V\bigl(u(r,t)\bigr) - V(m)\right)\, dr \to \eeeResAsympt[u]
\quad\text{as}\quad 
t\to+\infty
\,.
\]
\label{item:prop_no_invasion_implies_relaxation_asympt_energy}
\item The quantity
\[
\sup_{r\in[0,\rHom(t)]}\abs{u_t(r,t)} 
\]
goes to $0$ as time goes to $+\infty$. 
\label{item:prop_no_invasion_implies_relaxation_time_derivative}

\item For every quantity $c$ in the interval $(0,\cHom)$, the function
\[
t\mapsto \int_0^{ct} r^{d-1} u_t(r,t)^2 \, dr
\]
is integrable on a neighbourhood of $+\infty$. 
\label{item:prop_no_invasion_implies_relaxation_integrability_of_dissipation}
\end{enumerate}
\end{proposition}
\begin{proof}
\cite[\InvasionRelaxationThmUnderHypothesisHypHom]{Risler_noInvasionCaseHigherSpace_2020} states the same conclusions as \cref{prop:no_invasion_implies_relaxation} but in a broader setting, that is for solutions of system \cref{syst_higher_dim} without the hypothesis of radial symmetry. The reader is therefore referred to the proof provided in this reference. 
\end{proof}
\section{Relaxation implies convergence}
\label{sec:relax_implies_cv}
\subsection{Statement}
As everywhere else, let us consider a function $V$ in $\ccc^2(\rr^{\dState},\rr)$ satisfying the coercivity hypothesis \cref{hyp_coerc}. As in the previous \namecref{sec:non_invasion_implies_relaxation}, let us consider a point $m$ in $\mmm$ and a solution $(r,t)\mapsto u(r,t)$ of system \cref{syst_rad_sym}, and let us assume again that the same hypotheses \textup{(\hyperlink{hypHomBis}{\hypHomRef})} and \textup{(\hyperlink{hypNoInv}{\hypNoInvRef})} hold. Thus the conclusions of \cref{prop:no_invasion_implies_relaxation} hold. Let us keep all the notation introduced in the previous \namecref{sec:non_invasion_implies_relaxation}, and let us make the following additional (generic) assumption. 
\begin{description}
\item[\hypDiscStationarymLabel]\hypertarget{hypDiscStationarym} The set
\[
\bigl\{ \phi(0) : \phi\in\PhiZeroCentre(m) \bigr\}
\]
is totally discontinuous in $\rr^{\dState}$. That is, its connected components are singletons. Equivalently, the set $\PhiZeroCentre(m)$ is totally disconnected for the topology of compact convergence (uniform convergence on compact subsets of $[0,+\infty)$).
\end{description}
Note that hypothesis \textup{(\hyperlink{hypDiscStationary}{\hypDiscStationaryRef})} stated in \cref{subsubsec:add_hyp_V} is identical except that it concerns all elements of $\mmm$ instead of the single point $m$ as in \textup{(\hyperlink{hypDiscStationarym}{\hypDiscStationarymRef})}. 
The aim of this \namecref{sec:relax_implies_cv} is to prove the additional conclusion provided by the following proposition. 
\begin{proposition}[relaxation implies convergence]
\label{prop:relax_implies_cv}
The following conclusions hold.
\begin{enumerate}
\item There exists a stationary solution $\phi$ in $\PhiZeroCentre(m)$ such that
\begin{equation}
\label{conclusion_relax_implies_cv}
\sup_{r\in[0,\rHom(t)]}\abs{u(r,t)-\phi(r)}\to 0
\quad\text{as}\quad
t\to+\infty
\,.
\end{equation}
\label{item:prop_relax_implies_cv_approach_phi}
\item The residual asymptotic energy $\eeeResAsympt[u]$ of the solution is equal to the energy $\eee[\phi]$ of this stationary solution. 
\label{item:prop_relax_implies_cv_value_asympt_en}
\end{enumerate}
\end{proposition}
\subsection{Properties of the Escape radius}
Recall that, for every nonnegative time $t$, the escape radius $\rEsc(t)$ is, according to its definition (see \cref{subsec:inv_cv_def_hyp}), either equal to $-\infty$ or nonnegative. 
\subsubsection{Transversality}
\begin{lemma}[transversality at Escape radius]
\label{lem:transversality_at_rEsc}
There exists a positive time $\tEscTransv$ and a positive quantity $\epsEscTransv$ such that, for every time $t$ greater than or equal to $\tEscTransv$, if $\rEsc(t)$ is not equal to $-\infty$, then
\begin{equation}
\label{transversality_at_rEsc}
\Bigl(u\bigl(\rEsc(t),t\bigr)-m\Bigr)\cdot u_r\bigl(\rEsc(t),t\bigr) \le - \epsEscTransv
\,.
\end{equation}
\end{lemma}
\begin{proof}
Let us proceed by contradiction and assume that there exists a sequence $(t_n)_{n\in\nn}$ such that $t_n$ goes to $+\infty$ as $n$ goes to $+\infty$ and such that, for every positive integer $n$, 
\begin{equation}
\label{by_contradiction_assumption_bare_transversality}
-\infty < \rEsc(t_n) 
\quad \text{and}\quad 
\Bigl(u\bigl(\rEsc(t_n),t_n\bigr)-m\Bigr)\cdot u_r\bigl(\rEsc(t_n),t_n\bigr) \ge -\frac{1}{n}
\,. 
\end{equation}
Up to extracting a subsequence from the sequence $(t_n)_{n\in\nn}$, it may be assumed that one among the following two assertions holds:
\begin{enumerate}
\item $\rEsc(t_n)$ goes to $+\infty$ as $t$ goes to $+\infty$;
\label{item:rEsc_of_tn_goes_to_plus_infinity}
\item there exists a nonnegative (finite) quantity $\rEscInfty$ such that $\rEsc(t_n)$ goes to $\rEscInfty$ as $t$ goes to $+\infty$. 
\label{item:rEsc_of_tn_goes_to_a_finite_limit}
\end{enumerate}
If assertion \cref{item:rEsc_of_tn_goes_to_plus_infinity} holds, then, according to \cref{lem:compactness} and to conclusion \cref{item:prop_no_invasion_implies_relaxation_asympt_energy} of \cref{prop:no_invasion_implies_relaxation}, up to extracting again a subsequence from the sequence $(t_n)_{n\in\nn}$, it may be assumed that the functions $\rho\mapsto u\bigl(\rEsc(t_n)+\rho,t_n\bigr)$ converge, uniformly on every compact subset of $\rr$, towards a $\ccc^2$-function $\phi:\rr\to\rr^{\dState}$ satisfying the system
\begin{equation}
\label{syst_stationary_solutions_large_radius_limit}
\phi'' = \nabla V(\phi)
\,;
\end{equation}
If assertion \cref{item:rEsc_of_tn_goes_to_a_finite_limit} holds, then, according to \cref{lem:compactness_close_to_origin} and to conclusion \cref{item:prop_no_invasion_implies_relaxation_asympt_energy} of \cref{prop:no_invasion_implies_relaxation}, up to extracting again a subsequence from the sequence $(t_n)_{n\in\nn}$, the functions $r\mapsto u(r,t_n)$ converge, uniformly on every compact subset of the interval $[0,+\infty)$, towards a $\ccc^2$-function $\hat{\phi}$ such that the function
\[
\phi:[-\rEscInfty,+\infty)\to\rr^{\dState}\,, \quad \rho\mapsto \hat{\phi}(\rEscInfty+\rho)
\]
satisfies the system
\begin{equation}
\label{syst_stationary_solutions_domain_bounded_to_the_left}
\phi'' + \frac{d-1}{\rho+\rEscInfty}\phi' = \nabla V(\phi)
\quad\text{with the boundary condition}\quad
\phi'(-\rEscInfty) = 0
\,;
\end{equation}
In both cases, it follows from assumptions \textup{(\hyperlink{hypHomBis}{\hypHomRef})} and \textup{(\hyperlink{hypNoInv}{\hypNoInvRef})} and from the definition of $\rEsc(t)$ that 
\[
\sup_{\rho\in[0,+\infty)}\abs{\phi(\rho)}\le \dEsc(m)
\,,
\]
and it follows from assumption \cref{by_contradiction_assumption_bare_transversality} that 
\begin{equation}
\label{no_transversality_contradiction}
\bigl(\phi(0)-m\bigr)\cdot \phi'(0)\ge 0
\,.
\end{equation}
In both cases, \vref{lem:spatial_asymptotics_stat_sol_stable_at_infinity} applies to the function $\phi$, and inequality \cref{no_transversality_contradiction} conflicts conclusion \cref{item:transv_spatial_asymptotics_ss} of this lemma, a contradiction. \Cref{lem:transversality_at_rEsc} is proved.  
\end{proof}
\subsubsection{Finiteness/infiniteness of Escape radius}
\begin{corollary}
\label{cor:finiteness_infiniteness_of_Escape_radius}
One of the following two (mutually exclusive) alternatives occurs: 
\begin{enumerate}
\item for every time $t$ greater than or equal to $\tEscTransv$, the quantity $\rEsc(t)$ equals $-\infty$, 
\item (or) for every time $t$ greater than or equal to $\tEscTransv$, the quantity $\rEsc(t)$ is positive. 
\label{item:cor_finiteness_infiniteness_of_Escape_radius}
\end{enumerate}
In addition, if the second alternative occurs, then the function $t\mapsto \rEsc(t)$ is of class $\ccc^1$ on the interval $[\tEscTransv,+\infty)$ and
\begin{equation}
\label{rEscprime_of_t_goes_to_zero}
\rEsc'(t)\to 0 \quad\text{as}\quad t\to+\infty
\,.
\end{equation}
\end{corollary}
\begin{proof}
Let us introduce the function 
\[
f:[0,+\infty)\times[0,+\infty)\to\rr\,,
\quad
(r,t)\mapsto \frac{1}{2}\Bigl(\bigl(u(r,t)-m\bigr)^2 - \dEsc(m)^2\Bigr)
\,.
\]
According to the smoothness properties of the solution recalled in \cref{subsec:glob_exist}, this function $f$ is of class $\ccc^1$ on $[0,+\infty)\times(0,+\infty)$. For every $(r,t)$ in $[0,+\infty)\times[0,+\infty)$, if $\rEsc(t)$ is not equal to $-\infty$ then it is nonnegative and $f\bigl(\rEsc(t),t\bigr)$ vanishes. If in addition $t$ is greater than or equal to the (positive) quantity $\tEscTransv$ defined in \cref{lem:transversality_at_rEsc}, then, according to inequality \cref{transversality_at_rEsc},
\begin{equation}
\label{partial_r_f_non_zero}
\partial_r f \bigl(\rEsc(t),t\bigr) = \Bigl(u\bigl(\rEsc(t),t\bigr)-m\Bigr)\cdot u_r\bigl(\rEsc(t),t\bigr) \le - \epsEscTransv <0
\,.
\end{equation}
In this case, since according to the border condition in system \cref{syst_rad_sym} the quantity $u_r(0,t)$ vanishes, the quantity $\rEsc(t)$ is necessarily positive. 
Let us introduce the set
\[
\ttt = \bigl\{t\in[\tEscTransv,+\infty): \rEsc(t)>-\infty\bigr\}
\,.
\]
It follows from inequality \cref{partial_r_f_non_zero}, from the fact that $\rEsc(t)$ is positive if $t$ is in $\ttt$, and from the Implicit Function Theorem that $\ttt$ is open in $[\tEscTransv,+\infty)$. And it follows from the definition of $\rEsc(t)$ that this set is closed in $[\tEscTransv,+\infty)$. As a consequence, the set $\ttt$ is either empty or equal to $[\tEscTransv,+\infty)$, and this proves the alternative (the first assertion of the lemma). 

If $\ttt$ equals $[\tEscTransv,+\infty)$, then it again follows from the Implicit Function Theorem that $t\mapsto\rEsc(t)$ is of class $\ccc^1$, and for every time $t$ in this interval, 
\[
\rEsc'(t) = - \frac{\partial_t f\bigl(\rEsc(t),t\bigr)}{\partial_r f\bigl(\rEsc(t),t\bigr)} = - \frac{\Bigl(u\bigl(\rEsc(t),t\bigr)-m\Bigr)\cdot u_t\bigl(\rEsc(t),t\bigr)}{\Bigl(u\bigl(\rEsc(t),t\bigr)-m\Bigr)\cdot u_r\bigl(\rEsc(t),t\bigr)}
\,.
\]
According to conclusion \cref{item:prop_no_invasion_implies_relaxation_time_derivative} of \cref{prop:no_invasion_implies_relaxation}, the numerator of this expression goes to $0$ as time goes to $+\infty$, while according to inequality \cref{partial_r_f_non_zero} the absolute value of the denominator remains not smaller than $\epsEscTransv$; it follows that $\rEsc'(t)$ goes to $0$ as time goes to $+\infty$. \Cref{cor:finiteness_infiniteness_of_Escape_radius} is proved. 
\end{proof}
\subsubsection{Infiniteness alternative for the Escape radius}
\begin{lemma}[infiniteness alternative for the Escape radius]
\label{lem:approach_to_m_alternative}
Assume that the first alternative of \cref{cor:finiteness_infiniteness_of_Escape_radius} occurs (that is, $\rEsc(t)$ equals $-\infty$ for every time $t$ greater than or equal to $\tEscTransv$). Then, 
\[
\sup_{r\in[0,\rHom(t)]}\abs{u(r,t)-m}\to0
\quad\text{as}\quad
t\to+\infty
\,.
\]
\end{lemma}
\begin{proof}
Let us proceed by contradiction and assume that the converse holds. Then there exists a positive quantity $\varepsilon$ and a sequence $(t_n,r_n)_{n\in\nn}$ in $\rr^2$ such that $t_n$ goes to $+\infty$ as $t$ goes to $+\infty$ and, for every nonnegative integer $n$, 
\[
t_n\ge\tEscTransv
\quad\text{and}\quad
r_n\in \bigl[0,\rHom(t_n)\bigr]
\quad\text{and}\quad
\abs{u(r_n,t_n)-m}\ge\varepsilon
\,.
\]
Up to extracting a subsequence, it may be assumed that one among the following two assertions holds:
\begin{enumerate}
\item $r_n$ goes to $+\infty$ as $t$ goes to $+\infty$;
\label{item:assertion_rn_goes_to_infty}
\item there exists a nonnegative quantity $r_\infty$ such that $r_n$ goes to $r_\infty$ as $n$ goes to $+\infty$. 
\label{item:assertion_rn_converges}
\end{enumerate}
If assertion \cref{item:assertion_rn_goes_to_infty} holds, then, according to \cref{lem:compactness} and to conclusion \cref{item:prop_no_invasion_implies_relaxation_asympt_energy} of \cref{prop:no_invasion_implies_relaxation}, up to extracting again a subsequence, it may be assumed that the functions $\rho\mapsto u(r_n+\rho,t_n)$ converge, uniformly on every compact subset of $\rr$, towards a $\ccc^2$-function $\phi:\rr\to\rr^{\dState}$ satisfying system \cref{syst_stationary_solutions_large_radius_limit}, and satisfying 
\[
\sup_{\rho\in\rr}\abs{\phi(\rho)-m}\le\dEsc(m)
\quad\text{and}\quad
\abs{\phi(0)-m}\ge\varepsilon\,,
\quad\text{thus}\quad
\phi\not\equiv m
\,,
\]
a contradiction with conclusion \cref{item:escape_spatial_asymptotics_ss} of \cref{lem:spatial_asymptotics_stat_sol_stable_at_infinity}.

If assertion \cref{item:assertion_rn_converges} holds, then, according to \cref{lem:compactness_close_to_origin} and to conclusion \cref{item:prop_no_invasion_implies_relaxation_asympt_energy} of \cref{prop:no_invasion_implies_relaxation}, up to extracting again a subsequence, it may be assumed that the functions $r\mapsto u(r,t_n)$ converge, uniformly on every compact subset of the interval $[0,+\infty)$, towards a $\ccc^2$-function $\hat{\phi}$ such that the function 
\[
\phi:[-r_\infty,+\infty)\to\rr^{\dState}\,, \quad \rho\mapsto \hat{\phi}(r_\infty+\rho)
\]
satisfies system \cref{syst_stationary_solutions_domain_bounded_to_the_left}, and satisfies
\[
\sup_{\rho\in[-r_\infty,+\infty)}\abs{\phi(\rho)-m}\le\dEsc(m)
\quad\text{and}\quad
\abs{\phi(0)-m}\ge\varepsilon\,,
\quad\text{thus}\quad
\phi\not\equiv m
\,,
\]
again a contradiction with conclusion \cref{item:escape_spatial_asymptotics_ss} of \cref{lem:spatial_asymptotics_stat_sol_stable_at_infinity}.
\end{proof}
\subsubsection{Non divergence towards infinity in the finiteness alternative for the Escape radius}
If the first alternative of \cref{cor:finiteness_infiniteness_of_Escape_radius} occurs, then it follows from \cref{lem:approach_to_m_alternative} that the conclusion of \cref{prop:relax_implies_cv} holds. The following proposition is the main step towards completing the proof of \cref{prop:relax_implies_cv} when the second alternative of \cref{cor:finiteness_infiniteness_of_Escape_radius} occurs. 
\begin{proposition}[non divergence towards infinity for the Escape radius]
\label{prop:non_divergence_towards_infinity_for_Escape_radius}
Assume that the second alternative of \cref{cor:finiteness_infiniteness_of_Escape_radius} occurs, that is $\rEsc(t)$ is positive for all $t$ in $[\tEscTransv,+\infty)$. Then the following inequality holds:
\[
\liminf_{t\to+\infty} \rEsc(t) < +\infty
\,.
\]
\end{proposition}
\begin{proof}
\renewcommand{\qedsymbol}{}
Let us proceed by contradiction and assume that 
\begin{equation}
\label{assumption_by_contradiction_rEsc_goes_to_infinity}
\rEsc(t) \to +\infty
\quad\text{as}\quad
t\to+\infty
\,.
\end{equation}
Let $t$ be a time greater than or equal to $\tEscTransv$. Proceeding as in the proof of Pokhozhaev's identity (see for instance \cite{BerestyckiLions_existenceGroundState_1983}), the equality obtained by integrating over space the scalar product of system \cref{syst_rad_sym} by $r^d u_r$ (that is, by $r u_r$ times the factor $r^{d-1}$ induced by the Lebesgue measure on $\rr^d$) will be considered. The domain of integration will be the interval $\bigl[0,2\rEsc(t)\bigr]$. To simplify the writing, let us denote by $R$ the quantity $2\rEsc(t)$. This leads to the following three integrals:
\[
\begin{aligned}
\iii_1(t) &= \int_0^{R}r^d u_r(r,t)\cdot u_t(r,t)\,,\\
\text{and}\quad
\iii_2(t) &= - \int_0^{R}r^d u_r(r,t)\cdot\nabla V\bigl(u(r,t)\bigr)\, dr \,,\\
\text{and}\quad
\iii_3(t) &= \int_0^{R}r^d u_r(r,t)\cdot\left(\frac{d-1}{r}u_r(r,t)+u_{rr}(r,t)\right)\, dr 
\,.
\end{aligned}
\]
According to system \cref{syst_rad_sym}, 
\begin{equation}
\label{approximate_Pokhozhaev_at_time_t}
\iii_1(t) = \iii_2(t) + \iii_3(t)
\,.
\end{equation}
Let us introduce the ``normalized'' potential $V^\ddag$ defined as 
\[
V^\ddag(v) = V(v) - V(m)
\,,
\]
and let us introduce the additional notation:
\begin{equation}
\label{def_kkk_vvv_eee_ddd}
\begin{aligned}
\kkk(t) &= \int_0^{R}r^{d-1}\frac{1}{2}u_r(r,t)^2\, dr \,,
\quad&&\text{and}&\quad
\vvv(t) &= \int_0^{R} r^{d-1}V^\ddag\bigl(u(r,t)\bigr)\, dr \,,\\
\text{and}\quad
\tilde\eee(t) &= \kkk(t) + \vvv(t) \,,
\quad&&\text{and}&\quad
\tilde\ddd(t) &= \int_0^{R} r^{d-1}u_t(r,t)^2\, dr \,,
\end{aligned}
\end{equation}
and 
\[
H^\ddag(t) = \frac{1}{2} u(R,t)^2 - V^\ddag\bigl(u(R,t)\bigr)
\,.
\]
Observe that $\kkk(t)$ and $\vvv(t)$ respectively denote the \emph{kinetic part} and the \emph{potential part} of the \emph{energy} $\tilde\eee(t)$ and $\tilde\ddd(t)$ denotes its \emph{dissipation}, while the expression of $H^\ddag(t)$ is the (normalized) Hamiltonian of the system governing stationary solutions of system \cref{syst_rad_sym} in the large radius limit (in the notation $\tilde\eee(t)$ and $\tilde\ddd(t)$, the ``tilde'' is here only to avoid any confusion with the quantities $\eee(t)$ and $\ddd(t)$ introduced in \cref{subsec:relax_sch_tr_fr}). All this notation naturally leads to rephrase equality \cref{approximate_Pokhozhaev_at_time_t}, as the following calculation shows:
\[
\begin{aligned}
\iii_2(t) &= - \int_0^{R}r^d \partial_r V^\ddag\bigl(u(r,t)\bigr)\, dr \\
&= - R^d V^\ddag\bigl(u(R,t)\bigr) + \int_0^{R}\partial_r(r^d) V^\ddag\bigl(u(r,t)\bigr)\, dr  \\
&= - R^d V^\ddag\bigl(u(R,t)\bigr) + d\times \vvv(t) 
\,,
\end{aligned}
\]
and 
\[
\begin{aligned}
\iii_3(t) &= \int_0^{R}\Bigl((d-1)r^{d-1}u_r(r,t)^2 + r^d u_r(r,t)\cdot u_{rr}(r,t)\Bigr)\, dr \\
&= 2(d-1) \kkk(t)  +  \int_0^{R} r^d \partial_r \left(\frac{u_r(r,t)^2}{2}\right)\, dr \\
&= 2(d-1) \kkk(t) + R^d \frac{u_r(R,t)^2}{2} - d \kkk(t) \\
&=  (d-2)\kkk(t) + R^d \frac{u_r(R,t)^2}{2}
\,.
\end{aligned}
\]
Thus it follows from equality \cref{approximate_Pokhozhaev_at_time_t} that
\begin{equation}
\label{Pokhozhaev_identity_with_time_dependence}
d\times \tilde\eee(t) - 2 \kkk(t) + R^d H^\ddag(t) = \int_0^{R} r^d u_r(r,t)\cdot u_t(r,t)\, dr
\,.
\end{equation}
Observe that the first two terms of the left-hand side of this equality correspond to Pokhozhaev's identity \cref{Pokhozhaev}. According to Cauchy--Schwarz inequality, 
\begin{equation}
\label{Cauchy_Schwarz_inequality}
\begin{aligned}
\abs{\int_0^{R}r^d u_r(r,t)\cdot u_t(r,t)\, dr}^2 &\le 2 \kkk(t) \int_0^{R}r^{d+1} u_t(r,t)^2 \, dr \\
&\le 2 \kkk(t) \, R^2\,  \tilde\ddd(t)
\,.
\end{aligned}
\end{equation}
It follows from equality \cref{Pokhozhaev_identity_with_time_dependence} and inequality \cref{Cauchy_Schwarz_inequality} that
\[
\begin{aligned}
\tilde\ddd(t) &\ge \frac{1}{2R^2 \kkk(t)}\abs{2 \kkk(t) - d\times \tilde\eee(t) - R^d H^\ddag(t)}^2 \\
&= \frac{\kkk(t)}{R^2}\abs{1 - \frac{d\times\tilde\eee(t)}{2\kkk(t)} - \frac{R^d H^\ddag(t)}{2\kkk(t)}}^2
\,,
\end{aligned}
\]
provided that $\kkk(t)$ is positive, which is true at least for $t$ positive large enough according to \cref{lem:kkk_larger_than_rEsc_power_d_minus_one} below. Substituting the notation $R$ with its value $2\rEsc(t)$, this last inequality reads: 
\begin{equation}
\label{lower_bound_dissipation_rEsc}
\tilde\ddd(t) \ge \frac{\kkk(t)}{4\rEsc(t)^2}\abs{1 - \frac{d\times\tilde\eee(t)}{2\kkk(t)} - \frac{\bigl(2\rEsc(t)\bigr)^d H^\ddag(t)}{2\kkk(t)}}^2
\,.
\end{equation}
A contradiction will follow from this inequality and the following four lemmas. 
\end{proof}
\begin{lemma}
\label{lem:integrability_of_dissipation}
The function $t\mapsto\tilde\ddd(t)$ is integrable on the interval $[\tEscTransv,+\infty)$. 
\end{lemma}
\begin{proof}[Proof of \cref{lem:integrability_of_dissipation}]
This statement follows from conclusion \cref{item:prop_no_invasion_implies_relaxation_integrability_of_dissipation} of \cref{prop:no_invasion_implies_relaxation}.
\end{proof}
\begin{lemma}
\label{lem:kkk_larger_than_rEsc_power_d_minus_one}
The following inequality hold: 
\begin{equation}
\label{kkk_larger_than_rEsc}
\liminf_{t\to+\infty}\frac{\kkk(t)}{\rEsc(t)}>0
\,.
\end{equation}
\end{lemma}
\begin{proof}[Proof of \cref{lem:kkk_larger_than_rEsc_power_d_minus_one}]
It is sufficient to prove the following inequality, which is stronger than inequality \cref{kkk_larger_than_rEsc} since $d$ is not smaller than $2$:
\begin{equation}
\label{kkk_larger_than_rEsc_power_d_mimus_one}
\liminf_{t\to+\infty}\frac{\kkk(t)}{\rEsc(t)^{d-1}}>0
\,.
\end{equation}
To prove this stronger inequality \cref{kkk_larger_than_rEsc_power_d_mimus_one}, let us proceed by contradiction and assume that there exists a sequence $(t_n)_{n\in\nn}$ of times greater than or equal to $\tEscTransv$ and going to $+\infty$ as $n$ goes to $+\infty$, such that 
\begin{equation}
\label{kkk_over_rEsc_power_d_minus_one_goes_to_0}
\frac{\kkk(t_n)}{\rEsc(t_n)^{d-1}}\to0
\quad\text{as}\quad
t\to+\infty
\,.
\end{equation}
Observe that, for every large enough positive integer $n$, according to the assumption \cref{assumption_by_contradiction_rEsc_goes_to_infinity} and to the definition \cref{def_kkk_vvv_eee_ddd} of $\kkk(\cdot)$, 
\[
\begin{aligned}
\frac{\kkk(t_n)}{\rEsc(t_n)^{d-1}} &\ge \frac{1}{\rEsc(t_n)^{d-1}}\int_{\rEsc(t_n)-1}^{\rEsc(t_n)+1}r^{d-1}\frac{1}{2} u_r(r,t_n)^2\, dr \\
&\ge \frac{\bigl(\rEsc(t_n)-1\bigr)^{d-1}}{2\rEsc(t_n)^{d-1}}\int_{\rEsc(t_n)-1}^{\rEsc(t_n)+1} u_r(r,t_n)^2\, dr
\,,
\end{aligned}
\]
thus it follows from assumption \cref{kkk_over_rEsc_power_d_minus_one_goes_to_0} that
\[
\int_{\rEsc(t_n)-1}^{\rEsc(t_n)+1} u_r(r,t_n)^2\, dr\to0 
\quad\text{as}\quad t\to+\infty 
\,,
\]
and as a consequence, it follows from the bounds \cref{bound_u_ut_ck} on the solution that
\[
u_r\bigl(\rEsc(t_n),t_n\bigr)\to 0  
\quad\text{as}\quad t\to+\infty 
\,,
\]
a contradiction with inequality \cref{partial_r_f_non_zero}. \Cref{lem:kkk_larger_than_rEsc_power_d_minus_one} is proved. 
\end{proof}
Note that according to assumption \cref{assumption_by_contradiction_rEsc_goes_to_infinity}, it follows from this inequality \cref{kkk_larger_than_rEsc} that $\kkk(t)$ goes to $+\infty$ as $t$ goes to $+\infty$. 
\begin{lemma}
\label{lem:eee_bounded_from_above}
The following inequality holds:
\begin{equation}
\label{eee_bounded_from_above}
\limsup_{t\to+\infty}\tilde\eee(t)<+\infty
\,.
\end{equation}
\end{lemma}
\begin{proof}[Proof of \cref{lem:eee_bounded_from_above}]
For every quantity $c$ in the interval $(0,\cHom)$, the quantity $2\rEsc(t)$ is less than $ct$ as soon as $t$ is large enough positive, and in this case it follows from the definition of $\rEsc(t)$ that the integrand
\[
\frac{1}{2}u_r(r,t)^2 + V\bigl(u(r,t)\bigr) - V(m)
\]
of the energy is nonnegative for $r$ in the interval $\bigl[2\rEsc(t),ct\bigr]$. Thus inequality \cref{eee_bounded_from_above} follows from conclusion \cref{item:prop_no_invasion_implies_relaxation_asympt_energy} of \cref{prop:no_invasion_implies_relaxation}.
\end{proof}
\begin{lemma}
\label{lem:Hamiltonian_small_at_two_rEsc}
The following limit holds:
\begin{equation}
\label{Hamiltonian_small_at_two_rEsc}
\rEsc(t)^d H^\ddag(t)\to0 
\quad\text{as}\quad
t\to+\infty
\,.
\end{equation}
\end{lemma}
\begin{proof}[Proof of \cref{lem:Hamiltonian_small_at_two_rEsc}]
Let us call upon the notation $\rSmallCurv$ and $\psi_0(\cdot)$ and $\fff_0(\cdot,\cdot)$ introduced in \cref{subsubsec:def_fire_zero}. Let $t_1$ denote a time greater than or equal to $\tEscTransv$ and large enough so that, for every time $t$ greater than or equal to $t_1$, the quantity $2\rEsc(t)$ is greater than $\rSmallCurv$. For every time $t$ greater than or equal to $t_1$, let
\[
\fff_1(t) = \fff_0\bigl(2\rEsc(t),t\bigr)
\,.
\]
Then, for $t$ greater than or equal to $t_1$,
\[
\fff_1'(t) = \partial_t \fff_0\bigl(2\rEsc(t),t\bigr) + 2 \rEsc'(t)\partial_{\widebar{r}}\fff_0\bigl(2\rEsc(t),t\bigr)
\,.
\]
For every quantity $\widebar{r}$ greater than or equal to $\rSmallCurv$, according to the definitions of $\psi_0(\cdot)$ and $T_{\widebar{r}}\psi_0(\cdot)$, 
\[
\abs{\partial_{\widebar{r}} \bigl(T_{\widebar{r}}\psi_0\bigr)(r)} = \kappa_0 T_{\widebar{r}}\psi_0(r)
\]
(except at the two points $\rSmallCurv$ and $\widebar{r}$ where this partial derivative is not defined). It follows that 
\[
\abs{\partial_{\widebar{r}}\fff_0(\widebar{r},t)} \le \kappa_0 \fff_0(\widebar{r},t)
\,.
\]
As a consequence, it follows from inequality \cref{dt_fire_0} of \cref{lem:dt_fire_0} that, for every time $t$ greater than or equal to $t_1$, with the notation of \cref{lem:dt_fire_0},
\begin{equation}
\label{upper_bound_fff_prime_one}
\fff_1'(t) \le \bigl(- \nuFZero + 2\kappa_0\rEsc'(t)\bigr) \fff_1(t) + \KFZero \int_{\SigmaEscZero(t)} T_{2\rEsc(t)}\psi_0(r)\, dr
\,.
\end{equation}
Up to replacing the time $t_1$ by larger positive quantity, it may be assumed that, for every time $t$ greater than or equal to $t_1$, 
\[
2\kappa_0\rEsc'(t)\le \frac{\nuFZero}{4}
\quad\text{and}\quad
3\rEsc(t)\le\rHom(t)
\,,
\]
so that, still for $t$ greater than or equal to $t_1$, 
\[
\SigmaEscZero(t)\subset[0,\rEsc(t)]\cup[3\rEsc(t),+\infty)
\,,
\]
thus
\[
\int_{\SigmaEscZero(t)} T_{2\rEsc(t)}\psi_0(r)\, dr \le \frac{2}{\kappa_0}\exp\bigl(-\kappa_0\rEsc(t)\bigr)
\,,
\]
so that it follows from inequality \cref{upper_bound_fff_prime_one} that
\begin{equation}
\label{upper_bound_fff_prime_one_bis}
\fff_1'(t) \le -\frac{3\nuFZero}{4}\fff_1(t) + \frac{2\KFZero}{\kappa_0} \exp\bigl(-\kappa_0\rEsc(t)\bigr)
\,.
\end{equation}
For every time $t$ greater than or equal to $t_1$, let 
\[
\gggg_1(t) = \rEsc(t)^d \fff_1(t)
\,,
\]
so that
\begin{equation}
\label{expression_gggg_one_prime}
\gggg_1'(t) = \rEsc(t)^d \left(\fff_1'(t) + d \frac{\rEsc'(t)}{\rEsc(t)}\fff_1(t)\right)
\,.
\end{equation}
Up to replacing the time $t_1$ by a larger positive quantity, it may be assumed that, for $t$ greater than or equal to $t_1$, 
\[
d \frac{\rEsc'(t)}{\rEsc(t)} \le \frac{\nuFZero}{4}
\,,
\] 
so that, introducing the function $\varphi:[t_1,+\infty)\to\rr$ defined as
\[
\varphi(t) = \frac{2\KFZero}{\kappa_0} \rEsc(t)^d\exp\bigl(-\kappa_0\rEsc(t)\bigr)
\,,
\]
it follows from \cref{upper_bound_fff_prime_one_bis,expression_gggg_one_prime} that
\[
\gggg_1'(t) \le - \frac{\nuFZero}{2} \gggg_1(t) + \varphi(t)
\,,
\]
and since $\varphi(t)$ goes to $0$ as $t$ goes to $+\infty$, the same is true for $\gggg_1(t)$. Thus 
\[
\rEsc(t)^d \fff_1(t)\to 0 
\quad\text{as}\quad 
t\to+\infty
\,.
\]
Proceeding as in the proof of \cref{lem:esc_Esc}, it follows that
\[
\sup_{r\in\bigl[2\rEsc(t)-1,2\rEsc(t)+1\bigr]} \rEsc(t)^d\bigl(u(r,t)-m\bigr)^2 \to 0 
\quad\text{as}\quad 
t\to+\infty
\,,
\]
so that, according to the bounds \cref{bound_u_ut_ck} on the solution,
\[
\rEsc(t)^d \biggl(\Bigl(u\bigl(2\rEsc(t),t\bigr)-m\Bigr)^2 + u_r\bigl(2\rEsc(t),t\bigr)^2\biggr)\to 0 
\quad\text{as}\quad 
t\to+\infty
\,,
\]
and inequality \cref{Hamiltonian_small_at_two_rEsc} follows. \Cref{lem:Hamiltonian_small_at_two_rEsc} is proved. 
\end{proof}
\begin{proof}[End of the proof of \cref{prop:non_divergence_towards_infinity_for_Escape_radius}]
It follows from inequalities \cref{lower_bound_dissipation_rEsc,kkk_larger_than_rEsc,eee_bounded_from_above} and from the limit \cref{Hamiltonian_small_at_two_rEsc} that
\[
\liminf_{t\to+\infty}\frac{\rEsc(t)^2}{\kkk(t)}\ddd(t)\ge\frac{1}{4}
\,,
\]
so that, according to inequality \cref{kkk_larger_than_rEsc}, 
\[
\liminf_{t\to+\infty}\rEsc(t)\ddd(t) >0
\quad\text{thus}\quad
\liminf_{t\to+\infty}t\ddd(t) >0
\,,
\]
a contradiction with the integrability of $t\mapsto\ddd(t)$ stated in \cref{lem:integrability_of_dissipation}. \Cref{prop:non_divergence_towards_infinity_for_Escape_radius} is proved. 
\end{proof}
\subsection{Convergence}
\label{subsec:convergence_behind_fronts}
\begin{proof}[Proof of conclusion \cref{item:prop_relax_implies_cv_approach_phi} of \cref{prop:relax_implies_cv}]
\renewcommand{\qedsymbol}{}
If the first alternative of \cref{cor:finiteness_infiniteness_of_Escape_radius} occurs then the conclusion of \cref{prop:relax_implies_cv} follows from \cref{lem:approach_to_m_alternative}. Thus it remains to deal with the second alternative, that is the case where $\rEsc(t)$ is finite for $t$ greater than or equal to $\tEscTransv$. 

In this case, according to \cref{prop:non_divergence_towards_infinity_for_Escape_radius}, there exists a sequence $(t_n)_{n\in\nn}$ of positive times going to $+\infty$ such that, for every nonnegative integer $n$, the quantity $\rEsc(t_n)$ is finite (and positive according to conclusion \cref{item:cor_finiteness_infiniteness_of_Escape_radius} of \cref{cor:finiteness_infiniteness_of_Escape_radius}) and smaller than a positive quantity which does not depend on $n$. Up to extracting a subsequence, it may be assumed that there exists a nonnegative quantity $\rEscInfty$ such that $\rEsc(t_n)$ goes to $\rEscInfty$ as $t$ goes to $+\infty$, and up to extracting again a subsequence, it may be assumed, according to \cref{lem:compactness_close_to_origin} and to conclusion \cref{item:prop_no_invasion_implies_relaxation_asympt_energy} of \cref{prop:no_invasion_implies_relaxation}, that the functions $r\mapsto u(r,t_n)$ converge, uniformly on every compact subset of the interval $[0,+\infty)$, towards a $\ccc^2$-function $\phi$ satisfying system \cref{syst_rad_sym_stationary} (including the boundary condition at the left end, that is $\phi'(0)$ vanishes). In addition, according to the definition of $\rEsc(\cdot)$, the following property holds for $\phi$:
\begin{equation}
\label{properties_phi_infty}
\abs{\phi(\rEscInfty)-m} = \dEsc(m) 
\quad\text{and}\quad 
r\ge\rEscInfty \implies \abs{\phi(r)-m}\le\dEsc(m)
\,,
\end{equation}
so that, according to \cref{lem:spatial_asymptotics_stat_sol_stable_at_infinity} applied to the function $[-\rEscInfty,+\infty)\to\rr^{\dState}$, $\rho\mapsto\phi(\rEscInfty+\rho)$, 
\begin{equation}
\label{phi_infty_minus_m_smaller_than_dEsc_to_the_right_of_rEscInfty}
r>\rEscInfty \implies \abs{\phi(r)-m}<\dEsc(m)
\,,
\end{equation}
and the function $\phi$ actually belongs to the set $\PhiZeroCentre(m)$ (defined in \cref{definition_PhiZeroCentre_of_m}) of stationary solutions approaching $m$ at infinity. For every quantity $t$ greater than or equal to $\tEscTransv$, let us introduce the quantity
\[
\delta(t) = \max\bigl(\abs{u(0,t)-\phi(0)},\abs{\rEsc(t)-\rEscInfty}\bigr)
\,.
\]
According to the regularity of the solution (see \cref{subsec:glob_exist}) and to \cref{cor:finiteness_infiniteness_of_Escape_radius}, this quantity $\delta(t)$ depends continuously on $t$; and according to what precedes, 
\begin{equation}
\label{delta_goes_to_zero_as_n_goes_to_infty}
\delta(t_n)\to0
\quad\text{as}\quad
n\to+\infty
\,.
\end{equation}
The following lemma is the main step towards completing the proof of \cref{prop:relax_implies_cv}.
\end{proof}
\begin{lemma}[convergence of $\rEsc(t)$ and of the solution at $r$ equal to $0$]
\label{lem:delta_of_to_goes_to_zero}
The quantity $\delta(t)$ goes to $0$ as $t$ goes to $+\infty$. 
\end{lemma}
\begin{proof}

Let us proceed by contradiction and assume that the converse holds. Then there exists a positive quantity $\varepsilon$ such that
\begin{equation}
\label{assumption_by_contradiction_u_at_r_equal_zero_and_rEsc}
\limsup_{t\to+\infty}\delta(t) > \varepsilon
\,.
\end{equation}
According to hypothesis \textup{(\hyperlink{hypDiscStationarym}{\hypDiscStationarymRef})} and up to replacing $\varepsilon$ by a smaller positive quantity, it may be assumed that, for every function $\varphi$ in the set $\PhiZeroCentre(m)\setminus\{\phi\}$, 
\begin{equation}
\label{phi_infty_isolated_in_PhiCentre_of_m}
\abs{\varphi(0)-\phi(0)} > \varepsilon
\,.
\end{equation}
Besides, according to the limit \cref{delta_goes_to_zero_as_n_goes_to_infty}, it may be assumed, up to dropping enough terms at the beginning of the sequence $(t_n)_{n\in\nn}$, that for every nonnegative integer $n$, 
\begin{equation}
\label{delta_of_tn_smaller_than_epsilon}
\delta(t_n)<\varepsilon
\,.
\end{equation}
For every $n$ in $\nn$, it follows from assumption \cref{assumption_by_contradiction_u_at_r_equal_zero_and_rEsc} that the set
\[
\Bigl\{t\in[t_n,+\infty): \delta(t)\ge\varepsilon\Bigr\}
\]
is nonempty. Let $\tilde{t}_n$ denote the infimum of this set. It follows from inequality \cref{delta_of_tn_smaller_than_epsilon} and from the continuity of the function $t\mapsto\delta(t)$ that 
\begin{equation}
\label{delta_of_tilde_tn_equals_epsilon}
\delta(\tilde{t}_n) = \varepsilon
\,.
\end{equation}
In addition $\tilde{t}_n$ goes to $+\infty$ as $n$ goes to $+\infty$. Thus, up to extracting again a subsequence, it may be assumed that there exists a nonnegative quantity $\tilderEscInfty$ such that $\rEsc(t_n')$ goes to $\tilderEscInfty$ as $n$ goes to $+\infty$; and, up to extracting again a subsequence, according to \cref{lem:compactness_close_to_origin} and to conclusion \cref{item:prop_no_invasion_implies_relaxation_asympt_energy} of \cref{prop:no_invasion_implies_relaxation}, it may be assumed that the functions $r\mapsto u(r,\tilde{t}_n)$ converge, uniformly on every compact subset of $[0,+\infty)$, towards a $\ccc^2$-function $\tilde{\phi}$ satisfying system \cref{syst_rad_sym_stationary} (including the boundary condition at the left end of $[0,+\infty)$). In addition, according to the definition of $\rEsc(\cdot)$, the same properties as \cref{properties_phi_infty,phi_infty_minus_m_smaller_than_dEsc_to_the_right_of_rEscInfty} must hold for $\tilde{\phi}$:
\begin{equation}
\label{tilde_phi_infty_minus_m_smaller_than_dEsc_to_the_right_of_tilde_rEscInfty}
\abs{\tilde{\phi}(\tilderEscInfty)-m} = \dEsc(m) 
\quad\text{and}\quad 
r>\tilderEscInfty \implies \abs{\tilde{\phi}(r)-m}<\dEsc(m)
\,,
\end{equation}
so that, according to \cref{lem:spatial_asymptotics_stat_sol_stable_at_infinity} applied to the function $[-\tilderEscInfty,+\infty)\to\rr^{\dState}$, $\rho\mapsto\tilde{\phi}(\tilderEscInfty+\rho)$, the function $\tilde{\phi}$ must again belong to the set $\PhiZeroCentre(m)$. Then, it follows from \cref{phi_infty_isolated_in_PhiCentre_of_m,delta_of_tilde_tn_equals_epsilon} that $\tilde{\phi}$ is actually the same function as $\phi$. Thus it follows from the definition of $\delta(\cdot)$ and from \cref{delta_of_tilde_tn_equals_epsilon} that, for every large enough positive integer $n$, 
\[
\abs{\rEsc(\tilde{t}_n)-\rEscInfty} = \varepsilon
\quad\text{so that}\quad
\abs{\tilderEscInfty-\rEscInfty} = \varepsilon
\,,
\]
and so that $\tilderEscInfty$ differs from $\rEscInfty$, a contradiction with properties \cref{properties_phi_infty,phi_infty_minus_m_smaller_than_dEsc_to_the_right_of_rEscInfty,tilde_phi_infty_minus_m_smaller_than_dEsc_to_the_right_of_tilde_rEscInfty}. \Cref{lem:delta_of_to_goes_to_zero} is proved. 
\end{proof}
\begin{proof}[End of the proof of conclusion \cref{item:prop_relax_implies_cv_approach_phi} of \cref{prop:relax_implies_cv}]
It follows from \cref{lem:delta_of_to_goes_to_zero} that, for every positive quantity $L$, 
\[
\sup_{r\in[0,L]}\abs{u(r,t) - \phi(r)}\to0
\quad\text{as}\quad
t\to+\infty
\,.
\]
Since in addition $\rEsc(t)$ converges towards the finite quantity $\rEscInfty$ at $t$ goes to $+\infty$, it follows (proceeding as in the proof of \cref{lem:approach_to_m_alternative}), that the stronger limit \cref{conclusion_relax_implies_cv} actually holds. Conclusion \cref{item:prop_relax_implies_cv_approach_phi} of \cref{prop:relax_implies_cv} is proved. 
\end{proof}
\begin{proof}[Proof of conclusion \cref{item:prop_relax_implies_cv_value_asympt_en} of \cref{prop:relax_implies_cv}]
\renewcommand{\qedsymbol}{}
To complete the proof of \cref{prop:relax_implies_cv}, it remains to prove that the residual asymptotic energy $\eeeResAsympt[u]$ of the solution equals $\eee[\phi]$. The arguments are similar to those of \cite[\GlobalRelaxationSubsecValueAsymptoticEnergy]{Risler_globalRelaxation_2016}, and the notation introduced below is similar to the one of this reference. 

Let us assume that the second alternative of \cref{cor:finiteness_infiniteness_of_Escape_radius} occurs, and let us call upon the notation $V^\dag$ introduced in \cref{def_Vdag_udag} and the notation $E^\dag(r,t)$ and $F^\dag(r,t)$ introduced in \cref{def_Edag_Fdag}. For every nonnegative quantity $r$, let us introduce the quantity $E_\phi(r)$ defined as
\[
E_\phi(r) = \frac{1}{2}\phi'(r)^2 + V^\ddag\bigl(\phi(r)\bigr) = \frac{1}{2}\phi'(r)^2 + V^\dag\bigl(\phi(r)-m\bigr)
\,.
\]
The same construction as in \cite[\GlobalRelaxationSubsecValueAsymptoticEnergy]{Risler_globalRelaxation_2016} provides, for some time $t_0$ large enough positive, a $\ccc^1$-function $\rExt:[t_0,+\infty)\to[0,+\infty)$ such that the following limits hold as $t$ goes to $+\infty$:
\begin{equation}
\label{limits_rExt_and_rExtprime_and_energy}
\rExt(t)\to+\infty
\ \text{and}\
\rExt'(t)\to 0
\ \text{and}\
\int_0^{\rEsc(t) + \rExt(t)}r^{d-1}\bigl(E^\dag(r,t) - E_\phi(r) \bigr)\, dr\to 0
\,.
\end{equation}
Let 
\begin{equation}
\label{def_b_of_t}
b(t) = \rEsc(t) + \rExt(t)
\,,
\end{equation}
see \cref{fig:value_asymptotic_energy}. 
\begin{figure}[!htbp]
	\centering
    \includegraphics[width=.8\textwidth]{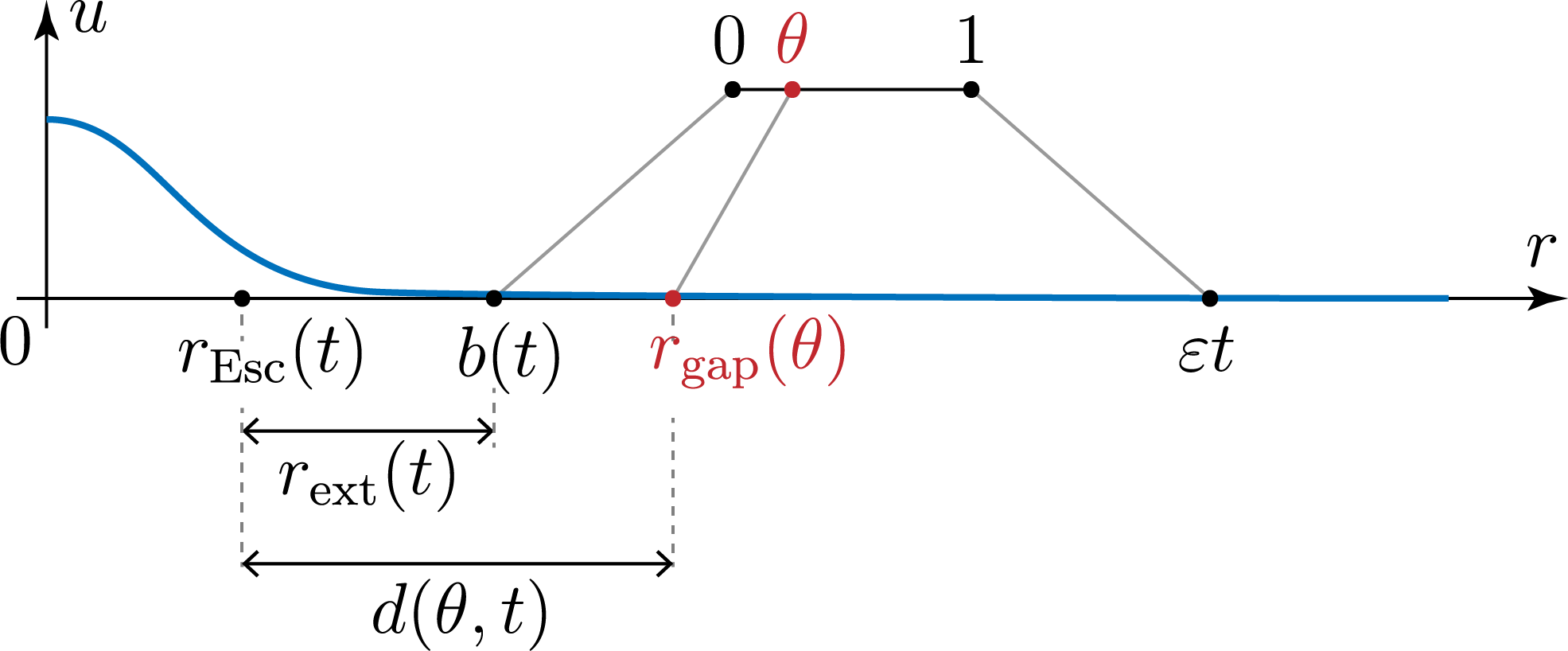}
    \caption{Illustration of the notation for the proof of conclusion \cref{item:prop_relax_implies_cv_value_asympt_en} of \cref{prop:relax_implies_cv}.}
    \label{fig:value_asymptotic_energy}
\end{figure}
Since $b(t)$ goes to $+\infty$ as $t$ goes to $+\infty$, it follows from the previous limit that 
\begin{equation}
\label{int_over_bulk_of_Edag_goes_to_E_of_phi}
\int_0^{b(t)} r^{d-1}E^\dag(r,t)\, dr \to E[\phi]
\quad\text{as}\quad
t\to+\infty
\,.
\end{equation}
Let $\varepsilon$ denote a small positive quantity to be chosen below (the value of $\varepsilon$ is provided in \cref{def_varepsilon} and depends only on $\phi$). According to conclusion \cref{item:prop_no_invasion_implies_relaxation_asympt_energy} of \cref{prop:no_invasion_implies_relaxation}, 
\begin{equation}
\label{int_up_to_eps_t_of_Edag_goes_to_eeeResAsympt}
\int_0^{\varepsilon t} r^{d-1}E^\dag(r,t)\, dr \to \eeeResAsympt[u] 
\quad\text{as}\quad
t\to+\infty
\,.
\end{equation}
As a consequence, conclusion \cref{item:prop_relax_implies_cv_value_asympt_en} of \cref{prop:relax_implies_cv} is a consequence of the following lemma. 
\end{proof}
\begin{lemma}[the energy over the interval ${[}b(t),\varepsilon t{]}$ goes to $0$]
\label{lem:energy_between_b_of_t_and_eps_t_goes_to_zero}
The following limit holds:
\begin{equation}
\label{energy_between_b_of_t_and_eps_t_goes_to_zero}
\int_{b(t)}^{\varepsilon t} r^{d-1}E^\dag(r,t)\, dr \to 0 
\quad\text{as}\quad
t\to+\infty
\,.
\end{equation}
\end{lemma}
\begin{proof}[Proof of \cref{lem:energy_between_b_of_t_and_eps_t_goes_to_zero}]
\renewcommand{\qedsymbol}{}
According to the limits \cref{rEscprime_of_t_goes_to_zero,limits_rExt_and_rExtprime_and_energy}, $b'(t)$ goes to $0$ as $t$ goes to $+\infty$. Thus there exists a time $t'_0$ greater than or equal to $t_0$ such that, for every time $t$ greater than or equal to $t'_0$, the quantity $b(t)$ is smaller than $\varepsilon t$. Let us call upon the notation $\kappa_0$ and $T_{\widebar{r}}\psi_0(r)$ and $\fff_0(\widebar{r},t)$ and $\SigmaEscZero(t)$ and $\nuFZero$ and $\KFZero$ introduced in \cref{def_kappa_zero,def_Tbarr_psi_zero,def_fffZero,def_nuZero,def_KFZero} (for the minimum point $m$ considered here), and let us introduce the functions $\rGap$ and $\gggg$, defined on $[0,1]\times[t'_0,+\infty)$ with values in $[0,+\infty)$, defined as
\begin{equation}
\label{def_rGap_and_gggg}
\rGap(\theta,t) = (1-\theta) b(t) + \theta \varepsilon t
\quad\text{and}\quad
\gggg(\theta,t) = \fff_0\bigl(\rGap(\theta,t),t\bigr)
\,,
\end{equation}
see \cref{fig:value_asymptotic_energy}. For every $(\theta,t)$ in $[0,1]\times[t_0,+\infty)$,
\[
\partial_t \gggg(\theta,t) = \partial_{\widebar{r}}\fff_0\bigl(\rGap(\theta,t),t\bigr)\partial_t \rGap(\theta,t) + \partial_t \fff_0\bigl(\rGap(\theta,t),t\bigr)
\,.
\]
According to inequality \cref{dt_fire_0} of \cref{lem:dt_fire_0}, 
\[
\begin{aligned}
\partial_t \fff_0\bigl(\rGap(\theta,t),t\bigr) &\le - \nuFZero \fff_0\bigl(\rGap(\theta,t),t\bigr) + \KFZero \int_{\SigmaEscZero(t)} T_{\rGap(\theta,t)}\psi_0(r)\, dr \\
&\le - \nuFZero \fff_0\bigl(\rGap(\theta,t),t\bigr) + \frac{2\KFZero}{\kappa_0}\exp\bigl(-\kappa_0 d(\theta,t)\bigr)
\,,
\end{aligned}
\]
where $d(\theta,t)$ denotes the distance between $\rGap(\theta,t)$ and the set $\SigmaEscZero(t)$ in $\rr$. Let us assume that $\varepsilon$ is smaller than $\cHom/2$ and, up to increasing $t'_0$, let us assume that, for every time $t$ greater than or equal to $t'_0$, the quantity $\rHom(t)$ is not smaller than $2\varepsilon t$. Then, for every time $t$ greater than or equal to $t'_0$,
\begin{equation}
\label{def_d_of_theta_t}
d(\theta,t) = \rGap(\theta,t)-\rEsc(t)
\,,
\end{equation}
see \cref{fig:value_asymptotic_energy}, and it follows from the previous inequality that
\[
\partial_t \fff_0\bigl(\rGap(\theta,t),t\bigr)\le - \nuFZero \fff_0\bigl(\rGap(\theta,t),t\bigr) + \frac{2\KFZero}{\kappa_0}\exp\bigl(-\kappa_0 d(\theta,t)\bigr)
\,.
\]
Besides, according to the definition \cref{def_Tbarr_psi_zero} of the weight function $T_{\widebar{r}}\psi_0(\cdot)$, 
\[
\abs{\partial_{\widebar{r}}\fff_0\bigl(\rGap(\theta,t),t\bigr)}\le \kappa_0 \fff_0\bigl(\rGap(\theta,t),t\bigr)
\]
It follows that, for every time $t$ greater than or equal to $t'_0$,
\[
\partial_t \gggg(\theta,t) \le - (\nuFZero - \varepsilon\kappa_0)\gggg(\theta,t) + \frac{2\KFZero}{\kappa_0}\exp\bigl(-\kappa_0 d(\theta,t)\bigr)
\,,
\]
so that if the quantity $\varepsilon$ is chosen as
\begin{equation}
\label{def_varepsilon}
\varepsilon = \min\left(\frac{\cHom}{2},\frac{\nuFZero}{16\kappa_0}\right)
\,,
\end{equation}
then the previous inequality yields
\begin{equation}
\label{upper_bound_partial_t_gggg}
\partial_t \gggg(\theta,t) \le - \frac{\nuFZero}{2}\gggg(\theta,t) + \frac{2\KFZero}{\kappa_0}\exp\bigl(-\kappa_0 d(\theta,t)\bigr)
\,.
\end{equation}
The factor $16$ in the denominator of the second quantity defining $\varepsilon$ will be useful for the next lemma, which, together with the forthcoming corollary, will complete the proof of \cref{lem:energy_between_b_of_t_and_eps_t_goes_to_zero}. 
\end{proof}
\begin{lemma}[upper bound on $\gggg(\theta,t)$ for $t$ large positive]
\label{lem:upper_bound_gggg}
There exists a time $t''_0$ greater than or equal to $t'_0$ such that, for every $\theta$ in $[0,1]$ and every time $t$ greater than or equal to $t''_0$, 
\begin{equation}
\label{upper_bound_gggg}
\gggg(\theta,t) \le \frac{8\KFZero}{\kappa_0\nuFZero} \exp\bigl(-\kappa_0 d(\theta,t)\bigr)
\,.
\end{equation}
\end{lemma}
\begin{proof}[Proof of \cref{lem:upper_bound_gggg}]
Let us introduce the function $\hhh(\cdot,\cdot)$ defined as
\begin{equation}
\label{def_hhh}
\hhh(\theta,t) = \gggg(\theta,t) - \frac{8\KFZero}{\kappa_0\nuFZero} \exp\bigl(-\kappa_0 d(\theta,t)\bigr)
\,.
\end{equation}
It follows from inequality \cref{upper_bound_partial_t_gggg} that, for every $\theta$ in $[0,1]$ and for every time $t$ greater than or equal to $t'_0$, 
\[
\begin{aligned}
\partial_t\hhh(\theta,t) &\le - \frac{\nuFZero}{2}\gggg(\theta,t)+ \frac{2\KFZero}{\kappa_0}\exp\bigl(-\kappa_0 d(\theta,t)\bigr)+ \frac{8\KFZero}{\nuFZero}\partial_td(\theta,t)\exp\bigl(-\kappa_0 d(\theta,t)\bigr) \\
&\le - \frac{\nuFZero}{2}\gggg(\theta,t)+\frac{2\KFZero}{\kappa_0}\exp\bigl(-\kappa_0 d(\theta,t)\bigr)\left(1 + \frac{4\kappa_0}{\nuFZero}\partial_td(\theta,t)\right) \\
&\le - \frac{\nuFZero}{2}\hhh(\theta,t) - \frac{2\KFZero}{\kappa_0}\exp\bigl(-\kappa_0 d(\theta,t)\bigr)\left(1 - \frac{4\kappa_0}{\nuFZero}\partial_td(\theta,t)\right)
\,.
\end{aligned}
\]
According to the definitions \cref{def_b_of_t,def_rGap_and_gggg,def_d_of_theta_t} of $b(t)$ and $\rGap(t)$ and $d(\theta,t)$,
\[
d(\theta,t) = (1-\theta) \rExt(t) + \theta\bigl(\varepsilon t - \rEsc(t)\bigr)\,,
\quad\text{so that}\quad
\partial_t d(\theta,t) = (1-\theta) \rExt'(t) + \theta\bigl(\varepsilon - \rEsc'(t)\bigr)
\,.
\]
Since $\rExt'(t)$ and $\rEsc'(t)$ go to $0$ as $t$ goes to $+\infty$, there exists a time $t'''_0$ greater than or equal to $t'_0$ such that, if $t$ is greater than or equal to $t'''_0$, then
\[
\partial_t d(\theta,t) \le 2\varepsilon\,,
\quad\text{thus}\quad
\frac{4\kappa_0}{\nuFZero}\partial_td(\theta,t) \le \frac{1}{2}
\,,
\]
and as a consequence, 
\[
\begin{aligned}
\partial_t \hhh(\theta,t) &\le - \frac{\nuFZero}{2}\hhh(\theta,t) - \frac{\KFZero}{\kappa_0}\exp\bigl(-\kappa_0 d(\theta,t)\bigr) \\
&\le - \frac{\nuFZero}{2}\hhh(\theta,t) - \frac{\KFZero}{\kappa_0}\exp\bigl(-\kappa_0 d(\theta,t'''_0)\bigr)\exp\bigl(-2\varepsilon\kappa_0(t-t'''_0)\bigr)
\,.
\end{aligned}
\]
Let us introduce the function $\jjj(\cdot,\cdot)$ defined as
\[
\jjj(\theta,t) = \hhh(\theta,t) \exp\bigl(2\varepsilon\kappa_0(t-t'''_0)\bigr)
\,.
\]
Then, if $t$ is greater than or equal to $t'''_0$, 
\[
\begin{aligned}
\partial_t\jjj(\theta,t) &\le \Biggl(\biggl(-\frac{\nuFZero}{2} + 2\varepsilon \kappa_0\biggr) \hhh(\theta,t) - \frac{\KFZero}{\kappa_0}\exp\bigl(-\kappa_0 d(\theta,t_0''')\bigr)\exp\bigl(-2 \varepsilon \kappa_0 (t-t_0''')\bigr)\Biggr) \\
& \qquad \times \exp\bigl(2\varepsilon \kappa_0 (t-t'''_0)\bigr) \\
&\le -\frac{\nuFZero}{4} \jjj(\theta,t) - \frac{\KFZero}{\kappa_0}\exp\bigl(-\kappa_0 d(\theta,t_0''')\bigr) \\
&\le -\frac{\nuFZero}{4} \jjj(\theta,t) - \frac{\KFZero}{\kappa_0}\exp\bigl(-\kappa_0 d(1,t_0''')\bigr)
\,.
\end{aligned}
\]
This last inequality shows that $\jjj(\theta,t)$ must eventually become negative (and remain negative afterwards) as time increases. More precisely, since according to the bounds \cref{bound_u_ut_ck} on the solution the quantity $\gggg(\theta,t'''_0)$ is bounded uniformly with respect to $\theta$, the same is true for the quantity $\jjj(\theta,t'''_0)$. As a consequence, there must exist a time $t''_0$ greater than or equal to $t'''_0$ such that, for every $\theta$ in $[0,1]$ and every time $t$ greater than or equal to $t''_0$, 
\[
\jjj(\theta,t) \le 0 \,,
\quad\text{so that}\quad
\hhh(\theta,t) \le 0
\,,
\]
and in view of the definition \cref{def_hhh} of $\hhh(\theta,t)$, inequality \cref{upper_bound_gggg} follows. \Cref{lem:upper_bound_gggg} is proved. 
\end{proof}
For every time $t$ greater than or equal to $t'_0$, let us write 
\[
\iii(t) = \int_{b(t)}^{\varepsilon t} \widebar{r}^{d-1} \fff_0(\widebar{r},t)\, d\widebar{r}
\,.
\]
\begin{corollary}[$\iii(t)$ goes to $0$]
\label{cor:limit_integral_firewall_over_gap}
The quantity $\iii(t)$ goes to $0$ as $t$ goes to $+\infty$. 
\end{corollary}
\begin{proof}[Proof of \cref{cor:limit_integral_firewall_over_gap}]
For every time $t$ greater than or equal to $t''_0$, it follows from inequality \cref{upper_bound_gggg} of \cref{lem:upper_bound_gggg} that, for every $\widebar{r}$ in $\bigl[b(t),\varepsilon t\bigr]$, 
\[
\fff_0(\widebar{r},t) \le \frac{8\KFZero}{\kappa_0\nuFZero} \exp\Bigl(-\kappa_0 \bigl(\widebar{r}-\rEsc(t)\bigr)\Bigr)
\]
so that
\[
\begin{aligned}
\iii(t) &\le \frac{8\KFZero}{\kappa_0\nuFZero}\exp\bigl(\kappa_0\rEsc(t)\bigr)\int_{b(t)}^{\varepsilon t} \widebar{r}^{d-1} \exp(-\kappa_0 \widebar{r})\,dr \\
&\le \frac{8\KFZero}{\kappa_0\nuFZero}\exp\bigl(\kappa_0\rEsc(t)\bigr)\frac{1}{\kappa_0^d}\int_{\kappa_0 b(t)}^{+\infty}r^{d-1} e^{-r}\, dr\\
&\le \frac{8\KFZero}{\kappa_0\nuFZero}\exp\bigl(\kappa_0\rEsc(t)\bigr)\frac{(d-1)!}{\kappa_0^d} \exp\bigl(-\kappa_0 b(t)\bigr) e_{d-1}\bigl(\kappa_0 b(t)\bigr)
\,,
\end{aligned}
\]
where $e_{d-1}(\cdot)$ denotes the exponential sum function defined as
\[
e_{d-1}(\tau) = \sum_{k=0}^{d-1}\frac{\tau^k}{k!}
\,.
\]
Since according to \cref{lem:delta_of_to_goes_to_zero} the quantity $\rEsc(t)$ converges as $t$ goes to $+\infty$, and since $b(t)$ goes to $+\infty$ as $t$ goes to $+\infty$, the intended limit follows. 
\end{proof}
\begin{proof}[End of the proof of \cref{lem:energy_between_b_of_t_and_eps_t_goes_to_zero}]
Let us assume that $t$ is positive large enough so that
\begin{equation}
\label{conditions_end_proof_energy_over_gap_goes_to_zero}
b(t)-1 \ge \rSmallCurv
\quad\text{and}\quad
b(t)+1 \le \varepsilon t - 1 
\,.
\end{equation}
Then, according to the nonnegativity of $F^\dag(r,t)$ (inequality \cref{coerc_Fdag_sf}), for every $\widebar{r}$ in $[b(t),\varepsilon t]$, 
\[
\fff_0(\widebar{r},t) \ge \int_{\widebar{r}-1}^{\widebar{r}+1}T_{\widebar{r}}\psi_0(r)F^\dag(r,t)\, dr
\,,
\]
so that, according to the definition \cref{def_psi_zero} of $\psi_0$ and the first of the conditions \cref{conditions_end_proof_energy_over_gap_goes_to_zero},
\[
\fff_0(\widebar{r},t) \ge e^{-\kappa_0}\int_{\widebar{r}-1}^{\widebar{r}+1}F^\dag(r,t)\, dr
\,,
\]
so that 
\[
\begin{aligned}
\iii(t)&\ge e^{-\kappa_0}\int_{b(t)}^{\varepsilon t} \widebar{r}^{d-1}\left(\int_{\widebar{r}-1}^{\widebar{r}+1}F^\dag(r,t)\, dr\right)\, d\widebar{r} \\
&\ge e^{-\kappa_0}\left(\frac{b(t)}{b(t)-1}\right)^{d-1} \int_{b(t)}^{\varepsilon t}\left(\int_{\widebar{r}-1}^{\widebar{r}+1}r^{d-1}F^\dag(r,t)\, dr\right)\, d\widebar{r} \\
&= e^{-\kappa_0}\left(\frac{b(t)}{b(t)-1}\right)^{d-1}\int_{b(t)}^{\varepsilon t}r^{d-1}F^\dag(r,t)\left(\int_{\max\bigl(r-1,b(t)\bigr)}^{\min\bigl(r+1,\varepsilon(t)\bigr)}\, d\widebar{r}\right)\, dr 
\,.
\end{aligned}
\]
The quantity
\[
\min\bigl(r+1,\varepsilon(t)\bigr)-\max\bigl(r-1,b(t)\bigr) = 2 + \min\bigl(r,\varepsilon(t)-1\bigr)-\max\bigl(r,b(t)+1\bigr) 
\]
is equal to 
\[
\left\{
\begin{aligned}
1 + r-b(t)\quad&\text{if}\quad b(t)+1\le r \le \varepsilon t-1\,,\\
2 \quad&\text{if}\quad b(t)+1\le r \le \varepsilon t-1\,, \\
1 + \varepsilon t - r \quad&\text{if}\quad\varepsilon t-1\le r \le \varepsilon t\,,
\end{aligned}
\right.
\]
and is therefore never less than $1$. It follows that
\begin{equation}
\label{iii_larger_than_Fdag}
\iii(t)\ge e^{-\kappa_0}\left(\frac{b(t)}{b(t)-1}\right)^{d-1}\int_{b(t)}^{\varepsilon t}r^{d-1}F^\dag(r,t)\, dr \,.
\end{equation}
On the other hand, since the interval $\bigl[b(t),\varepsilon t\bigr]$ does not intersect the set $\SigmaEscZero(t)$, the following inequalities hold:
\begin{equation}
\label{Fdag_larger_than_Edag}
\int_{b(t)}^{\varepsilon t}r^{d-1}F^\dag(r,t)\, dr\ge \int_{b(t)}^{\varepsilon t}r^{d-1}E^\dag(r,t)\, dr\ge 0
\,:
\end{equation}
The intended limit \cref{energy_between_b_of_t_and_eps_t_goes_to_zero} follows from \cref{cor:limit_integral_firewall_over_gap} and inequalities \cref{iii_larger_than_Fdag,Fdag_larger_than_Edag}. \Cref{lem:energy_between_b_of_t_and_eps_t_goes_to_zero} is proved. 
\end{proof}
\begin{proof}[End of the proof of conclusion \cref{item:prop_relax_implies_cv_value_asympt_en} of \cref{prop:relax_implies_cv}]
In view of the limits \cref{int_over_bulk_of_Edag_goes_to_E_of_phi} and \cref{int_up_to_eps_t_of_Edag_goes_to_eeeResAsympt},
conclusion \cref{item:prop_relax_implies_cv_value_asympt_en} of \cref{prop:relax_implies_cv} follows from \cref{lem:energy_between_b_of_t_and_eps_t_goes_to_zero} and is therefore proved. Since conclusion \cref{item:prop_relax_implies_cv_approach_phi} of \cref{prop:relax_implies_cv} was proved in \cref{subsec:convergence_behind_fronts}, the proof of \cref{prop:relax_implies_cv} is complete. 
\end{proof}
\section{Proofs of \texorpdfstring{\cref{thm:main,prop:resid_asympt_energy,prop:mountain_pass_existence_ground_state}}{Theorem \ref{thm:main} and Propositions \ref{prop:resid_asympt_energy} and \ref{prop:mountain_pass_existence_ground_state}}}
\begin{proof}[Proof of \cref{thm:main}]
Convergence towards the propagating terrace of bistable travelling fronts follows from \cref{prop:inv_cv}, and the convergence towards a stationary solution behind these fronts follows from \cref{prop:relax_implies_cv}. The proof is the same as that of \cite[\GlobalBehaviourThmMain]{Risler_globalBehaviour_2016} (see \GlobalBehaviourSecProofTheorem{} of this reference), thus details will not be reproduced here. 
\end{proof}
\begin{proof}[Proof of \cref{prop:resid_asympt_energy}]
This statement follows from conclusion \cref{item:prop_no_invasion_implies_relaxation_asympt_energy} of \cref{prop:no_invasion_implies_relaxation} and from \cref{prop:relax_implies_cv}. 
\end{proof}
\begin{proof}[Proof of \cref{prop:mountain_pass_existence_ground_state}]
The proof is very similar to the proof of \cite[\GlobalRelaxationCorExistenceStatSolLocMin{} and \GlobalRelaxationCorAttractorBorderBasinAtt]{Risler_globalRelaxation_2016}, see \GlobalRelaxationSubsecProofCorLocalMin{} of this reference for details. The proof relies mainly on the upper semi-continuity of the asymptotic energy which is proved (in the broader setting of system \cref{syst_higher_dim} without the radial symmetry hypothesis) in \cite{Risler_noInvasionCaseHigherSpace_2020} (see \InvasionRelaxationPropUpperSemicontAsymptEn{} of this reference). 
\end{proof}
\section{Spatial asymptotics for stationary solutions stable at the right end of space}
\begin{lemma}[spatial asymptotics for stationary solutions stable at the right end of space]
\label{lem:spatial_asymptotics_stat_sol_stable_at_infinity}
Let $m$ be a point of $\mmm$, let $\rhoLeftEnd$ be a quantity in $\{-\infty\}\cup(-\infty,0]$, let
\[
I = \left\{
\begin{aligned}
&[\rhoLeftEnd,+\infty)\quad\text{if}\quad \rhoLeftEnd\in(-\infty,0]\,, \\
&(-\infty,+\infty)\quad\text{if}\quad \rhoLeftEnd = -\infty
\,,
\end{aligned}
\right.
\]
and let $\phi:I\to\rr^{\dState}$, $\rho\mapsto\phi(\rho)$ denote a function which is a solution:
\begin{equation}
\label{syst_rad_sym_stationary_bis}
\begin{aligned}
\text{of system}\quad&\phi'' + \frac{d-1}{\rho-\rhoLeftEnd}\phi' = \nabla V(\phi)\\
\text{with the boundary condition}\quad&\phi'(\rhoLeftEnd) = 0
\quad\text{if}\quad \rhoLeftEnd\in(-\infty,0]\,, \\
\text{and of system}\quad&\phi'' = \nabla V(\phi)
\quad\text{if}\quad \rhoLeftEnd = -\infty
\,.
\end{aligned}
\end{equation}
Assume that  
\begin{equation}
\label{hyp_phi_minus_m_not_larger_than_dEsc_beyond_rZero}
\abs{\phi(\rho)-m}\le\dEsc(m) \quad\text{for every $\rho$ in $[0,+\infty)$}
\quad\text{and}\quad
\phi\not\equiv m
\,.
\end{equation}
Then the following conclusions hold. 
\begin{enumerate}
\item Both quantities $\abs{\phi(\rho)-m}$ and $\abs{\phi'(\rho)}$ go to $0$ as $\rho$ goes to $+\infty$. 
\label{item:cv_spatial_asymptotics_ss}
\item For every $\rho$ in $[0,+\infty)$, the scalar product $\bigl(\phi(\rho)-m\bigr)\cdot\phi'(\rho)$ is negative. 
\label{item:transv_spatial_asymptotics_ss}
\item For every $\rho$ in $[0,+\infty)$, the quantity $\abs{\phi(\rho)-m}$ is smaller than $\dEsc(m)$. 
\label{item:closer_spatial_asymptotics_ss}
\item The supremum $\sup_{\rho\in I}\abs{\phi(\rho)-m}$ is larger than $\dEsc(m)$. 
\label{item:escape_spatial_asymptotics_ss}
\end{enumerate}
\end{lemma}
\begin{proof}
If $\rhoLeftEnd$ equals $-\infty$, then all conclusions follow from \cite[\GlobalBehaviourLemTWApproachCriticalPoints]{Risler_globalBehaviour_2016}. Thus it may be assumed that $\rhoLeftEnd$ is in $(-\infty,0]$. Observe that the interval $[0,+\infty)$ is included in the interval $I$ where the function $\phi$ is defined. For every $\rho$ in $I$, let us introduce the quantities 
\begin{equation}
\label{def_Q_and_H}
Q(\rho)=\frac{1}{2} \bigl(\phi(\rho)-m\bigr)^2
\quad\text{and}\quad
H(\rho) = \frac{1}{2}\phi'(\rho)^2 - V\bigl(\phi(\rho)\bigr)
\,.
\end{equation}
Then, for all $\rho$ in $I$, it follows from system \cref{syst_rad_sym_stationary_bis} that, 
\[
\begin{aligned}
Q'(\rho) &= \bigl(\phi(\rho)-m\bigr)\cdot \phi'(\rho) \,,\\
\text{and}\quad
Q''(\rho) &= \phi'(\rho)^2 + \bigl(\phi(\rho)-m\bigr)\cdot \nabla V\bigl(\phi(\rho)\bigr) - \frac{d-1}{\rho-\rhoLeftEnd}\bigl(\phi(\rho)-m\bigr)\cdot \phi'(\rho)
\,,
\end{aligned}
\]
and thus, if $\rho$ is nonnegative, it follows from assumption \cref{hyp_phi_minus_m_not_larger_than_dEsc_beyond_rZero} and from inequality \cref{v_nablaV_controls_square_around_loc_min} that
\begin{equation}
\label{lower_bound_Q_second_with_scalar_product_term}
Q''(\rho)\ge \phi'(\rho)^2 + \frac{\eigVmin(m)}{2}\bigl(\phi(\rho)-m\bigr)^2 - \frac{d-1}{\rho-\rhoLeftEnd} \bigl(\phi(\rho)-m\bigr)\cdot \phi'(\rho)
\,.
\end{equation}
Let us introduce the quantity 
\[
\rho_0 = \rhoLeftEnd + (d-1)\sqrt{\frac{2}{\eigVmin(m)}}
\,.
\]
For every $\rho$ greater than or equal to $\rho_0$, 
\[
\begin{aligned}
\abs{\frac{d-1}{\rho-\rhoLeftEnd} \bigl(\phi(\rho)-m\bigr)\cdot \phi'(\rho)} &\le \frac{1}{2}\phi'(\rho)^2 + \frac{1}{2}\left(\frac{d-1}{\rho-\rhoLeftEnd}\right)^2 \bigl(\phi(\rho)-m\bigr)^2 \\
& \le \frac{1}{2}\phi'(\rho)^2 + \frac{\eigVmin(m)}{4}\bigl(\phi(\rho)-m\bigr)^2
\,,
\end{aligned}
\]
so that, for every $\rho$ greater than or equal to $\max(0,\rho_0)$, it follows from \cref{lower_bound_Q_second_with_scalar_product_term} that
\begin{equation}
\label{lower_bound_Q_second}
Q''(\rho)\ge \frac{1}{2}\phi'(\rho)^2 + \frac{\eigVmin(m)}{4}\bigl(\phi(\rho)-m\bigr)^2
\,.
\end{equation}
On the other hand, it again follows from system \cref{syst_rad_sym_stationary_bis} that, for all $\rho$ in $I$,
\begin{equation}
\label{expression_H_prime}
H'(\rho) = - \frac{d-1}{\rho-\rhoLeftEnd} \phi'(\rho)^2 
\,,
\end{equation}
thus the function $H(\cdot)$ is non-increasing and thus bounded from above on $I$, and it follows from assumption \cref{hyp_phi_minus_m_not_larger_than_dEsc_beyond_rZero} that $V\bigl(\phi(\cdot)\bigr)$ is bounded on $[0,+\infty)$. According to 
the expression \cref{def_Q_and_H} of $H(\cdot)$, it follows that $\phi'(\cdot)$ is bounded on $[0,+\infty)$, so that $Q'(\cdot)$ is bounded on $I$. Since according to inequality \cref{lower_bound_Q_second} the function $Q'(\cdot)$ is non-decreasing (and even strictly increasing since $\phi$ is not identically equal to $m$) on $\bigl[\max(0,\rho_0),+\infty\bigr)$, it follows that $Q'(\rho)$ must converge towards a finite limit as $\rho$ goes to $+\infty$; and thus it follows from inequality \cref{lower_bound_Q_second} that both functions $\phi'(\cdot)^2$ and $Q(\cdot)$ are square-integrable on $I$. Since $Q'(\cdot)$ is bounded, it follows that $Q(\rho)$ must converge towards $0$ as $\rho$ goes to $+\infty$. Thus $\phi(\rho)$ goes to $m$ as $\rho$ goes to $+\infty$, and as a consequence $V\bigl(\phi(\rho)\big)$ goes to $V(m)$ as $\rho$ goes to $+\infty$. Thus, since the function $H(\cdot)$ must converge to a finite limit at $+\infty$, it follows that $\phi'(\rho)^2$ must also go to a finite limit when $\rho$ goes to $+\infty$. Since $\phi'(\cdot)^2$ is integrable on $I$, its limit at $+\infty$ must be $0$. Assertion \cref{item:cv_spatial_asymptotics_ss} is proved. 

It follows from assertion \cref{item:cv_spatial_asymptotics_ss} that $Q'(\cdot)$ converges towards $0$ at $+\infty$, thus since this function is strictly increasing on $\bigl[\max(0,\rho_0),+\infty\bigr)$, it follows that $Q'(\rho)$ is negative for all $\rho$ in this interval. To prove assertion \cref{item:transv_spatial_asymptotics_ss}, let us proceed by contradiction and assume that $Q'(\cdot)$ takes a nonnegative value somewhere in $I$. Then $\rho_0$ must be larger than $\rhoLeftEnd$, and $Q'(\cdot)$ must takes a nonnegative value somewhere in $[\rhoLeftEnd,\rho_0)$. Let 
\[
\rho_1 = \max\{\rho\in[\rhoLeftEnd,\rho_0):Q'(\rho)\ge 0\}
\,.
\]
Then $Q'(\cdot)$ is negative on $(\rho_1,+\infty)$ and it follows from \cref{lower_bound_Q_second_with_scalar_product_term} that $Q''(\cdot)$ is positive on this interval, thus $Q'(\cdot)$ is strictly increasing on this interval, and it follows that $Q'(\rho_1)$ is negative, a contradiction with the definition of $\rho_1$. Assertion \cref{item:transv_spatial_asymptotics_ss} is proved, and assertion \cref{item:closer_spatial_asymptotics_ss} follows from assertion \cref{item:transv_spatial_asymptotics_ss}. Since according to the boundary condition \cref{syst_rad_sym_stationary_bis} the scalar product $\bigl(\phi(\rhoLeftEnd)-m\bigr)\cdot\phi'(\rhoLeftEnd)$ must be equal to $0$, it follows from assertion \cref{item:transv_spatial_asymptotics_ss} that $\rhoLeftEnd$ cannot be equal to $0$, or equivalently that assumption \cref{hyp_phi_minus_m_not_larger_than_dEsc_beyond_rZero} cannot hold up to $\rhoLeftEnd$, and this proves assertion \cref{item:escape_spatial_asymptotics_ss}. \Cref{lem:spatial_asymptotics_stat_sol_stable_at_infinity} is proved. 
\end{proof}
\subsubsection*{Acknowledgements} 
I am indebted to Thierry Gallay and Romain Joly for their help and interest through numerous fruitful discussions. 
\emergencystretch=1em
\printbibliography 
\bigskip
\mySignature
\end{document}